\documentclass[a4paper,10pt]{article}
\usepackage[utf8x]{inputenc}
\usepackage[T1]{fontenc}
\pdfoutput=1

\usepackage{ifthen}
\newif\ifdraftmode
\draftmodefalse

\usepackage{times}
\usepackage{geometry}
\usepackage[T1]{fontenc}
\usepackage{color}
\usepackage{amsmath}
\usepackage{amssymb}
\usepackage{MnSymbol}
\usepackage{array}
\usepackage{amsthm}
\usepackage{graphicx}
\usepackage{hyperref}
\usepackage{setspace}
\usepackage{stmaryrd}
\usepackage[normalem]{ulem} 
\usepackage[capitalize]{cleveref}
\usepackage{mathtools} 
\usepackage{slashed} 
\usepackage{enumerate}
\usepackage{tikz}
\usepackage{tikz-cd}
\ifdraftmode
\usepackage{refcheck}
\usepackage{booktabs}
\fi

\theoremstyle{plain}
\newtheorem{theorem}{Theorem}[section]
\newtheorem{proposition}[theorem]{Proposition}
\newtheorem{corollary}[theorem]{Corollary}
\newtheorem{lemma}[theorem]{Lemma}
\newtheorem*{theorem*}{Theorem}

\theoremstyle{definition}
\newtheorem{setup}[theorem]{Setup}
\newtheorem{definition}[theorem]{Definition}

\newtheorem{example}[theorem]{Example}

\theoremstyle{remark}
\newtheorem{remark}[theorem]{Remark}

\numberwithin{equation}{section}

\let\div\undefined

\let\phi\varphi
\let\epsilon\varepsilon
\let\theta\vartheta

\let\ol\overline

\newcommand{\Z}{\mathbb{Z}}
\newcommand{\R}{\mathbb{R}}

\newcommand{\C}{\mathbb{C}}

\newcommand{\CCl}{\C\mathrm{l}}

\newcommand{\Rho}{\mathrm{P}}

\newcommand{\grad}{\mathrm{grad}}
\newcommand{\supp}{\mathrm{supp}}

\newcommand{\scal}{\mathrm{scal}}
\newcommand{\id}{\mathrm{id}}
\newcommand{\Ric}{\mathrm{Ric}}
\newcommand{\ric}{\mathrm{ric}}
\newcommand{\Ein}{\mathrm{Ein}}
\newcommand{\trace}{\mathrm{tr}}
\DeclareMathOperator{\tr}{\trace}
\newcommand{\dvol}{\mathrm{dvol}}

\newcommand{\volume}{\mathrm{vol}}
\newcommand{\vol}{\volume}

\newcommand{\Diff}{\mathrm{Diff}}

\newcommand\upd{\mathrm{d}}
\newcommand\Lie{\mathcal{L}}

\newcommand{\del}{\partial}
\DeclareMathOperator{\div}{div}

\newcommand\II{\mathrm{I}\hskip-.3mm\mathrm{I}}

\newcommand\bullette{{\scalebox{0.5}{$\bullet$}}}
\newcommand\argu{\,\raisebox{.1\keyheightlength}{\bullette}\,}

\newcommand\ceq{\coloneqq} 

\newcommand{\ie}{i.\,e.~}
\newcommand{\eg}{e.\,g.~}
\newcommand{\cf}{cf.~}

\newcommand\LorMan{\overline{M}}

\newcommand\LorMet{\bar g}
\newcommand\lt{\tau}
\newcommand{\rc}{\rho}
\newcommand{\normal}{N}
\newcommand{\ModSp}{\mathcal{M}_{\ric = 0,\rc,+}}
\newcommand{\ModSpPar}{\mathcal{M}_{\parallel,\rc,+}}
\newcommand{\Mod}{\mathcal{M}_{\ric = 0,+}}
\newcommand{\ModPar}{\mathcal{M}_{\parallel,+}}
\newcommand{\Sol}{\mathrm{Sol}}
\newcommand{\quotsim}{/\!\!\sim}
\newcommand{\gsph}{g_{\mathrm{sph}}}

\definecolor{lightgreen}{rgb}{.9,1,.9}

\definecolor{darkgreen}{rgb}{.2,.6,.2}
\definecolor{darkblue}{rgb}{.2,.2,.8}
\definecolor{darksharpblue}{rgb}{.4,.2,.8}

\newcommand\red[1]{\textcolor{red}{#1}}

\newcommand\magenta[1]{\textcolor{magenta}{#1}}

\newcommand\hidden[1]{}

\ifdraftmode

\newcommand\footbernd[1]{\footnote{\magenta{B2KJ:\ #1}}}
\newcommand\bernd[1]{\magenta{B2KJ:\ #1}}
\newcommand\klaus[1]{\red{K2BJ:\ #1}}

\renewcommand\sout{\bgroup\markoverwith{\textcolor{red}{\rule[0.7ex]{3pt}{1.4pt}}}\ULon}
\newcounter{mnotecount}[section]
\newcommand{\mnote}[1]{\protect{\stepcounter{mnotecount}}%
		\ensuremath{{\raisebox{0.5\baselineskip}[0pt]{\makebox[0pt][c]{\tiny\em{\red{$\bullet$\themnotecount}}}}}} %
		\marginpar{\raggedright\tiny\em $\!\!\!\!\!\!\,\bullet$\themnotecount: #1}\ignorespaces}
\else

\newcommand\footbernd[1]{}
\newcommand\bernd[1]{}
\newcommand\klaus[1]{}
\renewcommand\sout[1]{}
\newcommand\mnote[1]{}
\fi

\input{stelle.txi}

\title{On pp-waves with lightlike parallel spinors}
\author{Bernd Ammann
\and Jonathan Glöckle  
\and Klaus Kr\"oncke}

\begin{document}
\maketitle

\begin{abstract}
We parametrize pp-wave spacetimes with compact codimension $2$ hypersurfaces.
In the vacuum case, we show that these spacetimes are locally in one-to-one correspondence with smooth curves of Riemannian Ricci-flat metrics modulo smooth curves of diffeomorphisms.
We also prove that this one-to-one correspondence extends to pp-waves with prescribed null Ricci curvature.
Moreover, the pp-wave spacetime carries a lightlike parallel spinor if and only if one (and hence all) of the Ricci-flat metrics carries a parallel spinor.		
\end{abstract}
\section{Introduction}

In general relativity, pp-wave spacetimes, or pp-waves for short, give rise to an important family of exact solutions of Einstein's field equation. 
There are various definitions of (generalized) pp-waves in the literature. We will use the following conventions throughout this paper:
A \emph{pp-wave} is defined as a Lorentzian manifold with a lightlike parallel vector field. If a pp-wave is Ricci-flat, we call it \emph{vacuum pp-wave}.

Despite being a classical topic, during the last decade they played an important role in the study of Lorentzian special holonomy, see for example \cite{galaev.15,baum.leistner.lischewski:16,leistner.lischewski:17,baum.leistner:18} (in chronolical order).  
Even more recently, they gained additional interest in the context of initial data rigidity results (\cf \cite{Eichmair.Galloway.Mendes:2021,gloeckle:2023,Chai.Wan:2024p,gloeckle:2025p}) and the borderline case $E = |P|$ in the positive mass theorem (\cf \cite{Hirsch.Zhang:2024p}).

Our interest in these notions arises from the attempt to classify Lorentzian manifolds $(\overline{M},\overline{g})$  with a lightlike parallel spinor, in particular in view of the studies \cite{Ammann.Gloeckle:2023} of the first and second author on Dirac-Witten harmonic spinors on initial data sets subject to the dominant energy condition. Such manifolds always admit a lightlike parallel vector field and therefore are pp-waves. In addition, it is well known that the Ricci tensor of such manifolds is always null, i.e.\ $\Ric(X)$ is null for every vector field $X$ on the manifold~$\overline{M}$. For clarification, we use the following notational convention throughout the paper: A vector field $Y$ is called null, if $\overline{g}(Y,Y)\equiv0$ and it is additionally called lightlike, if it is nowhere vanishing.

Every pp-wave is locally isometric to
\begin{align}\label{eq:AGK_metrics}
(J\times I\times Q, \upd v\otimes \upd s+ \upd s\otimes \upd v+u^{-2}\upd s^2+g_s),
\end{align}
where $I,J$ are intervals, $Q$ is a manifold, $u:I\times Q\to\R_{>0}$ is a smooth function and $g_s$, $s\in I$, is a smooth family of Riemannian metrics on $Q$, \cf \cref{Rem:PPLocally} below. In \eqref{eq:AGK_metrics}, the (canonical) lightlike parallel vector field is given by the coordinate vector field $\frac{\partial}{\partial v}$. 
It was shown in \cite{leistner.lischewski:17} that if the spacetime carries a lightlike parallel spinor (with associated vector field $\frac{\del}{\del v}$), that all Riemannian metrics $g_s$ are Ricci-flat and carry a parallel spinor. 
 Consequently, an isometry class of spacetimes with lightlike parallel spinor gives rise to a curve in the moduli space of Ricci-flat metrics with parallel spinors.
By using suitable coordinates, one can furthermore locally achieve $u\equiv 1$, see \cref{Prop:GeodNormCoord} below.

Assuming $Q$ to be closed, we prove in this article a converse to this statement. In fact, we show that every Lorentzian metric of the form \eqref{eq:AGK_metrics} carries a lightlike parallel spinor, provided that its Ricci tensor is null and the metrics $g_s$ all carry a parallel spinor.

A partial result was already shown by the first and third author and M\"{u}ller \cite{ammann.kroencke.mueller:21}: Given a closed manifold~$Q$ and a smooth curve $[\hat g_s]$ in the moduli space of Ricci-flat metrics with parallel spinor on $Q$, we choose a representative $g_s$ of $[\hat g_s]$ in the space of metrics, such that $\div^{g_s}(\frac{\upd}{\upd s}g_s)=0$ for all $s\in I$. Then, the pp-wave 
\begin{align}\label{eq:AKM_metrics}
\upd v\otimes \upd s+\upd s\otimes \upd v+\upd s^2+g_s,
\end{align}
always carries a lightlike parallel spinor.

Now consider conversely the existence of a lightlike parallel spinor. As said before, this implies the existence of lightlike parallel vector field and thus it is locally of the form \eqref{eq:AGK_metrics}. Recall that this parallel spinor also implies that the metric~\eqref{eq:AGK_metrics} has null Ricci curvature, which is equivalent to $ \ric^{g_s}=0$ and
     \begin{align}\label{eq:momentum_constraint}
 \div^{g_s}\dot{g}_s-\upd\,\trace^{g_s} \dot{g}_s&=0
     \end{align}
     for all $s\in I$, where $\upd$ is the differential on $Q$ and we write $\dot g_s$ for the derivative of $g_s$ with respect to~$s$.
In this article, we generalize the construction of \cite{ammann.kroencke.mueller:21} and we see that all metrics of the above form do carry a lightlike parallel spinor.
\begin{theorem}\label{thm:par_spinor}
Let $Q$ be a closed spin manifold. Consider a pp-wave of the form \eqref{eq:AGK_metrics} on $\R\times I\times Q$ such that the metrics $g_s$ are Ricci-flat and satisfy \eqref{eq:momentum_constraint}.
If there is an $s_0\in I$ such that the Riemannian metric~$g_{s_0}$ carries a parallel unit spinor $\phi$, then
$\overline{g}\ceq \upd v\otimes \upd s+ \upd s\otimes \upd v+u^{-2}\upd s^2+g_s$
carries a unique parallel spinor $\Psi$ whose Dirac current (\cf \cref{Def:DiracCurr}) is proportional to the canonical lightlike parallel vector field $\frac{\del}{\del v}$ and such that
\begin{align*}
	\Psi|_{\left\{0\right\}\times \left\{s_0\right\}\times Q}
	\overset{\cong}{\longmapsto} 
	\phi.
\end{align*}
under the bundle isomorphism~\eqref{eq:IdentifySpinors}.
\end{theorem}

Note that if $g_s$, $s\in I$ are Ricci-flat metrics on a closed  manifold~$Q$, then the condition $\div^{g_s}\dot{g}_s=0$ implies that $\trace^{g_s} \dot{g}_s$ is locally constant on $Q$, and thus we obtain \eqref{eq:momentum_constraint}.
Therefore, the metrics of the form \eqref{eq:AKM_metrics} with all $g_s$ Ricci-flat and  $\div^{g_s}\dot{g}_s=0$ form a subclass of the metrics \eqref{eq:AGK_metrics} with \eqref{eq:momentum_constraint}.
In fact, the metrics obtained in \cite{ammann.kroencke.mueller:21} form a proper subclass -- even though there is more flexibility than stated in~\eqref{eq:AKM_metrics} above, \cf \cref{Rem:AKM-subset}.

The spinor $\Psi$ is obtained by $\nabla^{\overline{g}}$-parallel transport starting from $\{0\} \times \{s_0\} \times Q$, where it is determined by $\phi$ via~\eqref{eq:IdentifySpinors}.
This can be viewed as a generalization of the parallel transport of $\phi$ along the curve~$g_s$, $s \in I$, in \cite{ammann.kroencke.mueller:21}, which takes the function~$u$ and the condition \eqref{eq:momentum_constraint} into account, but is conceptionally simpler.
The obtained spinor on $\{0\} \times I\times Q$ can be considered as an imaginary W-Killing spinor in the sense of~\cite{baum.leistner.lischewski:16}.
It will then be extended to a lightlike parallel spinor on $\R \times I\times Q$.

We will see in \cref{prop_splitting_constraint_solutions} below that \eqref{eq:momentum_constraint} is satisfied if and only if $\dot{g}_s$ is the sum of a divergence free tensor and a Lie-derivative of $g_s$ in the direction of a gradient vector field. 
Therefore, in order to describe a general metric
\eqref{eq:AGK_metrics} with \eqref{eq:momentum_constraint} one needs additional parameters, compared to the special case  \eqref{eq:AKM_metrics} in \cite{ammann.kroencke.mueller:21}.
In particular, a curve $[g_s]$ in the moduli space of Ricci-flat metrics has many representatives~$\tilde{g}_s$ satisfying \eqref{eq:momentum_constraint}. These lead to many nonisometric Lorentzian metrics of the form \eqref{eq:AGK_metrics} with null Ricci curvature.
They can be distinguished through the one remaining, not necessarily vanishing, component of their Ricci tensor.
In fact, using the additional freedom of a function~$u$ in~\eqref{eq:AGK_metrics} and allowing to rescale the $g_s$, it is possible to freely prescribe this curvature component, as explained in the following theorem.
We note again that it follows from \cref{Lem:ChangeHypersurface,Prop:GeodNormCoord} that locally around each $s \in I$ the function $u$ can be set to $1$ at the expanse of changing the representative $\tilde{g}_s$ of the class $[g_s]$.

\begin{theorem} \label{Thm:Ex}
	Let $Q$ be a closed manifold and $g_s$, $s\in I$, a smooth family of unit-volume Ricci-flat metrics on $Q$.
	Then there is a smooth family $\varphi_s$, $s \in I$, in $\Diff(Q)$, such that $\tilde{g}_s \coloneqq \varphi_s^*g_s$ satisfies $\div^{\tilde{g}_s}(\dot{\tilde{g}}_s) = 0$ and $tr^{\tilde{g}_s}(\dot{\tilde{g}}_s) = 0$.
	In particular, \eqref{eq:momentum_constraint} holds for $\tilde{g}_s$, \ie for all smooth functions $\lambda \colon I\to \R$ and $u \colon I\times Q\to \R_{>0}$ the pp-wave metric
	\begin{align} \label{eq:ppMetric}
		\overline{g} \coloneqq \upd v\otimes \upd s+\upd s\otimes \upd v+ u^{-2} \upd s^2 + \lambda^2 \tilde{g}_s,
	\end{align}
	defined on $\R\times (I\setminus \lambda^{-1}(0))\times Q$ has null Ricci curvature.
	
	Moreover, if $\tilde{g}_s$, $s \in I$ is a family of unit-volume Ricci-flat metrics on $Q$ with $\div^{\tilde{g}_s}(\dot{\tilde{g}}_s) = 0$ and $tr^{\tilde{g}_s}(\dot{\tilde{g}}_s) = 0$, then for any function $\rc\in C^{\infty}(I\times Q)$, there exist $\lambda \colon I\to \R$ and $u \colon I\times Q\to \R_{>0}$ such that $\lambda$ has a discrete set of zeros and the (null) Ricci curvature of~\eqref{eq:ppMetric} is given by $\ric^{\overline{g}} = \rc\, \upd s^2$. 
\end{theorem}

\begin{remark}\label{rem:scalefactor}
The function $\lambda: I\to\R$ in \cref{Thm:Ex} satisfies the second order differential equation
	\begin{align}\label{eq:ODE_lambda}
		\ddot{\lambda} = -\frac{1}{n-1} \left(\Rho + \frac{1}{4} \Sigma \right) \lambda
	\end{align}
        where $\Rho$ and $\Sigma$ are functions that will be defined in \cref{prop:ScaleODE} below, and that only depend on $g_s$, $s \in I$, and $\rc$.
 In general, we cannot avoid that $\lambda$ has zeros.
 Due to uniqueness under prescribed initial values, however, the zeros of $\lambda$ are always simple and therefore only occur at discrete times.
\end{remark}
\begin{remark}
For functions $c:I\to \R$ and $f:I\times Q\to\R$, we will often use the notation $c_s:=c(s)$ and $f_s=f(s,.)$, respectively, if we wish to interpret them as $s$-dependent objects. The derivative in the $s$-direction will be denoted by a dot.
\end{remark}
\begin{definition}
	Let $Q$ be a manifold.
	A \emph{parameterized simple pp-wave modeled on $Q$} is a Lorentzian manifold
\begin{align*}
	(\R\times I\times Q, \upd v\otimes \upd s+\upd s\otimes \upd v+ u^{-2} \upd s^2 +  g_s),
\end{align*}
where
\begin{itemize}
\item $I\subset\R$ is an open interval,
\item $g_s$, $s \in I$, is a smooth family of Riemannian metrics on $Q$ (without any restriction on the volume),
\item and $u \colon I \times Q \to \R_{>0}$ is a smooth function.
\end{itemize}
A Lorentzian manifold is called a \emph{simple pp-wave modeled on $Q$}, if it is isometric to a parameterized simple pp-wave modeled on $Q$.
      \end{definition}
In view of \eqref{eq:AGK_metrics} every pp-wave is locally isometric to a simple one, thus the crucial fact in this definition is that this holds globally.
In the following, we will generally assume that $Q$ is closed and connected.
In particular, this allows us to write the metrics on $Q$ as $\lambda_s^2 g_s$ with $\vol(Q,g_s)=1$ and $\lambda_s>0$.
In Section \ref{sec.hypersurfaces.exp-maps}, we will give a geometric characterization of simple pp-waves modeled on a such a manifold $Q$.
We are now looking for an efficient way to parametrize the ones with null Ricci curvature.	

Given a smooth family $g_s$, $s \in I$, of Riemannian metrics on $Q$ and a function $\rc \in C^\infty(I \times Q)$, we denote by $\Sol((g_s,\rc_s), s \in I) \subset C^\infty(I)$ the (two-dimensional) space of solutions $\lambda\colon I \to \R$ to the ODE~\eqref{eq:ODE_lambda}  associated to the parametrized curve $s \mapsto (g_s,\rc_s)$.
Let $\Sol_{+}((g_s,\rc_s), s \in I) \subset \Sol((g_s,\rc_s), s \in I)$ be the convex cone of positive solutions.
Observe that $\Sol_{+}((g_s,\rc_s), s \in I)$ can be empty but becomes nonempty when we restrict the curve to a sufficiently small interval.

We now consider parametrized smooth curves $I \ni s \mapsto (g_s,\rc_s,\lambda_s)$, where $g_s$, $s \in I$, is a smooth family of Ricci-flat metrics of unit volume,
$\rc \in C^\infty(I \times Q)=C^{\infty}(I,C^{\infty}(Q))$ and $\lambda \in \Sol_{+}((g_s,\rc_s), s \in I)$.
We use them as input data for our construction underlying \cref{Thm:Ex}.
Some of these curves yield isometric Lorentzian manifolds and to compensate for this we divide out an action by smooth families of diffeomorphisms.
Let
\begin{align*}
\mathcal{R}_{\vol=1}(Q)=\left\{g\mid g\text{ smooth metric on }Q\text{ s.t. }\vol^g(Q)=1\right\},
\end{align*}
and consider
\begin{align*}  
	\mathcal{C}=C^{\infty}(I,\mathcal{R}_{\vol=1}(Q)\times C^{\infty}(Q)\times \R)/C^\infty(I,\Diff(Q)),
\end{align*}
where a smooth family of diffeomorphisms $\phi_s$, $s \in I$, acts by pulling back $(g_s,\rc_s)$ and trivially on $\lambda_s$.

In \cref{prop:ScaleODE} below, we observe that the ODE~\eqref{eq:ODE_lambda} for $\lambda$ is invariant under pullback with the family $\varphi_s$ and thus 
\begin{align*}\Sol_+((\phi_s^*g_s,\phi_s^*\rc_s), s \in I) = \Sol_+((g_s,\rc_s), s \in I).
\end{align*}

Now we define 
\begin{align*}
\ModSp(I,Q)=\left\{ \bigl[(g_s,\rc_s,\lambda_s)\bigr]\in \mathcal{C} \,\Bigm|\, \ric^{g_s}=0\text{ for all }s\in I, \lambda\in \Sol_{+}((g_s,\rc_s), s \in I)\right\}
\end{align*}
We call elements of $\ModSp(I,Q)$ \emph{parametrized moduli curves}.
We also write $\ModSp(\bullet,Q)\ceq \coprod_I \ModSp(I,Q)$ where the disjoint union $\coprod_I$ runs over all open intervals.
We will discuss that some quotient of this space of parametrized moduli curves ``coincides'' with the space of isometry classes of simple pp-waves modeled on $Q$ with null Ricci curvature, see \cref{thm:1-1}.
Here, the correspondence is such that $[(g_s, \rc_s, \lambda_s), s \in I]$ gets associated to the simple pp-wave spacetime $\R \times I \times Q$ with metric \eqref{eq:ppMetric}, where $\tilde{g}_s = \phi_s^* g_s$ and where $\phi_s$, $s \in I$, and $u \colon I \times Q \to \R_{>0}$ are chosen such that the spacetime has null Ricci curvature $\ric^{\overline{g}} = \phi_s^* \rc_s \upd s \otimes \upd s$.
While the existence of at least one such choice is guaranteed by \cref{Thm:Ex}, we will still have to show that different such choices lead to isometric Lorentzian manifolds and in a second step that the well-defined map obtained this way is indeed a bijection.
  
\begin{remark}
It seems natural to just consider curves in the moduli space $\mathcal{R}_{\vol=1}(Q)/ \Diff(Q)$.
We however avoid this moduli space for subtle reasons, see \cref{Rem:ModuliSpaces} and \cref{ex_ModuliSpaceII}.
\end{remark}

The invariance of parametrized moduli curves under pullbacks by smooth families of diffeomorphisms is meant to compensate for certain isometries of simple pp-waves changing the parametrization, \cf \cref{Lem:ChangeHypersurface}.
A further action on the space of parametrized simple pp-waves that leaves the isometry class of the Lorentzian manifolds unchanged is given by affine reparametrization in the $s$-parameter. 
More precisely, for $(\alpha,\beta)\in (\R\setminus \left\{0\right\})\times \R$ the diffeomorphisms
\begin{equation} \label{eq:ScalingDiffeo}
	\begin{aligned}
		\Phi \colon \R\times I\times Q &\longrightarrow \R\times J \times Q \\
		(v,s,x) &\longmapsto (\alpha^{-1}v,\alpha s+\beta, x)   
	\end{aligned}
\end{equation}
preserve the form \eqref{eq:AGK_metrics}, although with the families $g_s$ and $u_s$ reparametrized.
In order to factor out the action of these diffeomorphisms, we call two parametrized moduli curves
 $[(g_s,\rc_s,\lambda_s)] \in \ModSp(I,Q)$ and $[(\hat{g}_t,\hat{\rc}_t,\hat{\lambda}_t)] \in \ModSp(J,Q)$ \emph{equivalent}, denoted as $[(g_s,\rc_s,\lambda_s), s \in I]\sim [(\hat{g}_t,\hat{\rc}_t,\hat{\lambda}_t), t \in J]$, if
there are $(\alpha,\beta)\in (\R\setminus \left\{0\right\})\times \R$ such that $I\to J$, $s\mapsto \alpha s+\beta$ is a bijection and 
\begin{align}\label{eq:equivalent_curves}
[(g_s,\rc_s,\lambda_s), s \in I]=[(\hat{g}_{\alpha s+\beta},\alpha^2\hat{\rc}_{\alpha s+\beta},\hat{\lambda}_{\alpha s+\beta}), s \in I].
\end{align}
The scaling of $\rc$ in this equivalence relation comes from diffeomorphism invariance of the Ricci tensor and the scaling of the $s$-coordinate since the spacetime has Ricci curvature $\ric^{\overline{g}} = \rc \, \upd s^2$.
The ODE~\eqref{eq:ODE_lambda} transforms in the parameter and gains an additional factor $\alpha^2$ on the right hand side, so that we have indeed $\bigl(s\mapsto \hat{\lambda}(\alpha s+\beta)\bigr) \in \Sol_+((\hat{g}_{\alpha s+\beta},\alpha^2\hat{\rc}_{\alpha s+\beta}), s \in I)$.

\begin{theorem}\label{thm:1-1}
	Up to isometries,
	simple pp-waves modeled on the closed connected manifold~$Q$ which have null Ricci curvature 
	are in one-to-one correspondence with $\ModSp(\bullet,Q)\quotsim$, \ie equivalence classes of parametrized moduli curves $[(g_s,\rc_s,\lambda_s)] \in \ModSp(I,Q)$, $I$ being any open interval.
\end{theorem}
\begin{remark}
  If $\Sol_{+}((g_s,\rc_s), s \in I)$ is empty for some family $(g_s,\rc_s)_{s \in I}$,
  then there is no corresponding Lorentzian metric in \cref{thm:1-1}.
	In this case, $\lambda$ turns zero somewhere for all the corresponding solutions in \cref{Thm:Ex}.
\end{remark}
\begin{remark}
  Replacing the equivalence class of $[(g_s,\rc_s,\lambda_s), s \in I]$ by $[(g_s,\rc_s,c\cdot\lambda_s), s \in I]$ for some constant $c>0$ corresponds to a homothetic rescaling of the  Lorentzian metric $\overline{g}$ by a factor $c^2$.

\end{remark}

By setting $\rc_s\equiv0$, we directly obtain a simplified correspondence for vacuum pp-waves. For this purpose, we introduce the notation
\begin{align*}
	\Mod(I,Q) &\coloneqq \{[(g_s,0,\lambda_s)]\in \ModSp(I,Q) \}.
\end{align*}
We can canonically identify this with a subset of $C^{\infty}(I,\mathcal{R}_{\vol=1}(Q) \times \R)	 /C^\infty(I,\Diff(Q))$ by leaving out the second component.
The equivalence relation defined before \cref{thm:1-1} naturally descends to an equivalence relation for this simplified kind of parametrized moduli curves.

\begin{corollary}\label{cor:Ricci-flat}
	Up to isometries,
	simple vacuum pp-waves modeled on the closed connected manifold~$Q$
	are in one-to-one correspondence with $\Mod(\bullet,Q)\quotsim$, \ie equivalence classes of parametrized moduli curves 
	$[(g_s,\lambda_s)] \in \Mod(I,Q)$.
\end{corollary}
\begin{remark}
Let us discuss heuristically why we have the right number of parameters on the right hand side of the one-to-one correspondences in \cref{thm:1-1} and \cref{cor:Ricci-flat}. 
Because $C^\infty(Q) \cong C^\infty(Q) \times \R$ -- for example by applying some Laplace operator to get the first component and calculating a mean for the second one -- it only makes sense to count the number of functions in $C^\infty(I \times Q)$, neglecting any functions in $C^\infty(I)$. 
Let $\tilde{g}_s$, $s\in I$, be a given smooth parametrized curve of unit-volume Ricci-flat metrics on $Q$. In order to get from $\tilde{g}_s$ to all possible metrics of the form \eqref{eq:AGK_metrics} with $g_s = \lambda_s^2 \phi_s^* \tilde{g}_s$ satisfying \eqref{eq:momentum_constraint}, we are a priori free to prescribe the functions $u \in C^\infty(I \times Q)$, $\lambda \in C^\infty(I)$ and the family of diffeomorphisms~$\phi_s$, $s \in I$.

According to \cref{Prop:GeodNormCoord} below, we can however always choose coordinates so that $u\equiv 1$. Therefore, by quotienting out the action of the diffeomorphisms, we effectively eliminate the degree of freedom given by the choice of $u$. The remaining freedom is to choose $\lambda$ and $\phi_s$, $s \in I$, such that $g_s$ satisfies \eqref{eq:momentum_constraint}. The latter condition is equivalent to
 \begin{align*}
\dot{g}_s=c_s\cdot g_s+\nabla^{g_s,2}v_s+h_s
\end{align*}
for some $c\in C^{\infty}(I)$, $v\in C^{\infty}(I\times Q)$ and family $h_s$ of (with respect to $g_s$) divergence-free and trace-free tensor fields, which is for each $s$ in the kernel of the Lichnerowicz Laplacian of the metric $g_s$, see \cref{prop_splitting_constraint_solutions} and \cref{rem:splitting_constraint_equations} below. Modulo initial conditions at some $s_0\in I$, the scaling (\ie $\lambda$) and gauge (\ie $\phi_s$, $s \in I$) freedom is encoded in the tuple $(c,v)$. Because we can only count modulo $C^\infty(I)$, the remaining degrees of freedom are effectively given by a single function $v\in C^{\infty}(I\times Q)$.

It thus seems natural to determine choice scaling and gauge through prescribing another function on $I\times Q$. Prescribing the Ricci curvature is a natural choice as it is a geometric invariant quantity. Assuming that it is null furthermore implies that it has exactly one non-zero component function $\rc\in C^{\infty}(I\times Q)$.
Let us moreover remark that, although not detectable by our simple count, also $\lambda$ has a natural geometric interpretation:
It can be viewed as the volumes of the horizons, \ie the lightlike hypersurfaces, orthogonal to the lightlike parallel vector field on  $(\overline{M},\overline{g})$.

\end{remark}
Finally, we turn our attention back to lightlike parallel spinors and formulate a spinorial version of \cref{thm:1-1} and \cref{cor:Ricci-flat}. We define the subsets 
\begin{align*}
\ModSpPar(I,Q)\subset\ModSp(I,Q),\qquad \ModPar(I,Q)\subset\Mod(I,Q),
\end{align*}
where the condition of Ricci-flatness is replaced by the stronger condition of having a nontrivial parallel spinor.
The equivalence relation obviously descends to these subspaces.

\begin{theorem}\label{thm:1-1_spin_case}
Up to isometries,
	simple pp-waves with a lightlike parallel spinor modeled on the closed connected manifold $Q$ 
	are in one-to-one correspondence with $\ModSpPar(\bullet,Q)\quotsim$, \ie equivalence classes of parametrized moduli curves $[(g_s,\rc_s,\lambda_s)] \in \ModSpPar(I,Q)$, $I$ being any open interval. 
	In particular,  simple vacuum pp-waves with a lightlike parallel spinor are in one-to-one correspondence  with $\ModPar(\bullet,Q)\quotsim$, \ie equivalence classes of parametrized moduli curves $[(g_s,\lambda_s)] \in \ModPar(I,Q)$.
\end{theorem}
Finally, we address the question which simple pp-waves modeled on a closed manifold admit spatially closed quotients in the sense that in the quotient there exists a compact spacelike hypersurface $M$ without boundary such that every integral curve of the lightlike parallel vector field intersects $M$ exactly once. 
We call these quotients \emph{spatially compact pp-waves}.
Obviously, such quotients can be constructed if  $I=\R$ and there exists a parameter $\ell$ and a diffeomorphism $\varphi\in\mathrm{Diff}(Q)$ such that $u_{s+\ell}=\varphi^* u_s$, $g_{s+\ell}=\varphi^*g_s$ and $\lambda_{s+\ell}=\lambda_s$ for all $s\in\R$. 
From the physical perspective, pp-waves satisfying the dominant energy condition are particularly interesting. These can have spatially closed quotients only in very rigid cases, as the following theorem shows.
\begin{theorem}\label{thm:spatially_compact}
A spatially compact pp-wave $(\overline{M},\ol{g})$ with null Ricci curvature which satisfies the dominant energy condition is Ricci-flat, hence a vacuum pp-wave. Moreover, it is 
 locally isometric to 
\begin{align*}
(I\times J\times Q,\upd v \otimes \upd s + \upd s \otimes \upd v + \upd s^2 +g),
\end{align*}
where $g$ is a Ricci-flat metric on $Q$, not depending on $s$.
\end{theorem}
Note in particular that the universal cover of $(\overline{M},\ol{g})$ as in the theorem is isometric to a product of the $1+1$-dimensional Minkowski space with a simply-connected Ricci-flat manifold.
Note that a simply-connected Ricci-flat manifold may be decomposed as Riemannian product into Euclidean factors and closed simply-connected Ricci-flat ones.
For a more detailed statement of \cref{thm:spatially_compact}, see \cref{thm:mapping_torus}.
As a corollary, we get:
\begin{corollary}
Let $(\overline{M},\ol{g})$ be a spatially compact pp-wave with lightlike parallel spinor which satisfies the dominant energy condition. Then the assertions of \cref{thm:spatially_compact} do hold.
\end{corollary}
\begin{remark}
We cannot exclude the assumption of null Ricci curvature. For example, the spatially compact simple pp-wave $(\R\times S^1\times S^{n-1},\upd v \otimes \upd s + \upd s \otimes \upd v + \upd s^2 +\gsph)$ satisfies the dominant energy condition but its Ricci curvature is not null. Here $\gsph$ is the standard metric on $S^{n-1}$. 
\end{remark}
This paper is structured as follows. In \cref{sec.hypersurfaces.exp-maps}, we collect some facts and properties of pp-waves, including the construction of special coordinate systems and a useful initial value problem formulation for pp-waves. In \cref{sec.classification}, we prove Theorem \ref{Thm:Ex} for pp-waves with prescribed Ricci curvature and the main classification Theorem \ref{thm:1-1}. In Section \ref{sec.constr.cyl.ids.llps}, we focus on spin manifolds and prove Theorem \ref{thm:par_spinor} which together with Theorem \ref{thm:1-1} implies the classification Theorem \ref{thm:1-1_spin_case} for pp-waves with lightlike parallel spinor. Finally, we focus on the spatially compact case in Section \ref{sec.spatially_compact}, where we prove \cref{thm:spatially_compact}.

\subsection*{Acknowlegements}
The authors want to thank Thomas Leistner for many good conversations and talks related to this article. Furthermore,
J.G.\ wants to thank Xiaoxiang Chai for some discussions on the topics related to \cref{Prop:FlatCurvVan}. We also thank Piotr Chrusciel, Anton Galaev and Olaf Müller for stimulating discussions on nearby topics.  

	All authors have been partially supported by SPP 2026 (Geometry at
	infinity), B.A.\ and J.G.\ also by the SFB 1085 (Higher Invariants), and J.G. by the Walter Benjamin grant no.~556505019, all funded by the
	DFG (German Science Foundation).
	K.K.\ acknowledges additional financial support by the Göran Gustafsson foundation.

\section{Properties of pp-wave spacetimes}\label{sec.hypersurfaces.exp-maps}
\subsection{Geometric conditions for simple pp-waves}

We start this subsection with a geometric characterization of simple pp-waves.
\begin{proposition}\label{lem:global_form}
A Lorentzian manifold $(\overline{M},\overline{g})$ is a simple pp-wave (modeled on a compact and connected manifold $Q$) if and only if the following assertions do hold:
\begin{itemize}
\item There exists a lightlike parallel vector field $V$ which is complete and the gradient field of a globally defined function $s\in C^{\infty}(\overline{M})$.
\item There is a connected spacelike hypersurface $M\subset\overline{M}$ such that
\begin{itemize}
\item Each integral curve of $V$ intersects $M$ exactly once
\item All level sets $M\cap s^{-1}(c)$ are compact and connected.
\end{itemize}
\end{itemize}
\end{proposition}
\begin{proof}
If $(\overline{M},\overline{g})$ is globally isometric to a metric of the form
\begin{align}\label{eq:AGK_metrics_global}
 (\R\times I\times Q, \upd v\otimes \upd s+\upd s\otimes \upd v+ u^{-2} \upd s^2 + g_s),
	\end{align}
	then all the stated assertions hold with $s:\R\times I\times Q\to I$ being the projection onto the $s$-coordinate, $V=\frac{\partial}{\partial v}$ and $M=\left\{0\right\}\times I\times Q$.
	
Conversely, let $(\overline{M},\overline{g})$ be a Lorentzian manifold, which admits a vector field $V$, a function $s$ and a hypersurface $M$ with properties as stated in the proposition.
Let $e_0$ be the time-like unit normal field of~$M$ pointing in the direction of $V$. Along $M$, the lightlike parallel vector field $V$ splits up uniquely as
\begin{align*}
V=u\cdot e_0-U,
\end{align*}
where $u:M\to\R$ is a smooth positive function and $U$ is a smooth vector field on $M$ which is nowhere vanishing. Let $\gamma$ be the induced metric on $M$ and note that $u = |U|_\gamma$ since
\begin{align*}
0=\overline{g}(V,V)=-u^2+\overline{g}(U,U)=-u^2+\gamma(U,U).
\end{align*}
Because $V=\grad^{\overline{g}}s$ for a function $s\in C^{\infty}(\overline{M})$, we get $U=-\grad^{\gamma}s$. Because $U$ is nowhere vanishing, $s$ has no critical points on $M$. In particular, there $s$ does not admit maximum and minimum on $M$. Because $M$ is connected, the set 
$I:=s(M)\subset \R$ is an open interval.

Consider the vector field $Z=-u^{-2}U$. It is proportional to $\mathrm{grad}^{\gamma}s$, hence orthogonal to the level sets $Q_z=M\cap s^{-1}(z)\subset M$. Let
\begin{align*}
\mathrm{Fl}^Z:\R\times M\supset \mathcal{U}\longrightarrow M 
\end{align*}
be the flow generated by $Z$.  As the level sets $Q_z$ are compact, we find for each $z\in I$ an $\epsilon>0$ such that for all $x\in Q_z$ the flow $\mathrm{Fl}^Z_t(x)=\mathrm{Fl}^Z(t,x)$ is defined for all $t\in (-\epsilon,\epsilon)$. 
The proportional factor for $Z$ is chosen so that $\mathrm{Fl}^Z_t$ maps $Q_z$ to $Q_{t+z}$. Since $Q_z$ is compact, its image $\mathrm{Fl}^Z_t(Q_z)\subset Q_{z+t}$ is also compact. On the other hand, since $\mathrm{Fl}^Z_t$ is a local diffeomorphism, the image is also open. Because all $Q_z$ are connected, $\mathrm{Fl}^Z_t:Q_z\to Q_{z+t}$ is a diffeomorphism. Patching these diffeomorphisms together, we see that for every $z\in I$ and every $t\in\R$ such that $t+z\in I$, the flow restricts to a diffeomorphism
\begin{align*}
\mathrm{Fl}^Z_t:Q_z\longrightarrow Q_{z+t}.
\end{align*}
In particular, all the level sets $Q_z$ are diffeomorphic.
By choosing $z_0\in I$ and letting $Q:=Q_{z_0}$, we therefore get a global diffeomorphism
\begin{align*}
\Phi:I\times Q\longrightarrow M,\qquad (s,x)\longmapsto \mathrm{Fl}^Z_{s-z_0}(x).
\end{align*}
Notice that the first coordinate function coincides with the previously defined function $s$, so there is no conflicting notation in using the letter $s$ here.
Because $\gamma(Z,Z)=u^{-4}\gamma(U,U)=u^{-2}$ and $Z$ is orthogonal to the level sets $Q_z$, we obtain $\Phi^*\gamma=u^{-2}\upd s^2+g_s$.

To conclude the proof, we recall that $V$ is complete. We use its flow to define a map
\begin{align*}
\Psi:\R\times I\times Q\longrightarrow \overline{M},\qquad (v,s,x)\longmapsto \mathrm{Fl}^V_{v}(\Phi(s,x)).
\end{align*}
Because the flow lines of $V$ intersect $M$ exactly once, the map $\Psi$ defines a global diffeomorphism.
Since $\overline{g}(\grad^{\overline{g}} s, V) = \overline{g}(V,V) = 0$, the flow of $V$ preserves the $s$-level sets.
So again, calling the second coordinate function $s$ causes no ambiguity. 

It now remains to compute $\Psi^*\overline{g}$.
Because $V$ is in particular a Killing vector field, $\mathrm{Fl}^V_{v}$ is a family of isometries.
So the coefficients of $\Psi^* \overline{g}$ are independent on $v$ and it suffices to calculate it along $\{0\} \times I \times Q$.
Here, both $\frac{\del}{\del v} = \upd (\Psi^{-1})(V)$ and $\frac{\del}{\del s} = \upd (\Psi^{-1})(Z)$ are $\overline{g}$-orthogonal to the level sets $Q_s$ and the induced metric on $Q_s$ is $g_s$.
Since, moreover,
\begin{align*}
	\overline{g}(V,V)=0,\qquad \overline{g}(V,Z)=\overline{g}(u\cdot e_0-U,Z)=\overline{g}(-U,Z)=1, \qquad \overline{g}(Z,Z)=u^{-2}
\end{align*}
the claimed formula follows immediately.
\end{proof}
\begin{remark} \label{Rem:PPLocally}
As a byproduct of the above proof, we obtain the well-known fact that every Lorentzian manifold with a lightlike parallel vector field is locally of the form 
\begin{align}\label{eq:AGK_metrics_2}
(J\times I\times Q,\upd v\otimes \upd s+\upd s\otimes \upd v+u^{-2}\upd s^2+g_s)
\end{align}
stated in the introduction
\end{remark}
\begin{remark}
The potential function $s\in C^{\infty}(\overline{M})$ is uniquely determined up to a constant. We may furthemore multiply $s$ with a constant. This corresponds to a constant rescaling of the vector field $V$.
\end{remark}
In the form \eqref{eq:AGK_metrics_global}, the initial spacelike hypersurface $M$ corresponds to the hypersurface $\left\{v=0\right\}$. A different spacelike hypersurface $\widetilde{M}$ can in these coordinates be represented by a graph 
\begin{align*}
\widetilde{M}=\left\{(f(s,x),s,x)\mid f\in C^{\infty}(I\times Q)\right\}.
\end{align*}
A natural diffeomorphism on $\overline{M}$ mapping $M$ onto $\widetilde{M}$ is given by
\begin{align*}
F:(\lt ,s,x)\mapsto (\lt  +f(s,x),s,x).
\end{align*}
A quick computation shows that 
\begin{align*}
F^*\overline{g}_{(\lt ,s,x)}=\begin{pmatrix} 0 & 1 & 0\\
 1 & u^{-2}+2\dot{f} & df_s\\
 0 & df_s & g_s
 \end{pmatrix}=
 \begin{pmatrix} 0 & 1 & 0\\
 1 & v & df_s\\
 0 & df_s & g_s
 \end{pmatrix},
\end{align*}
where  $v:= u^{-2} + 2\dot{f}$.

In order to obtain the form \eqref{eq:AGK_metrics_global}, we have to modify the above diffeomorphism.
We define an $s$-dependant family of vector fields on $Q$ by $X_s=\grad^{g_s} (f_s)$. We then have
  $$g_s(X,X)=df_s(X)=F^*\overline{g}(\partial_s,X).$$
Let $\phi_s\in \mathrm{Diff}(Q)$ be the diffeomorphisms generated by field $-X_s$, i.e.
  $$\dot{\phi}_s=-X_s\circ\phi_s, \quad \phi_0=\id_Q.$$
Let $\Phi:(\lt ,s,x)\mapsto (\lt ,s,\phi_s(x))$. Then we get
\begin{align*}
\Phi^*(F^*\overline{g})_{(\lt ,s,x)}&= \begin{pmatrix} 0 & 1 & 0\\
 1 & v\circ\phi_s-g_s(X_s,X_s)\circ\phi_s & 0\\
 0 & 0 & \phi_s^*g_s
 \end{pmatrix}
 = 
 \begin{pmatrix} 0 & 1 & 0\\
 1 & \phi_s^*\left(u_s^{-2} + 2\dot{f}_s - |\upd f_s|_{g_s}^2\right) & 0\\
 0 & 0 & \phi_s^*g_s
 \end{pmatrix}.
\end{align*}
Thus, the diffeomorphism
\begin{align*}F\circ \Phi:(\lt ,s,x)\mapsto \Bigl(\lt +f\bigl(s,\phi_s(x)\bigr),s,\phi_s(x)\Bigr).
\end{align*}
preserves the form \eqref{eq:AGK_metrics_global}. Note that since $Q$ is closed, the family of vector fields $X_s$ is complete. Therefore, $F\circ \Phi$ is defined on all of $\overline{M}$.
We summarize these calculations in the following lemma.
\begin{lemma} \label{Lem:ChangeHypersurface}
	Suppose that $(\overline{M},\overline{g})$ is given by \eqref{eq:AGK_metrics_global}.
	Let $f \in C^\infty(I \times Q)$ and $\phi_s$, $s \in I$, be a family of diffeomorphisms such that $\dot{\phi}_s = -\grad^{g_s}(f_s) \circ \phi_s$.	
	Then
	\begin{align*}
	\Psi \colon \R \times I \times Q &\longrightarrow \R \times I \times Q,\\
	(v,s,x) &\longmapsto \left(v+f\left(s,\phi_s(x)\right),s,\phi_s(x)\right)
	\end{align*}
	defines a diffeomorphism and
	\begin{gather*}
	\Psi^* \overline{g} = \upd v\otimes \upd s+\upd s\otimes \upd v+ \phi_s^*\left(u_s^{-2} + 2\dot{f}_s - |\upd f_s|_{g_s}^2 \right) \upd s^2 + \phi_s^*g_s.
	\end{gather*}
\end{lemma}
\subsection{Initial data sets}
For simple pp-waves, there is a good correspondence between the spacetime and data induced on a spacelike hypersurface intersecting all the flow lines once.
Later, when we construct parallel spinors, this allows us later to work with the Riemannian data on such a hypersurface, for which the analysis is better behaved, rather than on a Lorentzian manifold.

We recall that an \emph{initial data set} on a manifold $M$ is a pair $(\gamma,k)$ consisting of a Riemannian metric $\gamma$ on $M$ and a symmetric $2$-tensor $k$ on $M$.
If $M$ is a co-oriented spacelike hypersurface within a Lorentzian manifold $(\overline{M},\overline{g})$ then the \emph{induced initial data set} on $M$ consists of the induced metric $\gamma$ and the induced second fundamental form $k$, where the latter satisfies for all $X,Y \in \Gamma(TM)$
\begin{gather*}
	\nabla^{\overline{g}}_X Y = \nabla^{\gamma}_X Y + k(X,Y) e_0,
\end{gather*}
where $e_0$ is the unit normal defining the co-orientation.
It is an important property of initial data sets that by virtue of the Gauß and Codazzi equation certain components of the Einstein tensor of $(\overline{M}, \overline{g})$ are determined by $(\gamma,k)$.
More precisely, $\Ein^{\overline{g}}(e_0,\bullette) = -\mu \overline{g}(e_0, \bullette) + j$ for
\begin{equation} \label{eq:Constraints}
	\begin{aligned}
		\mu &\coloneqq \frac{1}{2}(\scal^\gamma + \tr^\gamma(k) - |k|_\gamma^2) \\
		j &\coloneqq \div^\gamma(k) - \upd \tr^\gamma(k).
	\end{aligned}
\end{equation}

If $(\overline{M},\overline{g})$ carries a lightlike parallel vector field $V$, then $V$ determines a time-orientation of $(\overline{M},\overline{g})$ and every spacelike hypersurface carries a co-orientation such that $\overline{g}(e_0,V) < 0$.
Then there is a unique decomposition $V_{|M} = ue_0 -U$ with positive $u \in C^\infty(M)$ and $U \in \Gamma(TM)$ and lightlikeness implies $u = |U|_{\gamma}$.
The tangential and the normal part of the equation $\nabla^\gamma_X V = 0$ for $X \in TM$ translate to
\begin{align*}
	\nabla^{\gamma}_X U &= uk(X, \bullette)^{\sharp^{\gamma}} \\
	\upd u(X) &= k(U,X),
\end{align*}
where actually the second equation follows from the first one if $u = |U|_{\gamma}$.

\begin{definition}
	Given an initial data set $(\gamma,k)$ on a manifold $M$, an nowhere vanishing vector field $U \in \Gamma(TM)$ is said to be \emph{lightlike-parallel} if 
	\begin{gather}  \label{eq:UPar}
		\nabla^{\gamma}_X U = |U|_{\gamma}k(X, \bullette)^{\sharp^{\gamma}}
	\end{gather}
	holds.
	If there is a lightlike-parallel vector field for $(\gamma,k)$, then we call it a \emph{pp-wave initial data set}.
\end{definition}

\begin{remark} \label{Rem:DetK}
	Through \eqref{eq:UPar}, the second fundamental form $k$ is completely determined by $\gamma$ and $U$ if $U$ is lightlike-parallel for $(\gamma,k)$.
\end{remark}

The discussion so far explains how to pass from pp-wave spacetimes to pp-wave initial data sets.
It is well-known that it is also possible to go back.
The construction is a special case of the Killing development studied by Beig and Chruściel \cite[eq.~(2.14)]{Beig.Chrusciel:1996}.

\begin{proposition}
	Let $(\gamma,k)$ be an initial data set on $M$ with lightlike-parallel vector field $U$.
	Then
	\begin{align} \label{eq:KilligDev}
		(\overline{M},\overline{g}) \coloneqq (\R \times M, -\upd v \otimes \gamma(U,\bullette) - \gamma(U,\bullette) \otimes \upd v + \gamma),
	\end{align}
	where $v$ is the $\R$-coordinate, is a pp-wave spacetime with complete lightlike parallel vector field $\frac{\del}{\del v}$ such that the induced initial data set on $M = \{0\} \times M$ is given by $(\gamma,k)$ and the orthogonal projection of $\frac{\del}{\del v}_{|M}$ on $TM$ is given by $-U$.
\end{proposition}
\begin{proof}
	It is clear that $V \coloneqq \frac{\del}{\del v}$ is a complete lightlike Killing vector field.
	Moreover, for all $X,Y \in TM$
	\begin{align*}
		\overline{g}(X,Y) &= \gamma(X,Y) & &\text{ and } &\overline{g}(V,X) &= -\gamma(U,X),
	\end{align*}
	and thus the induced metric on $M = \{0\} \times M$ is $\gamma$ and the orthogonal projection of $V$ on $M$ is $-U$.
	We shall prove that $\nabla^{\overline{g}} V = 0$ on $M$.
	Since every flow line of $V$ intersects $M$ and $V$ is a Killing vector field this implies that $V$ is parallel.
	Then $U$ is lightlike-parallel for the induced initial data set on $M$.
	Since the induced metric on $M$ is $\gamma$ and \eqref{eq:UPar} holds for $(\gamma,k)$, the induced second fundamental form must be given by $k$.
	
	To show that $\nabla^{\overline{g}} V = 0$ on $M$, we first observe that
	\begin{gather*}
		\overline{g}(\nabla^{\overline{g}}_W V, V) = \frac{1}{2} \del_W \overline{g}(V, V) = 0
	\end{gather*}
	for all $W \in T\overline{M}_{|M}$.
	Because of the Killing equation, this also implies that
	\begin{gather*}
		\overline{g}(\nabla^{\overline{g}}_V V, X) = -\overline{g}(\nabla^{\overline{g}}_X V, V) = 0
	\end{gather*}
	for all $X \in TM$.
	The remaining components of $\nabla^{\overline{g}} V$ are also zero, because for all $X, Y \in TM$ we have
	\begin{align*}
		2\overline{g}(\nabla^{\overline{g}}_X V, Y) 
			&= \overline{g}(\nabla^{\overline{g}}_X V, Y) - \overline{g}(\nabla^{\overline{g}}_Y V, X) \\
			&= \del_X \overline{g}(V, Y) - \del_Y \overline{g}(V, X) - \overline{g}(V, [X,Y]) \\
			&= -\del_X \gamma(U,Y) + \del_Y \gamma(U,X) - \gamma(U, [X,Y]) \\
			&= -\gamma(\nabla^{\gamma}_X U, Y) + \gamma(\nabla^{\gamma}_Y U, X) \\
			&= -|U|_{\gamma} (k(X,Y) - k(Y,X)) = 0. \qedhere
	\end{align*}  
\end{proof}

\begin{corollary} \label{Cor:1-1-IDS}
	The restriction and extension procedures yield a one-to-one correspondence between
	\begin{itemize}
		\item simple pp-waves modeled on $Q$ together with a spacelike hypersurface intersecting all flow lines of the canonical lightlike parallel vector field precisely once, and
		\item pp-wave initial data sets on $I \times Q$ for $I \subset Q$ an open interval, where the metric is of the form
		\begin{gather*}
		\gamma = u^{-2} \upd s^2 + \lambda^2g_s
		\end{gather*}
		and $-u^2\frac{\del}{\del s}$ is lightlike-parallel for some positive $u \in C^\infty(I \times Q)$ and a smooth family of Riemannian metrics $g_s$, $s \in I$, on $Q$.
	\end{itemize}  
\end{corollary}
\begin{proof}
	After potentially applying \cref{Lem:ChangeHypersurface}, we may assume that the spacelike hypersurface is given by $\{0\} \times I \times Q$ inside
	\begin{gather} \label{eq:ppWave}
		(\overline{M},\overline{g}) \coloneqq (\R\times I\times Q, \upd v\otimes \upd s+\upd s\otimes \upd v+ u^{-2} \upd s^2 + g_s)
	\end{gather}
	for some open interval $I \subset \R$, a positive function $u \in C^\infty(I \times Q)$ and a smooth family $g_s$, $s \in I$, of Riemannian metrics on $Q$.
	It is clear that the induced metric $\gamma$ on $M \coloneqq \{0\} \times I \times Q$ is as stated and due to
	\begin{align*}
		\overline{g}\left(\frac{\del}{\del v}, X\right) = \upd s(X) = \gamma\left(u^2\frac{\del}{\del s}, X\right)
	\end{align*}
	for all $X \in TM$ the tangential projection of the canonical lightlike parallel vector field $\frac{\del}{\del v}$ is given by $-U = u^2 \frac{\del}{\del s}$.
	
	Conversely, the formula~\eqref{eq:KilligDev} for the Killing development gives precisely back~\eqref{eq:ppWave} from the pp-wave initial data set $(\gamma, k)$ on $I \times Q$ with $\gamma = u^{-2} \upd s^2 + g_s$ and lightlike-parallel vector field $U = -u^2 \frac{\del}{\del s}$, because
	\begin{gather*}
		\gamma(U,\bullette) = -\upd s. \qedhere
	\end{gather*}
\end{proof}

This correspondence also extends to spinors that are parallel in a suitable sense.
We will discuss spinors in \cref{sec.constr.cyl.ids.llps}.

\subsection{Geodesic coordinates}
Note that if $(\overline{M},\overline{g})$ is of the form \eqref{eq:AGK_metrics_2}, the lightlike hypersurfaces $\left\{s=const\right\}$ are the integral manifolds of the distribution $V^{\perp}$, since $V$ is the $\overline{g}$ dual of $\upd s$. 

We know a priori that if the lightlike vector field $V$ is parallel, $V^{\perp}$ is an integrable distribution since the $1$-form $\overline{g}(V,.)$ is closed.
We can therefore find the hypersurfaces $\left\{s=const\right\}$ without assuming the local form \eqref{eq:AGK_metrics_2} by defining them as the leaves of the foliation defined by $V^{\flat}$. In the following, we will demonstrate that we may set $u\equiv 1$ in \eqref{eq:AGK_metrics_2}.
\begin{proposition}[Metric in geodesic normal coordinates] \label{Prop:GeodNormCoord}
	Let $(\overline{M},\overline{g})$ be a Lorentzian manifold with lightlike parallel vector field $V$, $\mathcal{L}$ be a leaf of the foliation defined by $V^\flat$ and $Q$ a spacelike hypersurface of $\mathcal{L}$.
	Then an open neighborhood $\mathcal{O}$ of $Q$ in $\mathcal{L}$ can be parametrized by a real parameter and a point in $Q$ using the normal exponential map\footnote{or, which is the same in this case: the flow of $V$} $(v,x) \mapsto \exp_x(vV)$.
	Let, furthermore, $\normal$ be the vector field along $\mathcal{O}$ with $\overline{g}(\normal,\normal) = 1$, $\overline{g}(\normal,V) = 1$ and $\normal \perp TQ_t$ for the foliation of $\mathcal{O}$ determined by 
\begin{align*}	
	Q_v = \{\exp_x(vV) \in \mathcal{O} \,|\, x \in Q\},\qquad v \in J.
	\end{align*}
	Define geodesic normal coordinates defined with respect to $\mathcal{O}$ and $\normal$, i.\,e.\ 
	\begin{align*}
		\Phi \colon \R \times \R \times Q \supset \mathcal{D} &\longrightarrow \overline{M} \\
			(v,s,x) &\longmapsto \exp_{\exp_x(vV)}(s\normal).
	\end{align*}
	Then, the pull-back metric on the domain of definition is of the form 
\begin{align*}	
	\Phi^* \overline{g} = \upd v \otimes \upd s + \upd s \otimes \upd v + \upd s^2 + g_s,
\end{align*}	
	 where $g_s$ is a Riemannian metric on a subset of $Q$ for each $s$.
	If $Q$ is compact, $\mathcal{D}$ can be chosen to be of the form $J \times I \times Q$ for intervals $J,I \subset \R$ with $J = \R$ if $V$ is complete.
\end{proposition}
\begin{proof}
	We write $\gamma_{v,x}$ for the geodesic defined by $\gamma_{v,x}(s) = \exp_{\exp_x(vV)}(s\normal)$.
	We see that 
	\begin{align*}\frac{\del}{\del s}\big\vert_{{v,s,x}} = \dot{\gamma}_{v,x}(s).
	\end{align*}
	 Because geodesics have constant speed and $\dot{\gamma}_{v,x}(0) = \normal|_{\exp_x(vV)}$, we have
\begin{align*}
\overline{g}\left(\frac{\del}{\del s}, \frac{\del}{\del s}\right) = \overline{g}(\normal,\normal) = 1.
\end{align*}
Let $X \in TQ$ and extend it to a vector field on $\mathcal{L}$ by flowing it along the isometries generated by $V$.	We consider the Jacobi field $J_X$ along $\gamma_{v,q}$ given by 
\begin{align*}J_X(s) = \frac{\upd}{\upd \tau}\big\vert_{\tau = 0} \gamma_{t,x(\tau)}(s)
\end{align*}
 for any curve $\tau \mapsto q(\tau)$ with $\dot{q}(0) = X$.
	Since 
	\begin{align*}
	\overline{g}(J_X(0), \dot{\gamma}_{v,q}(0)) = \overline{g}(X,\normal) = 0
\end{align*}	
	 and 
	\begin{align*}
\frac{d}{d s}\overline{g}(J_X(s),\dot{\gamma}_{v,x}(s))
& =	\overline{g}\left(\frac{\overline{\nabla}}{\upd s} J_X(s), \dot{\gamma}_{v,x}(s)\right)\\
& = \overline{g}\left(\frac{\overline{\nabla}}{\upd \tau}\big\vert_{\tau=0} \frac{\upd}{\upd s}\gamma_{v,x(\tau)}(s), \dot{\gamma}_{v,q}(s)\right) = \frac{1}{2} \del_X \overline{g}(\normal,\normal) = 0,
	\end{align*}
we obtain
\begin{align*}
\overline{g}\left(J_X(s),\frac{\del}{\del s}\right)=0.
\end{align*}
In order to compute the remaining components of $\Phi^*\overline{g}$, we first want to show that
\begin{align}\label{eq:two_Jacobi_fields}
V= \frac{\del}{\del v}=\frac{\upd}{\upd v} \gamma_{v,q}(s).
\end{align}
Along a geodesic of the form $s\mapsto  \gamma_{v,q}(s)$, $\frac{\del}{\del v}$ restricts to the Jacobi field $J_V(s) = \frac{\upd}{\upd v} \gamma_{v,q}(s)$. Because $V$ is parallel, it also restricts to a Jacobi field along  $s\mapsto  \gamma_{v,q}(s)$. In order to show \eqref{eq:two_Jacobi_fields}, it suffices to prove that along any such geodesic, these two Jacobi fields coincide. By construction, we already know that $J_V(0) = V$ so that it remains to show that the initial derivatives coincide.
We compute 
\begin{align*}
\frac{\overline{\nabla}}{\upd s}J_V(0) = \frac{\overline{\nabla}}{\upd v}\frac{\upd}{\upd s}\big\vert_{s=0} \gamma_{t,q}(s) = \overline{\nabla}_V \normal 
\end{align*}
and 
\begin{align*}
	\overline{g}(\overline{\nabla}_V \normal, \normal) = \frac{1}{2} \del_V \overline{g}(\normal,\normal) &= 0,\\
	 \overline{g}(\overline{\nabla}_V \normal, V) = - \overline{g}(\normal,\overline{\nabla}_V V) &= 0,\\
\overline{g}(\overline{\nabla}_V \normal, X) = - \overline{g}(\normal,\overline{\nabla}_V X) = -\overline{g}(\normal,\overline{\nabla}_X V)& = 0.
\end{align*}
In the last line, we used that any vector field $X \in TQ$ is extended to $\mathcal{L}$ by flowing it along the isometries generated by $V$.
Consequently, $\frac{\overline{\nabla}}{\upd s}J_V(0)=\overline{\nabla}_V \normal =0$. Because $V$ is parallel, \eqref{eq:two_Jacobi_fields} follows.

\noindent Because $V$ is parallel and $s\mapsto \gamma_{v,x}(s)$ is a geodesic,
\begin{align*}
\frac{d}{d s}\overline{g}(V,\dot{\gamma}_{v,x}(s))=0,
\end{align*}
so that 
\begin{align}\label{eq:geod_coord_sp}
\overline{g}\left(\frac{\del}{\del s}, V\right) = 1.
\end{align}
In order to finish the proof of the theorem, it remains to show that
\begin{align*}
\overline{g}\left(J_X(s), V\right)=0.
\end{align*}
We know that this equality holds for $s=0$. Differentiating in the $s$-direction and using that $V$ is parallel yields
\begin{align*}
\frac{d}{d s}\overline{g}(J_X(s),V)=\overline{g}(\frac{\overline{\nabla}}{\upd s}J_X(s),V)=
 \overline{g}\left(\frac{\overline{\nabla}}{\upd \tau}\big\vert_{\tau=0} \frac{\upd}{\upd s}\gamma_{v,x(\tau)}(s), V\right) 
 =\frac{d}{d \tau}\big\vert_{\tau=0}\overline{g}\left(\frac{\del}{\del s}|_{v,x(\tau),s},V\right)=0,
\end{align*}
where we used \eqref{eq:geod_coord_sp} in the last equality. This finishes the proof.
\end{proof}
\begin{remark}
Even if $(\overline{M},\overline{g})$ is of the form \eqref{eq:AGK_metrics_global}, the geodesic coordinates in \cref{Prop:GeodNormCoord} may not be defined on all of $\overline{M}$. The geodesic sent out at $\mathcal{L}$ in the direction of $\normal$ could a priori develop focal points at which the map $\Phi$ fails to be a diffeomorphism.
\end{remark}

\subsection{Lightlike Ricci curvature}\label{sec.compute.ricci}

\begin{lemma}\label{lem.lor.ricci-flat}
	Let $Q$ be a manifold, $(g_s)$, $s\in I$ be a curve of Riemannian metrics on $Q$ and $u \in C^{\infty}(I \times Q)$ a positive function.
	On $\LorMan\ceq\R\times (a,b) \times Q$ we define 
\begin{align*}	
	\LorMet \coloneqq \upd v \otimes \upd s + \upd s \otimes \upd v + u^{-2}\upd s^2 + g_s.
	\end{align*}
	Then $\LorMet$ has null Ricci curvature if and only if
     \begin{gather} \nonumber
   		\ric^{g_s}=0, \text{ and}\\
   		\label{eq.level.ricci}
   		\div^{g_s}\dot{g}_s-\upd\,\trace^{g_s} \dot{g}_s =0,
     \end{gather}
	for each $s\in I$.
	In this case, $\ric^{\LorMet}=\rc_s\,\upd s^2$ for a function $\rc:I\times Q\to\R$, which is given by
	\begin{align} \label{eq:CalcRc}
		 \rc_s &= \frac{1}{2}\Delta^{g_s} (u^{-2}) -\frac{1}{2}\frac{\upd}{\upd s}\trace^{g_s}\dot{g}_s-\frac{1}{4}|\dot{g}_s|^2_{g_s},
	\end{align}
	for each $s$.
\end{lemma}
\begin{proof}
First we observe that since $V$ is parallel,
\begin{align*}
\overline{g}(\Ric^{\overline{g}}(X),V) = \ric^{\overline{g}}(X,V) = 0
\end{align*}
for all vector fields $X$. If 
$\Ric^{\overline{g}}(X)$ is null for all $X \in T\overline{M}$, we therefore  obtain that $\Ric^{\overline{g}}(X)=\alpha(X)\cdot V$ for some $1$-form $\alpha(X)$ on $\overline{M}$.

The remainder of the proof is a straightforward computation in local coordinates. Let $\left\{i,j,k,\ldots\right\}$ be indices on $Q$, denote by $s,v$ the coordinates on $I$ and $\R$, respectively. Let $\left\{\mu,\nu,\lambda,\ldots\right\}$ be coordinates on $\overline{M}$. In matrix form, the metric and its inverse are
\begin{align}\label{def.LorMet}
\LorMet=\begin{pmatrix} 0 & 1 & 0\\
1 & u^{-2} & 0\\
0 & 0 & g_s,
\end{pmatrix},
\qquad \qquad
\LorMet^{-1}=\begin{pmatrix} -u^{-2} & 1 & 0\\
1 & 0 & 0\\
0 & 0 & g_s^{-1}
\end{pmatrix}.
\end{align}
The nonvanishing Christoffel symbols are
\begin{align*}
\Gamma(\LorMet)_{ij}^k&=\Gamma(g_s)_{ij}^k, & \Gamma(\LorMet)_{ij}^v&=-\frac{1}{2}(\dot{g}_s)_{ij},& \Gamma(\LorMet)_{is}^j&=\Gamma(\LorMet)_{si}^j=\frac{1}{2}(g_s)^{jk}(\dot{g}_s)_{ki}, \\
\Gamma(\LorMet)_{is}^v &= \Gamma(\LorMet)_{si}^v = \frac{1}{2}
\partial_i(u^{-2}), &
\Gamma(\LorMet)_{ss}^i &= -\frac{1}{2}(g_s)^{ij}\partial_j(u^{-2}), &
\Gamma(\LorMet)_{ss}^v &= \frac{1}{2}\frac{\upd}{\upd s}(u^{-2})
\end{align*}
Recall that the components of the Ricci tensor are computed through the formula
\begin{align*}\ric_{\mu\nu} = \del_{\xi} \Gamma_{\mu\nu}^{\xi} + \Gamma_{\xi \tau}^{\xi}\Gamma_{\mu\nu}^{\tau} - \del_{\mu} \Gamma_{\xi\nu}^{\xi} - \Gamma_{\mu \tau}^{\xi}\Gamma_{\xi\nu}^{\tau}.
\end{align*}
Its nonvanishing components are
\begin{align*}
\ric(\LorMet)_{ij}&=\ric(g_s)_{ij},\\
\ric(\LorMet)_{ss}&=\frac{1}{2}\Delta^{g_s} (u^{-2}) -\frac{1}{2} \frac{\upd}{\upd s}\trace^{g_s}\dot{g}_s-\frac{1}{4}|\dot{g}_s|^2_{g_s},\\
\ric(\LorMet)_{si}=\ric(\LorMet)_{is}&=\frac{1}{2}[\div^{g_s}\dot{g}_s-\upd\tr^{g_s}\dot{g}_s]_i.
\end{align*}
Regarded as an endomophism, its nonvanishing components are
\begin{align*}
\Ric(\LorMet)_s^t&=\ric(\LorMet)_{ss}=\frac{1}{2}\Delta^{g_s}(u^{-2})-\frac{1}{2}\frac{\upd}{\upd s}\trace^{g_s}\dot{g}_s-\frac{1}{4}|\dot{g}_s|^2_{g_s},\\
\Ric(\LorMet)^i_s=\Ric(\LorMet)_i^v&=\frac{1}{2}[\div^{g_s}\dot{g}_s-\upd\,\trace^{g_s}\dot{g}_s]_i,\\
\Ric(\LorMet)^i_j&=\Ric(g_s)^i_j.
\end{align*}
The condition $\Ric^{\overline{g}}(X)=\alpha(X)\cdot V$ for every vector field $X$ on $\overline{M}$ is equivalent to say that
$\Ric(\LorMet)_{\beta}^{\alpha}=0$ for each $\alpha\neq v$.
The statement follows.
\end{proof}

The condition~\eqref{eq.level.ricci} will play an important role.
As it describes how the leaves $\R \times \{s\} \times Q$, $s \in I$, and their Ricci-flat spatial metrics must be \emph{joined} together in order to obtain null Ricci curvature and because of its relation to the momentum density $j$ (\cf \cref{Prop:CharJEq}), we give it the following name.
\begin{definition} \label{Def:JEq}
	A family $g_s$, $s \in I$, of Riemannian metrics is said to satisfy the \emph{j-equation} if we have $\div^{g_s}(\dot{g}_s) - \upd \trace^{g_s}(\dot{g}_s) = 0$. 
\end{definition}

We now want to understand equation~\eqref{eq.level.ricci} more closely.
Following \cite[Chapter 4]{besse:87}, the space of symmetic $2$-tensors on a a closed Ricci-flat manifold $(Q,g)$ admits the decomposition 
\begin{align}\label{eq:2tensor_splitting}
C^{\infty}(S^2Q)=C^{\infty}(Q)\cdot g\oplus \left\{\mathcal{L}_{X}g\mid X\in C^{\infty}(TQ)\right\}\oplus C^{\infty}(TT),
\end{align}
where
\begin{align*}
C^{\infty}(TT):=\left\{h\in C^{\infty}(S^2Q)\mid \trace^{g}h=0,\mathrm{div}^{g}h=0\right\}
\end{align*}
denotes the space of transverse traceless tensors, or TT-tensors for short. Using the $L^2$-orthogonal splitting
\begin{align*}
C^{\infty}(TQ):=\nabla (C^{\infty}(Q))\oplus W_{g},
\end{align*}
with 
\begin{align*}
W_{g}=\left\{ X\in C^{\infty}(TQ) \mid\mathrm{div}^{g}X=0\right\},
\end{align*}
we can refine \eqref{eq:2tensor_splitting} to
\begin{align}\label{eq:ref_2tensor_splitting}
C^{\infty}(S^2Q)=C^{\infty}(Q)\cdot g\oplus 
\nabla^2(C^{\infty}(Q))
\oplus
\left\{\mathcal{L}_{X}g\mid X\in W_{g}\right\}\oplus C^{\infty}(TT).
\end{align}

\begin{proposition}\label{prop_splitting_constraint_solutions}
Let $(Q,g)$ be a connected closed Ricci-flat manifold and $h\in C^{\infty}(S^2Q)$ be a solution of the equation
\begin{align}\label{eq:constraint_eq_analyzed}
\div^{g}h-\upd\,\trace^{g} h=0.
\end{align}
Then there exists a function $f$, a constant $c$ and a $TT$-tensor $\sigma$ such that
\begin{align*}
h=c g+\nabla^2f+\sigma.
\end{align*}
All these components are $L^2$-orthogonal to each other.
The components $c\cdot g$ and $\sigma$ are unique, and $f$ is unique up to an additive constant.
\end{proposition}
\begin{proof}In the following we assume $\dim Q=n-1$. All geometric operators will be with respect to~$g$, and we suppress~$g$ in the notation.
According to \eqref{eq:ref_2tensor_splitting}, $h\in C^{\infty}(S^2Q)$ can be uniquely written as
\begin{align}\label{eq:decomp_symm_tensor}
h=u\cdot g+\nabla^2f+\mathcal{L}_{X}g+\sigma,
\end{align}
where $u,f\in C^{\infty}(Q)$, $X\in W_g$ and $\sigma \in C^{\infty}(TT)$. This completely determines $u$ and $\sigma$. The vector field $X$ is unique up to a Killing vector field. The function $f$ is unique up to a function $v$ with $\nabla^2v=0$. Taking the trace, we see that $v$ is harmonic. Therefore, it is constant, since $Q$ is closed.

Applying \eqref{eq:constraint_eq_analyzed} to the decomposition \eqref{eq:decomp_symm_tensor} yields
\begin{align*}
\div^{g}h-\upd\,\trace^{g} h
  &=(2-n)\upd u+\div^{g}(\nabla^2f)-\upd\,\trace^{g}(\nabla^2f)
  +\div^{g}\mathcal{L}_{X}g-\upd\,\trace^{g}\mathcal{L}_{X}g\\
  &=(2-n)\upd u-\nabla^*\nabla\upd f+\upd\Delta f  +\div^{g}(\mathcal{L}_{X}g)
-2  \upd(\div^gX)
  \\
  &=(2-n)\upd u-\Delta_H(X^{\flat}).
\end{align*}

Here $\Delta_H$ denotes the Hodge-Laplacian on $1$-forms and $X^{\flat}$ is the metric dual of $X$, i.e. $X^{\flat}=g(X,.)$.
In the last equation, we used that $X\in W_g$. Furthermore, we used some standard commutation formulas and the Ricci-flatness of $g$.

Now, observe that $\Delta_HX^{\flat}$ is divergence-free because $X^{\flat}$ is. Thus, it is orthogonal to $\upd u$. Now if  $\div^{g}h-\upd\,\trace^{g} h=0$, we get $\upd u=0$ and $\Delta_HX^{\flat}=0$ separately. These equations imply that $u=c\in\R$ and that $X$ is a Killing vector field. This completes the proof.
\end{proof}
\begin{remark}\label{rem:splitting_constraint_equations}
If $h\in C^{\infty}(Q)$ satisfies \eqref{eq:constraint_eq_analyzed} and $\frac{\upd}{\upd s}_{|s=0}\ric_{g+sh}=0$, then we additionally have that the TT-part $\sigma$ of the decomposition in \cref{prop_splitting_constraint_solutions} satisfies $\Delta_L\sigma=0$, where $\Delta_L=\nabla^*\nabla-2\mathring{R}$ is the Lichnerowicz Laplacian, c.f. \cite[Chapter 12D]{besse:87} for details.
\end{remark}

\subsection{Additional curvature vanishing} \label{subsec:CurvVan}
Using the observations from \cref{sec.compute.ricci}, we can answer a question posed by Chai and Wan. 
The results in this subsection will not be needed in the rest of the paper.

Suppose that $(\overline{M},\overline{g})$ is a pp-wave spacetime with null Ricci curvature $\ric^{\overline{g}} = \rc V^\flat \otimes V^\flat$, where $V$ is a lightlike parallel vector field and $\rc$ is some function.
Consider a spacelike hypersurface $M \subset \overline{M}$ equipped with the unit normal $e_0$ such that $\overline{g}(e_0,V) < 0$.
Then we can write $V_{|M} = u(e_0 + \nu)$ for some positive $u \in C^\infty(M)$ and some unit vector field  $\nu \in \Gamma(TM)$.
Given $p \in M$ and $Z \in \nu^\perp \cap T_pM = V^\perp \cap T_pM$, we have $\ric^{\overline{g}}(\nu,Z) = \rc \overline{g}(V,\nu) \overline{g}(V,Z) = 0$.
Notice that due to $\nabla^{\overline{g}} V = 0$, we have
\begin{equation} \label{eq:TracedCurv}
	\begin{aligned}
		0= \ric^{\overline{g}}(\nu,Z) &= -\frac{1}{2}R^{\overline{g}}(\nu, e_0 + \nu, e_0 - \nu, Z) - \frac{1}{2}R^{\overline{g}}(\nu, e_0 - \nu, e_0 + \nu, Z) + \sum_{i=1}^{n-1} R^{\overline{g}}(\nu,e_i,e_i,Z) \\
		&= \sum_{i=1}^{n-1} R^{\overline{g}}(\nu,e_i,e_i,Z),
	\end{aligned}
\end{equation}
where $(e_1, \ldots, e_n, \nu)$ is an orthonormal basis of $T_pM$.
This would of course be the case if
\begin{gather*}
	R^{\overline{g}}(\nu,X,Y,Z) = 0
\end{gather*}
for all $X,Y,Z \in \nu^{\perp} \cap T_pM$.
In \cite[Rem.~4.2]{Chai.Wan:2024p}, Chai and Wan wonder whether all these components of the ambient Riemann curvature tensor vanish in the context of the initial data rigidity theorem of Eichmair, Galloway and Mendes \cite{Eichmair.Galloway.Mendes:2021}.
It was recently shown by the second-named author \cite{gloeckle:2025p} that the initial data sets considered there embed into pp-wave spacetimes.
Actually, they are contained in simple pp-wave spacetimes modeled on a torus $T^{n-1}$ where the family of metrics $g_s$ on $T^{n-1}$ consists of flat metrics.
With the help of the following proposition, we can thus give an affirmative answer to the question by Chai and Wan.

\begin{proposition} \label{Prop:FlatCurvVan}
	Consider a simple pp-wave spacetime modeled on a closed manifold $Q$
	\begin{gather*}
	(\overline{M},\overline{g}) \coloneqq (\R \times I \times Q, \upd v \otimes \upd s + \upd s \otimes \upd v + u^{-2}\upd s^2 + g_s)
	\end{gather*}
	with null Ricci curvature.
	Assume that $g_s$ is flat for all $s \in I$.
	Then
	\begin{gather*}
	R^{\overline{g}}(\bullette,X,Y,Z) = 0
	\end{gather*}
	for all $X,Y,Z \in (\frac{\del}{\del v}_{|p})^\perp$, $p \in \overline{M}$.
\end{proposition}
\begin{proof}
	Let $p = (v,s,x) \in \overline{M}$.
	In the following, we will identify $Q$ with $\{v\} \times \{s\} \times Q$.
	First of all, we show that the the (vector-valued) second fundamental form $\vec{\II}$ of $Q$ in $\overline{M}$ is given by
	\begin{gather*}
	\vec{\II} = -\frac{1}{2}\dot{g}_s \frac{\del}{\del v}.
	\end{gather*}
	This directly follows from the observation that the normal bundle of $Q$ in $\overline{M}$ is spanned by $\frac{\del}{\del s}$ and $\frac{\del}{\del v}$ and the calculations
	\begin{align*}
	\overline{g}\left(\vec{\II}(X,Y), \frac{\del}{\del v}\right)
	&= -\frac{1}{2} \Lie_{\frac{\del}{\del v}} \overline{g}(X,Y) = 0 \\
	\overline{g}\left(\vec{\II}(X,Y), \frac{\del}{\del s}\right)
	&= -\frac{1}{2} \Lie_{\frac{\del}{\del s}} \overline{g}(X,Y) = -\frac{1}{2}\dot{g}_s(X,Y)
	\end{align*}
	for $X,Y \in TQ$.
	
	Now it is easy to calculate the relevant components of the curvature tensor.
	Notice that it suffices to consider $X,Y,Z \in T_xQ$ since $\left(\frac{\del}{\del v}_{|p} \right)^\perp = \R \cdot \frac{\del}{\del v}_{|p} \oplus T_xQ$ and $\frac{\del}{\del v}$ is parallel.
	Because $\frac{\del}{\del v}$ is lightlike and $g_s$ is flat, the Gauß formula shows that
	\begin{gather*}
	R^{\overline{g}}(X,Y,Z,W) = R^{g_s}(X,Y,Z,W) - \overline{g}(\vec{\II}(X,W),\vec{\II}(Y,Z)) + \overline{g}(\vec{\II}(X,Z),\vec{\II}(Y,W)) = 0.
	\end{gather*}
	for all $X,Y,Z,W \in T_xQ$.
	As $\frac{\del}{\del v}$ is parallel, the Codazzi formula allows us to calculate
	\begin{align*}
	R^{\overline{g}}(X,Y,Z,N) &= \overline{g}(\overline{\nabla}_X (\vec{\II}(Y,Z)) - \vec{\II}(\nabla^{g_s}_X Y, Z) - \vec{\II}(Y, \nabla^{g_s}_X Z), N) \\
	&\phantom{=}\;- \overline{g}(\overline{\nabla}_Y (\vec{\II}(X,Z)) - \vec{\II}(\nabla^{g_s}_Y X, Z) - \vec{\II}(X, \nabla^{g_s}_Y Z), N) \\
	&= -\frac{1}{2} \left(\nabla^{g_s}_X \dot{g}_s(Y,Z) - \nabla^{g_s}_Y \dot{g}_s(X,Z) \right) \overline{g}\left(\frac{\del}{\del v},N \right)
	\end{align*}
	for all $X,Y,Z \in T_xQ$ and $N \in T_xQ^\perp$.
	Thus the claim follows once we have shown that
	\begin{gather*}
	\nabla^{g_s}_X \dot{g}_s(Y,Z) - \nabla^{g_s}_Y \dot{g}_s(X,Z) = 0
	\end{gather*}
	for all $X,Y,Z \in T_xQ$.
	
	To do so, we recall from \cref{lem.lor.ricci-flat,prop_splitting_constraint_solutions} that we have a decomposition $\dot{g}_s = cg_s + \nabla^2 f + \sigma$, where $c \in \R$, $f \in C^\infty(Q)$ and $\sigma$ is a TT-tensor of $(Q,g_s)$.
	Moreover, according to \cref{rem:splitting_constraint_equations}, $\sigma$ is in the kernel of the Lichnerowicz Laplacian, which in this case is just the connection Laplacian since $g_s$ is flat.
	Thus partial integration over the compact manifold $Q$ shows that $\sigma$ is parallel with respect to the Levi-Civita connection of $(Q,g_s)$.
	Hence,
	\begin{align*}
	\nabla^{g_s}_X \dot{g}_s(Y,Z) - \nabla^{g_s}_Y \dot{g}_s(X,Z) &=
	\nabla^{g_s}_X(\nabla^{g_s} \upd f)(Y,Z) - \nabla^{g_s}_Y (\nabla^{g_s} \upd f)(X,Z) \\
	&= R^{g_s}(X,Y,\grad^{g_s}f,Z) = 0
	\end{align*}
	for all $X,Y,Z \in T_xQ$.
\end{proof}

\begin{corollary}
	Let $(\overline{M},\overline{g})$ be a simple pp-wave spacetime modeled on a closed manifold $Q$ with null Ricci curvature and denote its canonical lightlike parallel vector field by $V$.
	Assume that $\dim(\overline{M}) \leq 5$.
	Then
	\begin{gather*}
	R^{\overline{g}}(\bullette,X,Y,Z) = 0
	\end{gather*}
	for $X,Y,Z \in (V)^\perp$. 
\end{corollary}
\begin{proof}
	This follows directly from \cref{Prop:FlatCurvVan} since up to dimension $3$ every Ricci-flat metric is flat.
\end{proof}

In the proof of \cref{Prop:FlatCurvVan}, we crucially used flatness of all the $g_s$, $s \in I$.
As we shall prove in \cref{sec.classification}, there are many more (simple) pp-wave spacetimes with null Ricci curvature.
An example is the simple pp-wave spacetime
\begin{gather*}
	(\overline{M}, \overline{g}) \coloneqq (\R \times \R \times Q, \upd v \otimes \upd s + \upd s \otimes \upd v + u^{-2}\upd s^2 + g),
\end{gather*} 
where $(Q,g)$ is a closed Ricci-flat Riemannian manifold and $u \in C^\infty(\R \times Q)$ is positive.
The calculations in the proof of \cref{Prop:FlatCurvVan} also show that $R^{\overline{g}}(X,Y,Z,N) = 0$ for all $X,Y,Z \in T_p(\{v\} \times \{s\} \times Q)$ and $N \in T_p(\{v\} \times \{s\} \times Q)^\perp$, $p = (v,s,x) \in \overline{M}$.
So if $M$ is the spacelike hypersurface $\{0\} \times \R \times Q$, then $R^{\overline{g}}(\nu,X,Y,Z) = 0$ for all $X,Y,Z \in \nu^\perp \cap T_pM$, $p \in M$, where $\nu = u\frac{\del}{\del s}$ is the normalized projection of $\frac{\del}{\del v}$ on $TM$.
However, we will show below that unless $g$ is flat as in \cref{Prop:FlatCurvVan}, this curvature vanishing is not very stable and crucially depends on the choice of spacelike hypersurface:
It is possible to slightly perturb $M$ inside $\overline{M}$ such that $R^{\overline{g}}(\nu,X,Y,Z) = 0$ no longer holds true for all $X,Y,Z \in \nu^\perp \cap T_pM$.
Thus we have constructed a lot of examples where this strict form of curvature vanishing does not hold.

It is worth noting that this picture does also not change if we require $(\overline{M},\overline{g})$ to carry a lightlike parallel spinor for, by \cref{Thm:ExParSpinors}, the only additional condition is that $(Q,g)$ carries a parallel spinor.
Nonetheless, the existence of a lightlike parallel spinor implies that a certain spinoral curvature condition is satisfied, which is weaker than the discussed vanishing of $R^{\overline{g}}(\nu,X,Y,Z)$ but a priori stronger than vanishing of its trace~\eqref{eq:TracedCurv}, \cf \cref{Cor:SpinCurvVan}.

\begin{proposition}
	Let $M$ be a spacelike hypersurface within a spacetime $(\overline{M},\overline{g})$ with lightlike parallel vector field $V$.
	Suppose that $p \in M$, $Q \subset M$ is the leaf of the foliation defined by $V^\perp$ on $M$ and $g$ is the induced metric on $Q$.
	Assume that the curvature of $(Q,g)$ does not vanish in $p$.
	Then, given any open neighborhood $U$ of $p$, any norm $\|\bullette\|$ on $C^\infty_c(U)$ and any $\epsilon > 0$, there is a compactly supported function $f \in C^\infty(M)$ with $\supp(f) \subset U$ and $\|f\| \leq \epsilon$ such that
	\begin{gather*}
	\tilde{M} = \{\mathrm{Fl}^V(f(q),q) \mid q \in M\}
	\end{gather*}
	is a spacelike hypersurface of $(\overline{M}, \overline{g})$ through $p$ with the following property:
	There are $X,Y,Z \in \nu^\perp \cap T_p\tilde{M}$ such that
	\begin{gather*}
	R^{\overline{g}}(\nu,X,Y,Z) \neq 0,
	\end{gather*}
	where $\nu \in \Gamma(T\tilde{M})$ denotes the orthogonal projection of $V$ on $T\tilde{M}$ normalized to unit length.
\end{proposition}
\begin{proof}
	Let $0 \neq N \in T_pM$ such that $N \perp T_pQ$.
	Note for later that this implies $\overline{g}(N,V) \neq 0$.
	We can assume in the following that $R^{\overline{g}}(N,X,Y,Z) = 0$ for all $X,Y,Z \in T_pM$ for otherwise $f \equiv 0$, and thus $\tilde{M}=M$, does the job.
	
	Now for any smooth function $f \in C^\infty(M)$ with $f(p) = 0$, we consider hypersurface 
	\begin{gather*}
	\tilde{M} = \{\mathrm{Fl}^V(f(q),q) \mid q \in M\} \subset \overline{M},
	\end{gather*}
	where $\mathrm{Fl}^V$ denotes the flow of $V$.
	Clearly, $p \in \tilde{M}$ and $T_p\tilde{M} = \{X + \upd f(X) V_{|p} \mid X \in T_pM\}$.
	We let $\tilde{Q}$ be the leaf through $p$ of the foliation defined by $V^\perp$ on $\tilde{M}$.
	Note that if $X \in T_pQ$, then $X + \upd f(X) V_{|p} \in T_p\tilde{Q}$ and vice versa.
	Hence,
	\begin{gather*}
	\tilde{N} \coloneqq N + \upd f(N) V_{|p} - \overline{g}(N,V) \left(\grad^g(f) + |\grad^g(f)|_g^2 V_{|p} \right) \in T_p \tilde{M}
	\end{gather*}
	is perpendicular to $T_p\tilde{Q}$.
	It suffices to show that there are a suitable function $f$ and $X,Y,Z \in T_pQ$ such that
	\begin{gather*}
	R^{\overline{g}}(\tilde{N}, X + \upd f(X) V_{|p}, Y + \upd f(Y) V_{|p}, Z + \upd f(Z) V_{|p}) \neq 0.
	\end{gather*}
	
	Since $V$ is parallel and we assumed $R^{\overline{g}}(N,X,Y,Z) = 0$,
	\begin{align*}
	R^{\overline{g}}(\tilde{N}, X + \upd f(X) V_{|p}, Y + \upd f(Y) V_{|p}, Z + \upd f(Z) V_{|p}) &= R^{\overline{g}}(N - \overline{g}(N,V) \grad^g(f), X, Y, Z) \\
	&= -\overline{g}(N,V) R^{\overline{g}}(\grad^g(f),X,Y,Z)
	\end{align*}
	for any $X,Y,Z \in T_pM$.
	Applying the Gauß formula as in the proof of \cref{Prop:FlatCurvVan}, we obtain for any $X,Y,Z \in T_pM$ that
	\begin{gather*}
	R^{\overline{g}}(\grad^g(f),X,Y,Z) = R^g(\grad^g(f),X,Y,Z).
	\end{gather*}
	
	Now, since $g$ is not flat in $p$, there are $X,Y,Z,W \in T_pM$ such that $R^g(W,X,Y,Z) \neq 0$.
	We choose a function $F \in C^\infty(M)$ such that $\grad^g(F) = W$, which is always possible.
	Of course, $F$ can be chosen such that $F$ is compactly supported and $\supp(F) \subset U$.
	The function we are supposed to construct will be $f = t F$ for a small constant $t > 0$.
	If $t$ is small enough, then $\|f\| < \epsilon$ is satisfied.
	Moreover, the hypersurface $\tilde{M}$ will be spacelike if the $C^1$-norm of $f$ is small enough, which can be achieved by choosing $t$ sufficiently small.
	Finally, taking the equations above together, we obtain
	\begin{gather*}
	R^{\overline{g}}(\tilde{N}, X + \upd f(X) V_{|p}, Y + \upd f(Y) V_{|p}, Z + \upd f(Z) V_{|p}) = -\overline{g}(N,V) R^g(t W,X,Y,Z) \neq 0,
	\end{gather*}
	which was left to show.
\end{proof}


\section{Classification of simple pp-waves with null Ricci curvature}\label{sec.classification}
\subsection{The role of the leafwise volume}
We let again $Q$ be a closed $(n-1)$-dimensional manifold (admitting Ricci-flat metrics).
We consider metrics of the form \eqref{eq:AGK_metrics} and analyze the volume function $s \mapsto \vol^{g_s}(Q)$ or rather its $(n-1)$-st root.

\begin{proposition}\label{prop:ScaleODE}
	Given a smooth family $g_s$, $s \in I$, of Riemannian metrics on $Q$ and a smooth function $\rc \in C^\infty(I \times Q)$, let the functions $\Rho, \Sigma \in C^\infty(I)$ be defined by
	\begin{equation} \label{eq:DefRhoSigma}
		\begin{aligned}
			\Rho_s &\coloneqq \strokedint_Q \rc_s \dvol^{g_s} \\
			\Sigma_s &\coloneqq \strokedint_Q |\sigma_s|_{g_s}^2 \dvol^{g_s}.
		\end{aligned}
	\end{equation}
	Here, $\sigma_s$ denotes the TT-part of $\dot{g}_s$ (with respect to $g_s$).
	Also, we use the shorthand $\strokedint_Q \bullet \;\dvol^{g_s} \coloneqq \vol^{g_s}(Q)^{-1}\int_Q \bullet \;\dvol^{g_s}$ for the mean.
	These $\Rho$ and $\Sigma$ are invariant under pulling back $(g_s,\rho_s)$, $s \in I$, along a smooth family $\phi_s$, $s \in I$, of diffeomorphisms of $Q$ and under rescaling $g_s$, $s \in I$, with a positive function $c \in C^\infty(I)$.
	Suppose now that
	\begin{gather} \label{eq:AGK_metrics_3}
		(\overline{M},\overline{g}) = (\R \times I \times Q, \upd v \otimes \upd s + \upd s \otimes \upd v + u^{-2}\upd s + g_s)	
	\end{gather}
	has null Ricci curvature $\ric^{\overline{g}} = \rc_s \upd s^{2}$.
	Then $\lambda_s \coloneqq \sqrt[n-1]{\vol^{g_s}(Q)}$ is subject to the ODE
	\begin{gather} \label{eq:ScaleODE}
		\ddot{\lambda}_s = -\frac{1}{n-1} \left(\Rho_s + \frac{1}{4} \Sigma_s \right) \lambda_s,
	\end{gather}
	where $\Rho_s$ and $\Sigma_s$ are the quantities determined by $s \mapsto (g_s,\rho_s)$.
\end{proposition}
\begin{proof}
	The invariance of $\Rho_s = \strokedint_Q \rc_s \dvol^{g_s}$ is immediate, since the mean of a function is invariant under diffeomorphisms and independent of the scaling of the metric.
	Next, we show that $\Sigma_s$ stays unchanged under these operations.
	If $\tilde{g}_s = c_s \phi_s^* g_s$ for a positive function $c \in C^\infty(I)$ and a smooth family of diffeomorphisms $\phi_s$, $s \in I$, then $\dot{\tilde{g}}_s = c_s \phi_s^* \dot{g}_s + \dot{c}_s \phi_s^* g_s + c_s \phi_s^*(\Lie_{X_s} g_s)$ for $X_s \circ \phi_s = \dot{\phi}_s$.
	This shows if $\sigma_s$ is the TT-part of $\dot{g}_s$ with respect to $g_s$, then the TT-part of $\dot{\tilde{g}}_s$ with respect to $\tilde{g}_s$ is given by $\tilde{\sigma}_s \coloneqq c_s \phi_s^* \sigma_s$.
	Thus $|\tilde{\sigma}_s|_{\tilde{g}_s}^2 = \phi_s^* |\sigma_s|_{g_s}^2$ and so $\Sigma_s$ is the same for $g_s$ and $\tilde{g}_s$, $s \in I$.
	
	Now we derive the ODE.
	Recall that $\vol^{g_s}(Q) = \lambda_s^{n-1}$ and observe that we have
	\begin{gather*}
		 \frac{1}{2}\strokedint_Q \tr^{g_s}(\dot{g}_s) \dvol^{g_s} = \lambda_s^{-(n-1)} \frac{\upd}{\upd s} \vol^{g_s}(Q) = (n-1)\lambda_s^{-1}\dot{\lambda}_s.
	\end{gather*}
	Similarly, we obtain that on the one hand
	\begin{align*}
		\lambda_s^{-(n-1)}\frac{\upd^2}{\upd s^2} \vol^{g_s}(Q) &= (n-1) \lambda_s^{-1} \ddot{\lambda}_s +  (n-2)(n-1) \lambda_s^{-2} \dot{\lambda}_s^2,
	\intertext{while on the other hand}
		\lambda_s^{-(n-1)}\frac{\upd^2}{\upd s^2} \vol^{g_s}(Q) &= \lambda_s^{-(n-1)}\frac{\upd}{\upd s} \int_Q \frac{1}{2} \tr^{g_s}(\dot{g}_s) \dvol^{g_s} \\
			&= \lambda_s^{-(n-1)}\int_Q \left(\frac{1}{2} \frac{\upd}{\upd s} \tr^{g_s}(\dot{g}_s) + \frac{1}{4} \tr^{g_s}(\dot{g}_s)^2 \right) \dvol^{g_s} \\
			&= \strokedint_Q \left(-\rc_s + \frac{1}{4} \tr^{g_s}(\dot{g}_s)^2 - \frac{1}{4} |\dot{g}_s|_{g_s}^2  \right) \dvol^{g_s}\\
			&= - \strokedint_Q \rc_s \dvol^{g_s} -\frac{1}{4} \strokedint_Q (|\dot{g}_s|_{g_s}^2 - \tr^{g_s}(\dot{g}_s)^2) \dvol^{g_s}, 
	\end{align*}
	where we used \eqref{eq:CalcRc} in the third step.
	Adding
	\begin{align*}
	0 &= \frac{n-2}{n-1} \left((n-1)\lambda_s^{-1}\dot{\lambda}_s -\frac{1}{2}\strokedint_Q \tr^{g_s}(\dot{g}_s) \dvol^{g_s} \right)\left((n-1)\lambda_s^{-1}\dot{\lambda}_s + \frac{1}{2}\strokedint_Q \tr^{g_s}(\dot{g}_s) \dvol^{g_s}\right) \\
	&= - \frac{1}{4} \cdot \frac{n-2}{n-1} \left(\strokedint_Q \tr^{g_s}(\dot{g}_s) \dvol^{g_s}\right)^2 + (n-2)(n-1)\lambda_s^{-2}\dot{\lambda}_s^2,
	\end{align*}
	we just have to show that 
	\begin{gather*}
	\Sigma_s = \strokedint_Q (|\dot{g}_s|_{g_s}^2 - \tr^{g_s}(\dot{g}_s)^2) \dvol^{g_s} + \frac{n-2}{n-1} \left(\strokedint_Q \tr^{g_s}(\dot{g}_s) \dvol^{g_s}\right)^2
	\end{gather*}
	in order to obtain
	\begin{gather*}
		(n-1)\lambda_s^{-1} \ddot{\lambda}_s = -\Rho_s -\frac{1}{4} \Sigma_s,
	\end{gather*}
	which is equivaltent to the ODE~\eqref{eq:ScaleODE}.
	
	To do so, we recall from \cref{lem.lor.ricci-flat} that $g_s$, $s \in I$, has to be a family of Ricci-flat metrics satisfying the j-equation~\eqref{eq.level.ricci}.
	Thus $\dot{g}_s$ decomposes as $\dot{g}_s = c_s g_s + \nabla^2 f_s + \sigma_s$ according to \cref{prop_splitting_constraint_solutions}.
	Since this decomposition is $L^2$-orthogonal, we have
	\begin{align*}
	\strokedint_Q |\dot{g}_s|_{g_s}^2 \dvol^{g_s} &= (n-1)c_s^2 + \strokedint_Q |\nabla^2 f_s|_{g_s}^2 \dvol^{g_s} + \strokedint_Q |\sigma_s|_{g_s}^2 \dvol^{g_s}.
	\intertext{Furthermore, we have}
	\strokedint_Q \tr^{g_s}(\dot{g}_s)^2 \dvol^{g_s}
	&= \strokedint_Q ((n-1)c_s - \Delta f)^2 \dvol^{g_s} \\
	&= (n-1)^2 c_s^2 + \strokedint_Q (\Delta f)^2 \dvol^{g_s}
	\intertext{and}
	\left(\strokedint_Q \tr^{g_s}(\dot{g}_s) \dvol^{g_s} \right)^2 &= (n-1)^2 c_s^2,
	\end{align*}
	where we used in both cases that $\Delta f$ integrates to zero over the closed manifold $Q$.
	Now it is easy to see that the $c_s$-dependent terms cancel.
	Since moreover
	\begin{align*}
	\strokedint_Q |\nabla^2 f_s|_{g_s}^2 \dvol^{g_s} &= \strokedint_Q g_s(\nabla f_s, \Delta \nabla f_s) \dvol^{g_s} \\
	&= \strokedint_Q g_s(\nabla f_s, \nabla \Delta f_s) \dvol^{g_s} \\
	&= \strokedint_Q (\Delta f_s)^2 \dvol^{g_s}
	\end{align*}
	holds -- where we used $\ric^{g_s}(\nabla f_s, \nabla f_s) = 0$ in the second step -- also the terms containing $f_s$ cancel, so that we are just left with $\strokedint_Q |\sigma_s|_{g_s}^2 \dvol^{g_s}$, as desired.
\end{proof}

Notice that $|\lambda_s|$ has a natural interpretation as length scaling factor.
If $g_s$ was obtained by scaling the lengths of a unit-volume metric by the factor of $|\lambda_s|$, \ie the unit-volume metric was $\lambda_s^{-2} g_s$, then indeed its volume is given by $|\lambda_s|^{n-1} = \vol^{g_s}(Q)$.
Although in the end we are only interested in the positive solutions of~\eqref{eq:ScaleODE}, which are the ones leading to non-degenerate metrics, the possibility of the solution of~\eqref{eq:ScaleODE} turning zero simplifies the analysis.

\begin{lemma} \label{lem:ODEComparison}
	All local solutions to the ODE~\eqref{eq:ScaleODE} can be extended to global solutions and these global solutions form $2$-dimensional vector space.
	Let $\lambda$ be a non-zero solution.
	Then $\lambda$ has a discrete set of simple zeros $D \subset I$.
	Moreover, if there is some $C > 0$ such that 
	\begin{align*}
	\frac{1}{n-1}\left(\Rho_s + \frac{1}{4} \Sigma_s\right) \leq C
	\end{align*}
	for all $s \in I$, then the distance of any elements in $D$ of $\lambda$ is at least $\frac{\pi}{\sqrt{C}}$.
	In particular, if the upper bound $C$ can be chosen to be zero, then $D$ has at most one element.
	If there is some $c > 0$ such that 
	\begin{align*}
	c \leq \frac{1}{n-1}\left(\Rho_s + \frac{1}{4} \Sigma_s\right)
	\end{align*}
	for all $s \in I$, then each connected component of $I \setminus D$ has at most length $\frac{\pi}{\sqrt{c}}$.
	Furthermore, if the lower bound $c$ can be chosen to be zero and $I=\R$, then $D$ has at least one element unless $\frac{1}{n-1}\left(\Rho_s + \frac{1}{4} \Sigma_s\right) \equiv 0$ and $\lambda$ is a constant solution.
\end{lemma}
\begin{proof}
	The first statement is a general fact about linear ODEs (of second order).
	If $\lambda$ is a solution and $s \in I$ is such that $\lambda_s = 0$ and $\dot{\lambda}_s \neq 0$, then $\lambda \equiv 0$ since solutions are uniquely determined by a pair of initial values.
	This shows that the zeros of a non-zero solution $\lambda$ are simple and in particular form a discrete set.
	
	Now assume an upper bound $C > 0$ and let $s \in D$.
	Without loss of generality, we assume $s=0$.
	We just need to consider the case $\dot{\lambda}_0 > 0$ and show that there is no zero of $\lambda$ on $(0,\frac{\pi}{C}) \cap I$.
	The other cases follow by multiplying $\lambda$ with $-1$ and/or precomposition of $\lambda$, $\Rho$ and $\Sigma$ with $s \mapsto -s$.
	We denote by $(r, \phi) \colon \R^2 \setminus (\R_{\leq 0} \times \{0\}) \to \R_{>0} \times (-\pi,\pi)$ polar coordinates.
	Then we consider the function $\alpha \colon s \mapsto \phi(\lambda_s,-\frac{1}{\sqrt{C}}\dot{\lambda}_s) = \phi(\sqrt{C}\lambda_s,-\dot{\lambda}_s)$.
	We have $\alpha_0 = -\frac{\pi}{2}$ and, since $\upd_{(x,y)} \phi = \frac{1}{x^2+y^2}(-y,x)$,
	\begin{gather*}
	\dot{\alpha}_s = \frac{\sqrt{C}\dot{\lambda}_s^2 - \sqrt{C}\lambda_s\ddot{\lambda}_s}{C\lambda_s^2 + \dot{\lambda}_s^2}
	= \sqrt{C} \; \frac{\dot{\lambda}_s^2 + \frac{1}{n-1}\left(\Rho_s + \frac{1}{4} \Sigma_s \right)\lambda_s^2}{C\lambda_s^2 + \dot{\lambda}_s^2}
	\leq \sqrt{C}.
	\end{gather*}
	Because $\alpha$ takes the value $\frac{\pi}{2}$ at the next zero of $\lambda$, it follows that this cannot happen before $\frac{\pi}{\sqrt{C}}$.
	The “in particular” part follows by sending $C \longrightarrow 0$.
	
	If we have a lower bound $c$, we consider a connected component $J$ of $I \setminus D$ and assume without loss of generality that $\lambda$ is positive on it.
	Then we can run the same construction as above on $J$ with $c$ instead of $C$ and obtain the inequality $\dot{\alpha}_s \geq \sqrt{c}$.
	Since $\lambda$ has no zero on $J$, we must have $|\alpha_s| < \frac{\pi}{2}$ for all $s \in J$ and thus $J$ can  at most have length $\frac{\pi}{\sqrt{c}}$.
	
	Lastly, we assume $c = 0$ and $I = \R$.
	We suppose that $\lambda$ has no zeros; without loss of generality $\lambda > 0$.
	Let us assume for contradiction that $\dot{\lambda}_s \neq 0$ somewhere.
	After shifting and possibly reflecting the $\R$-coordinate, we may assume that this is the case in $s=0$ and that $\dot{\lambda}_0 < 0$.
	Since $\lambda > 0$, we have $\ddot{\lambda} \leq 0$ and thus $\dot{\lambda}_s \leq \dot{\lambda}_0 < 0$ for all $s \geq 0$.
	Then $\lambda$ has to hit zero at $s = \frac{\lambda_0}{\dot{\lambda}_0}$ at latest, which is a contradiction.
	So $\lambda$ has to be constant, which implies $\frac{1}{n-1}\left(\Rho_s + \frac{1}{4} \Sigma_s\right) \equiv 0$.
\end{proof}

We now start proving one of the main results by constructing Lorentzian manifolds with prescribed null Ricci curvature from curves of Ricci-flat Riemannian metrics.
\begin{lemma} \label{Lem:MakeTT}
	Suppose that $g_s$, $s \in I$, is a smooth family of unit-volume Ricci-flat metrics on a closed manifold $Q$.
	Then there is a smooth family $\phi_s$, $s \in I$, of diffeomorphisms of $Q$, such that $\dot{\tilde{g}}_s = \frac{\upd}{\upd s} \tilde{g}_s$ is a TT-tensor with respect to $\tilde{g}_s \coloneqq \varphi_s^*g_s$ for all $s \in I$.
	In particular, $\tilde{g}_s$ satisfies the j-equation \eqref{eq.level.ricci}.
\end{lemma}
\begin{proof}
	According to \cite[Lem.~31]{ammann.kroencke.mueller:21}, there is a smooth family of diffeomorphisms $\phi_s$, $s \in I$, such that $\tilde{g}_s \coloneqq \varphi_s^*g_s$ is divergence-free in the sense $\div^{\tilde{g}_s}(\dot{\tilde{g}}) = 0$.
	Since all the $g_s$ and thus the $\tilde{g}_s$ have unit volume, we have $\int_Q \tr^{\tilde{g}}(\dot{\tilde{g}}_s) \dvol^{\tilde{g}_s} = 0$.
	In addition, $g_s$, $s \in I$, and thus $\tilde{g}_s$, $s \in I$, are families of Ricci-flat metrics, making $\tilde{g}_s$, $s \in I$, an infinitesimal Einstein deformation in the sense of \cite[Def.~12.19]{besse:87}.
	A theorem due to Berger and Ebin \cite[Thm.~12.30]{besse:87} then allows us to conclude that $\tr^{\tilde{g}_s}(\dot{\tilde{g}}_s) = 0$.
	So $\dot{\tilde{g}}_s$ is a TT-tensor with respect to $\tilde{g}_s$.
\end{proof}

\begin{proposition} \label{Prop:SolveLaplace}
	Suppose that $g_s$, $s \in I$, is a smooth family of unit-volume Ricci-flat metrics on a closed manifold $Q$ such that $\dot{g}_s$ is a TT-tensor with respect to $g_s$ for all $s \in I$ and let $\rc \in C^\infty(I \times Q)$ be a smooth function.
	Let $\lambda \in C^\infty(I)$ be a non-zero solution to the ODE~\eqref{eq:ScaleODE} determined by the parametrized curve $s \mapsto (g_s,\rc_s)$ and $D \coloneqq \lambda^{-1}(\{0\}) \subset I$ its discrete set of zeros.
	Then there is a positive function $u \in C^\infty(I \times Q)$ such that the Lorentzian manifold
	\begin{gather*}
		(\overline{M},\overline{g}) \coloneqq (\R \times (I \setminus D) \times Q, \upd v \otimes \upd s + \upd s \otimes \upd v + u^{-2}\upd s^2 + \lambda_s^2 g_s)
	\end{gather*}
	has $\frac{\del}{\del v}$ as complete lightlike parallel vector field and null Ricci curvature
	\begin{gather*}
		\ric^{\overline{g}} = \rc_s {\textstyle \frac{\del}{\del v}}^\flat \otimes {\textstyle \frac{\del}{\del v}}^\flat = \rc_s \upd s^2.
	\end{gather*}
\end{proposition}

\begin{proof}	
	We consider $\tilde{g}_s = \lambda_s^2g_s$, where $\lambda$ is the non-zero solution to the ODE~\eqref{eq:ScaleODE}, whose set of zeros $D$ is discrete according to \cref{lem:ODEComparison}.
	For all $s \in I \setminus D$, we notice that $\dot{\tilde{g}}_s = 2 \dot{\lambda}_s\lambda_s^{-1} \tilde{g}_s + \lambda_s^2\dot{g}_s$ and that $\dot{g}_s$ is a TT-tensor with respect to $\tilde{g}_s$.
	In view of~\eqref{eq:CalcRc}, we solve
	\begin{align*}
		\frac{1}{2}\Delta^{g_s}(w_s) &= \rc_s + \frac{1}{2}\frac{\upd}{\upd s} \tr^{\tilde{g}_s}(\dot{\tilde{g}}_s) + \frac{1}{4} |\dot{\tilde{g}}_s|_{\tilde{g}_s}^2 \\
			&= \rc_s + (n-1)\ddot{\lambda}_s\lambda_s^{-1} - (n-1)\dot{\lambda}_s^2\lambda_s^{-2} + (n-1)\dot{\lambda}_s^2\lambda_s^{-2} + \frac{1}{4}|\dot{g}_s|_{g_s}^2 \\
			&= \rc_s - \left(\Rho_s + \frac{1}{4} \Sigma_s\right) + \frac{1}{4}|\dot{g}_s|_{g_s}^2.
	\end{align*}
	Notice that the both the left hand side and the last line on the right hand side are independent of $\lambda$ and in particular well-defined for all $s \in I$.
	As the right hand side integrates to zero due to~\eqref{eq:DefRhoSigma}, the family of Laplace equations can be solved, yielding a smooth family of solutions $w_s$, $s \in I$.
	After potentially adding a smooth family of constants, we may assume that all the $w_s$ are positive.
	Then we set $u \colon I \times Q \to \R,\, (s,x) \mapsto (w_s(x))^{-\frac{1}{2}}$.
	On $\R \times (I \setminus D) \times Q$, the symmetric bilinear form $\LorMet \coloneqq \upd v \otimes \upd s + \upd s \otimes \upd v + u^{-2}\upd s^2 + \tilde{g}_s$ defines a Lorentzian metric for which $\frac{\del}{\del v}$ is a complete lightlike parallel vector field.
	According to \eqref{eq:CalcRc}, the Lorentzian metric~$\overline{g}$ has the desired Ricci curvature.
\end{proof}
	
\begin{proof}[Proof of \cref{Thm:Ex}]
	Recall that by \cref{lem.lor.ricci-flat}, the Lorentzian metric~\eqref{eq:ppMetric} has null Ricci curvature if and only if $\lambda^2 \tilde{g}_s$, $s \in I$, or equivalently $\tilde{g}_s$, $s \in I$, is a smooth family of Ricci-flat metrics subject to the j-equation~\eqref{eq.level.ricci}.
	Now, noting that by \cref{lem:ODEComparison} the ODE~\eqref{eq:ScaleODE} always has a two-dimensional vector space of solutions and thus in particular a non-zero solution, \cref{Thm:Ex} is an immediate consequence from \cref{Lem:MakeTT,Prop:SolveLaplace}.
\end{proof}	

For future reference, we also summarize these results in the following proposition:

\begin{proposition} \label{Prop:SurjModSp}
	Suppose that $g_s$, $s \in I$, is a smooth family of unit-volume Ricci-flat metrics on a closed manifold $Q$ and $\rc \in C^\infty(I \times Q)$ is a smooth function.
	Let $\lambda \in C^\infty(I)$ be a non-zero solution to the ODE~\eqref{eq:ScaleODE} determined by the parametrized curve $s \mapsto (g_s,\rc_s)$ and $D \coloneqq \lambda^{-1}(\{0\}) \subset I$ its discrete set of zeros.
	Then there is a smooth family $\phi_s$, $s \in I$, of diffeomorphisms of $Q$, and a positive function $u \in C^\infty(I \times Q)$ such that the Lorentzian manifold
	\begin{gather*}
		(\overline{M},\overline{g}) \coloneqq (\R \times (I \setminus D) \times Q, \upd v \otimes \upd s + \upd s \otimes \upd v + u^{-2}\upd s^2 + \lambda_s^2\phi_s^* g_s)
	\end{gather*}
	has $\frac{\del}{\del v}$ as complete lightlike parallel vector field and null Ricci curvature
	\begin{gather*}
		\ric^{\overline{g}} = (\phi_s^*\tilde{\rc}_s) {\textstyle \frac{\del}{\del v}}^\flat \otimes {\textstyle \frac{\del}{\del v}}^\flat = (\phi_s^*\tilde{\rc}_s) \upd s^2.
	\end{gather*} 
	Moreover, after possibly restricting the interval $I$, the choices of $\phi_s$, $s \in I$, and $u$ can be made such that $u \equiv 1$.
\end{proposition}
\begin{proof}
	According to \cref{Lem:MakeTT}, we may choose a smooth family $\phi_s$, $s \in I$, of diffeomorphisms of $Q$ such that $\dot{\tilde{g}}_s$ is a TT-tensor with respect to $\tilde{g}_s \coloneqq \phi_s^* g_s$ for all $s \in I$.
	Since by \cref{prop:ScaleODE}, $\lambda$ is also a non-zero solution of the ODE~\eqref{eq:ScaleODE} associated to $s \mapsto (\tilde{g}_s, \tilde{\rc}_s)$ for $\tilde{\rc}_s \coloneqq \phi_s^* \rc_s$, \cref{Prop:SolveLaplace} applies and yields the main part of the result.
	
	Moreover, choosing a spacelike hypersurface within a leaf $\mathcal{L}$ of the distribution defined on $(\overline{M},\overline{g})$ by $\frac{\del}{\del v}^\flat$ and invoking \cref{Prop:GeodNormCoord}, we find a way to express the same Lorentzian metric in different coordinates where $u \equiv 1$.
	Note that the family of metrics on $Q$ appearing in \cref{Prop:GeodNormCoord} is related to $\tilde{g}_s$, $s \in I$, through pulling back with a family of $s$-dependent diffeomorphisms.	
\end{proof}	

\subsection{Proving the classification}
We next discuss how to build a one-to-one correspondence.
Together with \cref{Prop:SurjModSp}, the following proposition allows us to produce a well-defined map assigning to a smooth family $g_s$, $s \in I$, of unit-volume Ricci flat metrics on $Q$, a smooth function $\rc \in \C^\infty(I \times Q)$ and a solution $\lambda$ of the ODE~\eqref{eq:ScaleODE} an isometry class of Lorentzian metrics.

\begin{proposition} \label{Prop:WellDefModSp}
	Suppose that $g_s$ and $\tilde{g}_s$, $s \in I$, are smooth families of unit-volume Ricci-flat metrics on a closed manifold $Q$, which are both subject to the j-equation.
	Assume further that they are related as $\tilde{g}_s = \psi_s^* g_s$ via a smooth family of diffeomorphisms $\psi_s$, $s \in I$, on $Q$.
	Let $\rc \in C^\infty(I \times Q)$ and define $\tilde{\rc}$ through $\tilde{\rc}_s \coloneqq \psi_s^* \rc_s$.
	Let furthermore $\lambda$ be a positive solution of the ODE~\eqref{eq:ScaleODE} associated to $(g_s,\rc_s)$, which is by \cref{prop:ScaleODE} also the ODE associated to $(\tilde{g}_s,\tilde{\rc}_s)$.
	Then choose smooth families of diffeomorphisms $\phi_s, \tilde{\phi}_s$, $s \in I$, and positive functions $u,\tilde{u} \in C^\infty(I \times Q)$ such that the Lorentzian metrics on $\R \times I \times Q$ defined by 
	\begin{align*}
		\overline{g} &\coloneqq \upd v \otimes \upd s + \upd s \otimes \upd v + u^{-2} \upd s^2 + \lambda_s \phi_s^* g_s &
		&\text{and} &
		\tilde{\overline{g}} &\coloneqq \upd v \otimes \upd s + \upd s \otimes \upd v +  \tilde{u}^{-2} \upd s^2 + \lambda_s \tilde{\phi}_s^* \tilde{g}_s
	\end{align*}
	have Ricci curvature $\ric^{\overline{g}} = (\phi_s^*\rc_s) \upd s^2$ and $\ric^{\tilde{\overline{g}}} = (\tilde{\phi}_s^*\tilde{\rc}_s^*) \upd s^2$.
	Then $\overline{g}$ and $\tilde{\overline{g}}$ are isometric.
	More precisely, the isometry can be chosen to be of the kind considered in \cref{Lem:ChangeHypersurface}.
\end{proposition}
\begin{proof}
	First of all, we realize that it suffices to prove this statement in the case where both $\phi_s$ and $\tilde{\phi}_s$ are the identity for all $s \in I$.
	The general case is then obtained by replacing $(g_s, \rc_s)$ with $\phi_s^* (g_s,\rc_s)$, $(\tilde{g}_s, \tilde{\rc}_s)$ with $\tilde{\phi}_s^* (\tilde{g}_s, \tilde{\rc}_s)$ and $\psi_s$ with $\phi_s^{-1} \circ \psi_s \circ \tilde{\phi}_s$, since $\tilde{\phi}_s^* \tilde{g}_s = \tilde{\phi}_s^* \psi_s^* (\phi_s^{-1})^* \phi_s^* g_s = (\phi_s^{-1} \circ \psi_s \circ \tilde{\phi}_s)^* \phi_s^* g_s$.
	
	We consider the family of vector fields $X_s$, $s \in I$, defined by $X_s \circ \psi_s \coloneqq -\dot{\psi}_s$ and note that $\dot{\tilde{g}}_s = \psi_s^*(\dot{g}_s - \Lie_{X_s} g_s)$. 
	Hence the j-equation~\eqref{eq.level.ricci} for $\tilde{g}_s$ is equivalent to 
	\begin{gather*}
		\div^{g_s}(\dot{g}_s - \Lie_{X_s} g_s) - \upd \tr^{g_s}(\dot{g}_s - \Lie_{X_s} g_s) = 0.
	\end{gather*}
	Because $(g_s)$ is subject to the j-equation and $\Lie_{X_s} g_s$ is $L^2$-orthogonal to $g_s$ and the TT-tensors, we deduce from \cref{prop_splitting_constraint_solutions} that there is a family of functions $(f_s)$ such that $X_s = \grad^{g_s}(f_s)$, which is unique up to adding constants.
	
	Now we consider the diffeomorphism $\Psi \coloneqq \Phi \circ F \colon \R \times I \times Q \to \R \times I \times Q$ with
	\begin{align*}
		F \colon (v,s,q) &\longmapsto (v +f_s(q),s,q) & &\text{and}&
		\Phi \colon (v,s,q) &\longmapsto (v,s,\psi_s(q)).
	\end{align*}
	\Cref{Lem:ChangeHypersurface} shows that
	\begin{align*}
		\Psi^* \overline{g} = \upd v \otimes \upd s + \upd s \otimes \upd v + \psi_s^*(u^{-2} + 2\dot{f}_s -|X_s|^2) \upd s^{2} + \lambda_s^2 \tilde{g}_s.
	\end{align*}
	Moreover, we find that
	\begin{gather*}
		\ric^{(\Phi \circ F)^* \overline{g}} = (\Phi \circ F)^* \ric^{\overline{g}} = (\Phi \circ F)^* (\rc_s \upd s^2) = \tilde{\rc}_s \upd s^2.
	\end{gather*}
	By~\eqref{eq:CalcRc}, this amounts to
	\begin{gather*}
		\tilde{\rc}_s = \frac{1}{2} \Delta^{\tilde{g}_s} \psi_s^*(u_s^{-2} + 2\dot{f}_s -|X_s|^2) - \frac{1}{2} \frac{\upd}{\upd s} \tr^{\tilde{g}_s} \dot{\tilde{g}}_s - \frac{1}{4} |\dot{\tilde{g}}_s|_{\tilde{g}_s}^2,
	\end{gather*}
	while the Ricci curvature of $\tilde{\overline{g}}$ leads to the equation
	\begin{gather*}
		\tilde{\rc}_s = \frac{1}{2} \Delta^{\tilde{g}_s} \tilde{u}_s^{-2} - \frac{1}{2} \frac{\upd}{\upd s} \tr^{\tilde{g}_s} \dot{\tilde{g}}_s - \frac{1}{4} |\dot{\tilde{g}}_s|_{\tilde{g}_s}^2.
	\end{gather*}
	Subtracting these two equations, we find that
	\begin{gather*}
		\Delta^{\tilde{g}_s} \left( \psi_s^*(u_s^{-2} + 2\dot{f}_s -|X_s|^2) - \tilde{u}_s^{-2} \right) = 0.
	\end{gather*}
	So there is a function $c \in C^\infty(I)$ such that $ \psi_s^*(u_s^{-2} + 2\dot{f}_s -|X_s|^2) - \tilde{u}_s^{-2} = 2c_s$ for all $s \in I$.
	Considering $f_s - \int_a^s c_s \upd s$ for some $a \in I$ instead, we may assume without loss of generality that $f_s$ was chosen above such that $ \psi_s^*(u_s^{-2} + 2\dot{f}_s -|X_s|^2) = \tilde{u}_s^{-2}$ holds.
	But then $(\Phi \circ F)^* \overline{g} = \tilde{\overline{g}}$, so the metrics are isometric.
\end{proof}

Recall that the Lorentzian manifold~\eqref{eq:AGK_metrics_3} has a canonical complete lightlike parallel vector field $V = \frac{\del}{\del v} = \grad^{\overline{g}} s$.
The essential idea for the injectivity part of the one-to-one correspondence is that $g_s$, $s \in I$, $\rc$ and $\lambda$ can be reconstructed as geometrical data from a Lorentzian manifold with given lightlike parallel vector field.
Before doing so, we discuss what happens when there are several charts of the kind~\eqref{eq:AGK_metrics_3} whose canonical lightlike parallel vector fields point in different directions.
We do so in the next lemma.

\begin{lemma} \label{Lem:TwoLlParV}
	Let $(\overline{M},\overline{g})$ be a Lorentzian manifold which can be written as~\eqref{eq:AGK_metrics_3}  (with $Q$ not necessarily closed) in two different ways such that the respective canonical (complete) lightlike parallel vector fields $V$ and $\tilde{V}$ are linearly independent.
	Then $\overline{g}(V,\tilde{V}) \neq 0$ is constant and there is a Riemannian metric $g$ on $Q$ and a diffeomorphism $\Phi \colon \R \times \R \times Q \to \overline{M},\,(v,\tilde{v},x) \mapsto \Phi(v,\tilde{v},x)$ such that $V = \upd \Phi(\frac{\del}{\del v})$, $\tilde{V} = \upd \Phi(\frac{\del}{\del \tilde{v}})$ and
	\begin{align*}
		\Phi^* \overline{g} = \overline{g}(V,\tilde{V}) \upd v \otimes \upd \tilde{v} +\overline{g}(V,\tilde{V}) \upd \tilde{v} \otimes \upd v + g.
	\end{align*}
	In particular, there is an isometry of $(\overline{M},\overline{g})$ mapping $V$ to $\tilde{V}$ and $\tilde{V}$ to $V$.
\end{lemma}
\begin{proof}
	Let $\Psi \colon \R \times I \times Q \to \overline{M}$ and $\tilde{\Psi} \colon \R \times \tilde{I} \times Q \to \overline{M}$ be diffeomorphisms such that the metrics $\Psi^* \overline{g}$ and $\tilde{\Psi}^* \overline{g}$ are as in~\eqref{eq:AGK_metrics} and such that $V = \grad^{\overline{g}} s$ and $\tilde{V} = \grad^{\overline{g}} \tilde {s}$, where $s$ and $\tilde{s}$ are the compositions of $\Psi$ and $\tilde{\Psi}$ with the projection on the second coordinate, respectively.
	We choose $s_0 \in I$ and $\tilde{s}_0 \in \tilde{I}$ and show that $s^{-1}(s_0)$ and $\tilde{s}^{-1}(\tilde{s}_0)$ intersect in a manifold diffeomorphic to $Q$.
	
	To this aim, first note that $\overline{g}(V,\tilde{V})$ is constant since $V$ and $\tilde{V}$ are both parallel and that $\overline{g}(V,\tilde{V}) \neq 0$ as $V$ and $\tilde{V}$ are lightlike and linearly independent.
	Now	consider the restriction of $\tilde{s}$ to $s^{-1}(s_0)$.
	Since $\upd \tilde{s}(V) = \overline{g}(\tilde{V},V) \neq 0$ is constant, each flow line of $V$ has a unique point of intersection with $\tilde{s}^{-1}(\tilde{s}_0)$ (and $\tilde{I} = \R$).
	Moreover, this shows that $\Psi^{-1}(s^{-1}(s_0) \cap \tilde{s}^{-1}(\tilde{s}_0))$ is a graph over $\{0\} \times \{s_0\} \times Q$ within $\R \times \{s_0\} \times Q$, so the intersection is diffeomorphic to $Q$ in particular.
	
	We now identify $Q$ with $s^{-1}(s_0) \cap \tilde{s}^{-1}(\tilde{s}_0)$ and let $v \mapsto \Phi^{V}_v$ and $\tilde{v} \mapsto \Phi^{\tilde{V}}_{\tilde{v}}$ be the flows of $V$ and $\tilde{V}$, respectively.
	Notice that the flow of $V$ does not change the value of $s$ since $\upd s(V) = \overline{g}(V,V) = 0$ but changes the value of $\tilde{s}$ at a constant rate as seen before.
	Likewise, the flow of $\tilde{V}$ leaves $\tilde{s}$ invariant but changes $s$ at a constant rate. 
	Hence for all $(t, \tilde{t}) \in I \times \tilde{I} (= \R \times \R)$, there is a unique pair $(v, \tilde{v})$ such that $\Phi^{V}_v \circ \Phi^{\tilde{V}}_{\tilde{v}} \colon Q \to s^{-1}(t) \cap \tilde{s}^{-1}(\tilde{t})$ is a diffeomorphism.
	Thus we get a diffeomorphism
	\begin{align*}
		\Phi \colon \R \times \R \times Q &\longrightarrow \overline{M} \\
		(v,\tilde{v}, x) &\longmapsto \Phi^{V}_v \circ \Phi^{\tilde{V}}_{\tilde{v}}(x).
	\end{align*}
	It has the property that $\upd \Phi \left(\frac{\del}{\del v} \right) = V$.
	Since the flows commute due to $[V, \tilde{V}] = \nabla^{\overline{g}}_V \tilde{V} - \nabla^{\overline{g}}_{\tilde{V}} V = 0$, $\upd \Phi \left(\frac{\del}{\del \tilde{v}} \right) = \tilde{V}$ holds as well.
	
	As $V$ and $\tilde{V}$ are in particular Killing vector fields, the metric coefficients of $\Phi^* \overline{g}$ do not depend on $v$ and $\tilde{v}$.
	Since $Q$ is orthogonal to both $V$ and $\tilde{V}$ as it is contained in a level set of $s$ and $\tilde{s}$, we immediately obtain the claimed formula for the metric.
	Finally note that the diffeomorphism $\R \times \R \times Q \to \R \times \R \times Q$ swapping $v$ and $\tilde{v}$ leaves $\Phi^* \overline{g}$ invariant and thus induces an isometry on $(\overline{M},\overline{g})$ swapping $V$ and $\tilde{V}$.
\end{proof}

\begin{proposition} \label{Prop:InjModSp}
	Let $g_s$, $s \in I$, and $\tilde{g}_s$, $s \in \tilde{I}$, be smooth families of metrics on a closed, connected manifold $Q$ and $u \in C^{\infty}(I \times Q)$ and $\tilde{u} \in C^{\infty}(\tilde{I} \times Q)$ be positive functions.
	Consider the two Lorentzian metrics 
	\begin{align*}
		\overline{g} &\coloneqq \upd v \otimes \upd s + \upd s \otimes \upd v + u^{-2} \upd s^2 + g_s &
		&\text{and} &
		\tilde{\overline{g}} &\coloneqq \upd v \otimes \upd s + \upd s \otimes \upd v +  \tilde{u}^{-2} \upd s^2 + \tilde{g}_s
	\end{align*}
	on $\R \times I \times Q$ and $\R \times \tilde{I} \times Q$, respectively.
	Assume that there is an isometry between these Lorentzian manifolds that maps the canonical (complete) lightlike parallel vector fields to each other and fixes the $s$-coordinate of some point $p$.
	Then $I = \tilde{I}$ and there exists a smooth family of diffeomorphisms $\phi_s$, $s \in I$, on $Q$ such that $\tilde{g}_s = \phi_s^* g_s$.
	In particular, we have
	\begin{gather*}
		\sqrt[n-1]{\vol^{\phantom{\tilde{}}g_s}(Q)} = \sqrt[n-1]{\vol^{\overline{g}_s}(Q)}
	\end{gather*}
	for all $s \in I$.
	Moreover, if the Lorentzian metrics have null Ricci curvature, given by $\ric^{\overline{g}} = \rc_s \upd s^2$ and $\ric^{\tilde{\overline{g}}} = \tilde{\rc}_s^* \upd s^2$, respectively, then the family $\phi_s$, $s \in I$, can be chosen such that in addition $\tilde{\rc}_s = \phi_s^* \rc_s$ holds.
\end{proposition}
\begin{proof}
	Let us give the name $\Phi \colon \R \times I \times Q \to \R \times \tilde{I} \times Q$ to the isometry from the assumption.
	We first prove that $\Phi$ fixes the $s$-coordinate.
	By assumption, $\Phi$ maps the gradients of $s$ with respect to the respective metrics onto each other, so $\Phi_*(\grad^{\overline{g}}s) = \grad^{\tilde{\overline{g}}}s$.
	On the other hand, since $\Phi$ is an isometry, we have $\grad^{\tilde{\overline{g}}}s = \Phi_*(\grad^{\Phi^*\tilde{\overline{g}}} \Phi^*s) = \Phi_*(\grad^{\overline{g}}(s \circ \Phi))$.
	Thus $\upd s = \upd (s \circ \Phi)$.
	Since in addition $s(p) = s(\Phi(p))$ and $\R \times I \times Q$ is connected, we get $s = s \circ \Phi$ and in particular $I = \tilde{I}$.
	
	This, together with $\upd \Phi \left(\frac{\del}{\del v} \right) = \frac{\del}{\del v}$, implies that $\Phi(v,s,x) = (v+f(s,x),s,\phi_s(x))$ for some function $f \in C^\infty(I \times Q)$ and a smooth family of smooth maps $\phi_s \colon Q \to Q$, $s \in I$.
	Actually, all the $\phi_s$ have to be diffeomorphisms of $Q$ because $\Phi$ is a diffeomorphism.
	For each $s \in I$, the (degenerate) symmetric bilinear form induced by $\tilde{\overline{g}}$ on the hypersurface $\R \times \{s\} \times Q$ is given by $\pi^* \tilde{g}_s$, where $\pi$ denotes the canonical projection on $Q$.
	Letting $\iota \colon Q \to \R \times \{s\} \times Q,\, x \mapsto (0,s,x)$ and $i \colon \R \times \{s\} \times Q \to \R \times I \times Q$ be the inclusion, we can calculate
	\begin{gather*}
		g_s = \iota^* i^*\overline{g} = \iota^* i^* \Phi^* \tilde{\overline{g}} = \iota^* \Phi^* i^* \tilde{\overline{g}} = \iota^* \Phi^* \pi^* \tilde{g}_s = (\pi \circ \Phi \circ \iota)^* \tilde{g}_s = \phi_s^* \tilde{g}_s.
	\end{gather*}
	This implies of course that the volumes of $g_s$ and $\tilde{g}_s$ agree.
	
	Moreover, in the case of null Ricci curvature, we have
	\begin{gather*}
		\rc_s \upd s^2 = \ric^{\overline{g}} = \ric^{\Phi^* \tilde{\overline{g}}} = \Phi^* (\tilde{\rc}_s \upd s^2) = (\Phi^* \tilde{\rc}_s) \upd s^2,
	\end{gather*}
	because $\Phi^* \upd s = \upd (\Phi^* s) = \upd (s \circ \Phi) = \upd s$.
	This equation just tells us that $\rc_s = \phi_s^* \tilde{\rc}_s$ for all $s \in I$.
\end{proof}

\begin{proof}[Proof of \cref{thm:1-1}]
	We begin to define a map assigning to a smooth parametrized curve $s \mapsto (g_s,\rc_s,\lambda_s)$ an isometry class of Lorentzian manifolds.
	Here $g_s$, $s \in I$, is a family of unit-volume Ricci-flat metrics on $Q$, $\rc \in C^\infty(I \times Q)$ and $\lambda$ is a positive solution associated to the ODE~\eqref{eq:ScaleODE}.
	The isometry class we associate is represented by any Lorentzian manifold
	\begin{gather*}
		(\R \times I \times Q, \upd v \otimes \upd s + \upd s \otimes \upd v + u^{-2}\upd s^2 + \lambda_s^2\phi_s^* g_s),
	\end{gather*}
	where the smooth family $\phi_s$, $s \in I$, of diffeomorphisms of $Q$ and $u \in C^\infty(I \times Q)$ are chosen such that the Lorentzian manifold has null Ricci curvature $(\phi_s^* \rho_s)\upd s^2$.
	\Cref{Prop:WellDefModSp} states all Lorentzian manifolds of this kind belong to the same isometry class.
	Since \cref{Prop:SurjModSp} ensures that there is at least one of those, we get a well-defined map.
	
	Moreover, this map is obviously surjective on isometry classes of Lorentzian manifolds of the form~\eqref{eq:AGK_metrics_3} with null Ricci curvature:
	Defining $\lambda$ as in \cref{prop:ScaleODE} and $\rho \in C^\infty(I \times Q)$ through $\ric^{\overline{g}} = \rho \upd s^2$, we have that $\lambda_s^{-2}g_s$, $s \in I$, is a family of unit-volume Ricci flat metrics and by \cref{prop:ScaleODE} that $\lambda$ is a positive solution of the ODE~\eqref{eq:ScaleODE} associated to $(\lambda_s^{-2}g_s,\rho_s)$, $s \in I$.
	By definition, the curve $s \mapsto (\lambda_s^{-2}g_s,\rho_s,\lambda_s)$ gets mapped to the isometry class of~\eqref{eq:AGK_metrics_3}.
	
	This map is not yet injective.
	We make it injective by first dividing out the action of smooth families $\psi_s$, $s \in I$, of diffeomorphisms of $Q$ and afterwards affine reparametrizations of~$s$.
	For the first step it is clear that $\psi_s^*(g_s,\rc_s,\lambda_s) = (\psi_s^*g_s,\psi_s^*\rc_s,\lambda_s)$, $s \in I$, gets mapped to the same isometry class of Lorentzian manifolds since the family of diffeomorphisms $\psi_s$, $s \in I$, that gets chosen in the construction can entirely compensate for the effect of $\psi_s$, $s \in I$.
	Thus we have a well-defined map from the domain $\ModSp(I,Q)$. 
	For the second step, we note that elements in $\ModSp(I,Q)$ which are equivalent in the sense of~\eqref{eq:equivalent_curves} for some $(\alpha,\beta) \in (\R \setminus \{0\}) \times \R$ get mapped to Lorentzian manifolds that are isometric via~\eqref{eq:ScalingDiffeo}.
	So we get a well-defined surjection from $\ModSp(\bullet,Q)/\sim$ onto isometry classes of Lorentzian manifolds of the form~\eqref{eq:AGK_metrics_3} with null Ricci curvature and we will show that this map is injective.
	
	Let $(g_s,\rc_s,\lambda_s)$, $s \in I$, and $(\tilde{g}_s,\tilde{\rc}_s,\tilde{\lambda}_s)$, $s \in \tilde{I}$, each represent an element in $\ModSp(\bullet,Q)/\sim$.
	They get mapped to the isometry classes of
	\begin{align*}
		\overline{g} &\coloneqq \upd v \otimes \upd s + \upd s \otimes \upd v + u^{-2}\upd s^2 + \lambda_s^2\phi_s^* g_s &&\text{and} & \tilde{\overline{g}} &\coloneqq \upd v \otimes \upd s + \upd s \otimes \upd v + \tilde{u}^{-2}\upd s^2 + \tilde{\lambda}_s^2\tilde{\phi}_s^* \tilde{g}_s
	\end{align*}
	on $\R \times I \times Q$ and $\R \times \tilde{I} \times Q$, respectively, for suitable choices of families $\phi_s$, $s \in I$ and $\tilde{\phi}_s$, $s \in \tilde{I}$ and functions $u$ and $\tilde{u}$. 
	We assume that these isometry classes are the same, \ie that there is an isometry $\Psi \colon \R \times I \times Q \to \R \times \tilde{I} \times Q$ between the Lorentzian manifolds above.
	Consider the canonical complete lightlike parallel vector fields $\frac{\del}{\del v}$ of these metrics.
	If compared using $\Psi$, these vector fields could be linearly independent.
	In this case, \cref{Lem:TwoLlParV} shows that $\Psi$ can be chosen such that these vector fields are mapped to each other.
	Thus we can assume in any case that $\upd \Psi \left(\frac{\del}{\del v} \right) = \alpha^{-1} \frac{\del}{\del v}$ for some $\alpha \neq 0$, which is necessarily constant.
	
	Now we consider $\Phi \colon \R \times \tilde{J} \times Q \to \R \times \tilde{I} \times Q$ as in \eqref{eq:ScalingDiffeo} with $\alpha$ as above and $\beta$ chosen such that for some fixed point $p \in \R \times I \times Q$ the $s$-coordinates of $p$ and $\Phi^{-1}(\Psi(p))$ agree.
	We have that
	\begin{gather*}
		\Phi^* \tilde{\overline{g}} = \upd v \otimes \upd s + \upd s \otimes \upd v + \hat{u}^{-2}\upd s^2 + \tilde{\lambda}_{\alpha s + \beta}^2\tilde{\phi}_{\alpha s + \beta}^* \tilde{g}_{\alpha s + \beta},
	\end{gather*}
	where $\hat{u}(s,x) = \alpha^{-1}\tilde{u}(\alpha s + \beta,x)$ and $\ric^{\tilde{\overline{g}}} = \alpha^2 \tilde{\phi}_{\alpha s + \beta}^*\tilde{\rc}_{\alpha s + \beta} \upd s^2$.
	Then $\Phi^{-1} \circ \Psi$ is an isometry which maps the canonical lightlike parallel vector field of $\overline{g}$ onto the one of $\Phi^* \tilde{\overline{g}}$ and fixes the $s$-coordinate of the point $p$.
	Then \cref{Prop:InjModSp} tells us that $I = \tilde{J}$ and that there is a smooth family of diffeomorphisms $\psi_s$, $s \in I$, such that $\tilde{\lambda}_{\alpha s + \beta}^2\tilde{\phi}_{\alpha s + \beta}^* \tilde{g}_{\alpha s + \beta} = \lambda_s^2 \psi_s^*\phi_s^* g_s$ and $\alpha^2 \tilde{\phi}_{\alpha s + \beta}^*\tilde{\rc}_{\alpha s + \beta} = \psi_s^* \phi_s^*\rc_s$.
	So the parametrized moduli curves
	\begin{align*}
		[(g_s,\rc_s,\lambda_s)] &= [(\psi_s^*\phi_s^*g_s,\psi_s^*\phi_s^*\rc_s,\lambda_s)]
	\intertext{and}
	 	[(\tilde{g}_s,\tilde{\rc}_s,\tilde{\lambda}_s)] &= [(\tilde{\phi}_s^*\tilde{g}_s,\tilde{\phi}_s^*\tilde{\rc}_s,\tilde{\lambda}_s)]
	\end{align*}
	are equivalent and thus define the same element in $\ModSp(\bullet,Q)/\sim$.
\end{proof}

\subsection{Discussion and comparison with the AKM-construction}
As explained in the introduction, we aim at a strengthening of the main result of Ammann, Kröncke and Müller \cite{ammann.kroencke.mueller:21} relating initial data sets with lightlike parallel spinor with parametrized curves in the premoduli space $\mathcal{R}_{\parallel}(Q)/ \Diff_0(Q)$.
Here $\mathcal{R}_{\parallel}(Q)$ denotes the space of (Ricci-flat) Riemannian metrics on $Q$ that admit a parallel spinor and $\Diff_0(Q) \subset \Diff(Q)$ is the identity component of the space of diffeomorphisms of $Q$.
Before constructing the lightlike parallel spinors in the next section, let us briefly compare the AKM-construction with the classification developed in this section on a metric
level.

Recall that the metric part of this construction \cite[Main~Constr.~15]{ammann.kroencke.mueller:21} takes a smooth curve $I \to \mathcal{R}_{\parallel}(Q)/ \Diff_0(Q)$ and a positive smooth function $F = \sqrt{u} \colon I \to \R$ as input.
The first step is to choose a smooth representative $g_s$, $s \in I$, of the curve that is gauged such that $\div^{g_s}(\dot{g}_s) = 0$.
The initial data set then consists of $\gamma = \upd s^2 + g_s$ and $k$, where the latter is implicitely determined by the condition that $u\frac{\del}{\del s}$ should be lightlike-parallel, \cf \cref{Rem:DetK}.
By reparametrizing the $s$-coordinate, it is possible to get into the form of this article, where $\gamma = u^{-2}\upd s^2 + g_s$ and $-u^2\frac{\del}{\del s}$ is lightlike-parallel.

\begin{remark}\label{Rem:AKM-subset}
	We see that the initial data sets resulting from the AKM-construction correspond to the pp-wave metrics $\LorMet \coloneqq \upd v \otimes \upd s + \upd s \otimes \upd v + u^{-2}\upd s^2 + g_s$ with null Ricci curvature, where $u$ is constant along each leaf $\R \times \{s\} \times Q$ and $\div^{g_s}(\dot{g}_s) = 0$.
	The latter condition forces the Hessian part in the decomposition \cref{prop_splitting_constraint_solutions} to vanish.
	Writing $g_s = \lambda_s^2 \tilde{g}_s$ for a family of unit-volume metrics $\tilde{g}_s$, $s \in I$, and a positive function $\lambda \in C^\infty(I)$, and $\tilde{\sigma}_s$ for the TT-part of $\dot{\tilde{g}}_s$ with respect to $\tilde{g}_s$, we see that this pp-wave corresponds via \cref{thm:1-1} to the parametrized moduli curve
	\begin{gather*}
	[(\tilde{g}_s,\rc_s,\lambda_s)] \in \ModSpPar(\bullet,Q)/\sim, \text{ where } \rc_s = -(n-1) \frac{\ddot{\lambda}_s}{\lambda_s} - \frac{1}{4} |\tilde{\sigma}_s|_{\tilde{g}_s}^2.
	\end{gather*}
	Observe that this shows that in the AKM-construction $\rc_s$ is determined up to adding leafwise constant functions.
	To get parallel spinors on those pp-waves with null Ricci curvature where $\rc$ differs from $s \mapsto - \frac{1}{4} |\tilde{\sigma}_s|_{\tilde{g}_s}^2$ by more than a function in $C^\infty(I)$, we need the improved construction from the next section.
\end{remark}

\begin{remark} \label{Rem:ModuliSpaces}
	The formulation with the moduli space in \cite{ammann.kroencke.mueller:21} might be a bit misleading.
	First of all, it is remarkable that $\mathcal{R}_\parallel(Q) / \Diff_0(Q)$ carries the structure of a finite-dimensional smooth manifold -- which is for instance no longer true in general if $\Diff_0(Q)$ gets replaced by $\Diff(Q)$.
	The smooth structure is such that a smooth map $I \to \mathcal{R}_\parallel(Q) / \Diff_0(Q)$ is one that can be represented by a smooth map $I \to \mathcal{R}_\parallel(Q)$.
	In other words, the set $C^\infty(I,\mathcal{R}_{\parallel}(Q) / \Diff_0(Q))$ of smooth maps $I \to \mathcal{R}_{\parallel}(Q) / \Diff_0(Q)$ is the image of
	\begin{gather*}
	C^\infty(I,\mathcal{R}_{\parallel}(Q)) \hookrightarrow \mathcal{R}_{\parallel}(Q)^I\longrightarrow (\mathcal{R}_{\parallel}(Q) / \Diff_0(Q))^I.
	\end{gather*}
	The AKM-construction now starts with such a smooth representative and makes it divergence-free, but the result might a priori depend on the chosen representative.
	Two representatives lead to the same initial data sets, though, if they are related by a smooth family of diffeomorphisms.
	Thus the result of the AKM-construction depends only on the curve in $I \to \mathcal{R}_\parallel(Q) / \Diff_0(Q)$ and not on the representative (in $I \to \mathcal{R}_\parallel(Q)$) if
	\begin{gather} \label{eq:ModCurvVsCurvMod}
	C^\infty(I,\mathcal{R}_{\parallel}(Q)) / C^\infty(I, \Diff_0(Q)) \longrightarrow C^\infty(I,\mathcal{R}_{\parallel}(Q) / \Diff_0(Q))
	\end{gather}
	is injective.
	At the moment, we do neither have a proof nor a disproof for the injectivity of~\eqref{eq:ModCurvVsCurvMod}.
\end{remark}

In our analysis, we were lead to considering the full $\Diff(Q)$ instead of $\Diff_0(Q)$.
In this case, injectivity of the corresponding map~\eqref{eq:ModCurvVsCurvMod} fails.
Even worse, since we want to prescribe the Ricci tensor of the spacetime with null Ricci curvature, we have to enrich our configuration space by a smooth function.
Once we replace $\mathcal{R}_{\parallel}(Q)$ by $\mathcal{R}_{\parallel}(Q) \times C^\infty(Q)$, the analog of~\eqref{eq:ModCurvVsCurvMod} is definitely non-injective as \cref{ex_ModuliSpaceII} shows, even when we stick to $\Diff_0(Q)$.
As such, it explains why in our one-to-one correspondences in \cref{thm:1-1,cor:Ricci-flat,thm:1-1_spin_case}, we really have to consider moduli curves instead of curves in a moduli space.

\begin{example}\label{ex_ModuliSpaceII}
	Consider $Q = T^2 = \R^2 / \Z^2$ equipped with the constant family $g_s = g_{\mathrm{eucl}}$ for all $s \in \R$, where $g_{\mathrm{eucl}}$ is the standard Euclidean metric.
	Now, let $f \in C^\infty(Q)$ be a function with a unique maximum.
	Let us fix some isometry $\id \neq S \colon Q \to Q$ that is induced by a shift in $\R^2$.
	Furthermore, we fix a smooth function $\theta \colon \R \to \R$ such that $\theta$ is positive on $\R_{>0}$ and vanishes identically on $\R_{<0}$, and let $\sigma \colon \R \to \R$, $t \mapsto -t$.
	Now, consider the smooth families $\R \to \mathcal{R}_{\parallel}(Q) \times C^\infty(Q)$ given by $(g_s, \theta(s)f + \theta(\sigma(s))f)$, $s \in I$, and $(g_s, \theta(s) f + \theta(\sigma(s))(f \circ S))$, $s \in I$.
	It is easy to see that for each $s \in \R$ the pairs are related by an element in $\Diff_0(Q)$.
	In fact, we can choose the identity for $s \leq 0$ and $S \colon Q \to Q$ for $s > 0$.
	On the other hand, there is no smooth -- not even a continuous -- choice for such diffeomorphisms $\phi_s$, $s \in I$, since if $p \in Q$ is the place of the unique maximum of $f$, then we must have $\phi_s(p) = p$ for $s < 0$ and $\phi_s(p) = S(p) \neq p$ for $s > 0$.
\end{example}

\section{Constructing parallel spinors on simple pp-waves}\label{sec.constr.cyl.ids.llps}
\subsection{Parallel spinors on spacetimes and initial data sets}
The purpose of this section is to construct parallel spinors on simple pp-waves.
We do this in a two step process.
First, we show that there is a one-to-one correspondence between lightlike parallel spinors on simple pp-waves and suitable spinors on initial data sets.
Then, in the next subsections, we construct the latter.

Let $(\overline{M},\overline{g})$ be a time-oriented Lorentzian spin manifold with a chosen spin structure.
For simplicity of the presentation, we will work with the classical complex irreducible spinor bundles here, although the constructions work equally well for others as, for instance, the real $\mathrm{Cl}$-linear spinor bundle (which is considered in \cite{Gloeckle:2024_a}).
We take such a spinor bundle $\Sigma \overline{M} \to \overline{M}$ over $\overline{M}$ associated to the chosen spin structure.
As usual, it comes with a Clifford multiplication which is parallel with respect to the connection $\nabla^{\overline{g}}$ induced by the Levi-Civita connection.
In contrast to the Riemannian situation, there is no compatible positive definite scalar product, but there is an indefinite (Hermitian) inner product $\llangle \bullette, \bullette \rrangle$ that is parallel and for which the Clifford multiplication by any vector in $T\overline{M}$ is self-adjoint.
Moreover, it can be taken such that $\langle \bullette, \bullette \rangle_T \coloneqq \llangle T \cdot \bullette, \bullette \rrangle$ is positive definite for any future-timelike vector $T$.
How to obtain the inner product is for example briefly explained in \cite[end of Ch.~3.2]{Ammann.Gloeckle:2023} following the original work by Helga Baum \cite[esp.~Ch.~1.5]{Baum:1981}.

Given $\Psi \in \Sigma_p \overline{M}$, $p \in \overline{M}$, self-adjointness of the Clifford multiplication by a vector implies that $\overline{\llangle X \cdot \Psi, \Psi \rrangle} = \llangle \Psi, X \cdot \Psi \rrangle = \llangle  X \cdot \Psi, \Psi \rrangle$ for all $X \in T_p \overline{M}$.
Hence, $-\llangle  \bullette \cdot \Psi, \Psi \rrangle$ defines a real-valued $1$-form and is thus represented by a vector.
\begin{definition} \label{Def:DiracCurr}
	The (Lorentzian) \emph{Dirac current} of a spinor $\Psi \in \Sigma_p \overline{M}$, $p \in \overline{M}$, is the vector $V_{\Psi} \in T_p \overline{M}$ uniquely determined by
	\begin{gather*}
		\overline{g}(V_\Psi,X) = -\llangle X \cdot \Psi, \Psi \rrangle \qquad \text{ for all } X \in T_p \overline{M}.
	\end{gather*}
\end{definition}

A short calculation (\cf \cite[Lem.~1.3.8]{Gloeckle:2024_c}) yields that
\begin{align} \label{eq:DiracCurrFlavor}
	\langle V_\Psi \cdot \Psi, V_\Psi \cdot \Psi \rangle_T = -\overline{g}(V_\Psi, V_\Psi) \langle \Psi, \Psi \rangle_T
\end{align}
for any future-timelike vector $T \in T_p \overline{M}$.
Thus the Dirac-Current is either timelike or lightlike or zero.
Of course, it is zero if and only if $\Psi = 0$.
Otherwise, it is a future vector because $\overline{g}(V_\Psi,T) = -\langle \Psi, \Psi \rangle_T \leq 0$ for any future-timelike $T$.
Equation~\eqref{eq:DiracCurrFlavor} also shows that if $\Psi \neq 0$, then $V_\Psi$ is lightlike if and only if $V_\Psi \cdot \Psi = 0$.
Furthermore, if $V \cdot \Psi = 0$ for a future-lightlike vector $V$, then $0 = \overline{g}(V_\Psi, V)$ and thus $V_\Psi$ is proportional to $V$.

\begin{definition}
	A spinor $\Psi \in \Sigma_p \overline{M}$ is called \emph{timelike} or \emph{lightlike} if $V_\Psi$ is timelike or lightlike, respectively. 
\end{definition}

\begin{lemma} \label{Lem:LlSpinor}
	Let $\Psi \in \Gamma(\Sigma \overline{M})$ be a lightlike parallel spinor.
	Then $(\overline{M},\overline{g})$ is a pp-wave spacetime with lightlike parallel vector field $V_\Psi$ and null Ricci curvature.
\end{lemma}
\begin{proof}
	Because the inner product, the Clifford multiplication and the metric $\overline{g}$ are parallel with respect to $\nabla^{\overline{g}}$, parallelism of $\Psi$ implies that $\nabla^{\overline{g}} V_\Psi = 0$.
	Because $V_\Psi$ is lightlike by assumption, this implies that $(\overline{M},\overline{g})$ is a pp-wave.
	
	Moreover, since $\Psi$ is parallel, $R^{\overline{g}}(X,Y) \cdot \Psi = 0$ for $X,Y \in T\overline{M}$.
	Clifford multiplying with $Y$ and tracing over $Y$, the same calculation as in the Riemannian case (\cf \eg \cite[(3.17)]{Roe:1999}) yields
	\begin{gather*}
		0 =\tr^{\overline{g}}(\bullette \cdot R^{\overline{g}}(X,\bullette) \Psi) = -\frac{1}{2}\Ric^{\overline{g}}(X) \cdot \Psi
	\end{gather*}
	for all $X \in T\overline{M}$.
	In particular, for all $X \in T\overline{M}$
	\begin{gather*}
		0 = \Ric^{\overline{g}}(X) \cdot \Ric^{\overline{g}}(X) \cdot \Psi = -\overline{g}( \Ric^{\overline{g}}(X), \Ric^{\overline{g}}(X)) \Psi
	\end{gather*}
	and so the Ricci curvature is null.
\end{proof}
\begin{remark}
	Note that if $\Psi \in \Gamma(\Sigma \overline{M})$ is parallel (and hence also $V_\Psi$), then whether $V_\Psi$ is timelike, lightlike or zero does not change within a connected component of $\overline{M}$.
	Thus for connected $\overline{M}$, it suffices to check that $\Psi$ is lightlike in one point in order to be able to draw the conclusions from \cref{Lem:LlSpinor}.
\end{remark}

We now consider a spacelike hypersurface $M$ in $(\overline{M},\overline{g})$ and denote by $(\gamma,k)$ the induced initial data set.
First of all, we remark that the spin structure of $(\overline{M},\overline{g})$ induces a spin structure on $(M,\gamma)$, and in particular $M$ is spin.
The following partial converse of this will also be relevant for us:
If there is a nowhere vanishing vector field $V$ on $(\overline{M},\overline{g})$ whose flow lines intersect $M$ precisely once, then $(\overline{M},\overline{g})$ is spin if $M$ is spin and for every spin structure of $(M,\gamma)$ there is a unique spin structure of $(\overline{M},\overline{g})$ that restricts to it.

The restriction $\Sigma \overline{M}_{|M} \to M$ has a Clifford multiplication by vectors in $TM$, which is simply obtained by restricting the Clifford multiplication to $TM \subset T\overline{M}_{|M}$.
Moreover, denoting the future unit normal on $M$ by $e_0$, it comes with an involution $e_0 \cdot$ that anti-commutes with the $TM$-Clifford multiplication and a positive definite scalar product $\langle \bullette, \bullette \rangle_{e_0}$ for which multiplication by $X \in TM$ is skew-adjoint and $e_0 \cdot$ is self-adjoint.
Note though that the $TM$-Clifford multiplication, the involution and the scalar product are in general not parallel with respect to the restricted connection $\nabla^{\overline{g}}$.

It is important to note that $\Sigma \overline{M}_{|M} \to M$ can be reconstructed from the initial data set alone.
We ignore the connection for a moment and just consider Clifford-modules with compatible scalar product.
In the following, isomorphisms are understood to preserve these structures.
There are two cases:
If $n = \dim(M)$ is even, then there are two non-isomorphic choices $\Sigma^\pm_{n,1}$ of an irreducible $\CCl_{n,1}$-module.
Whichever we take, however, the restriction of the Clifford multiplication along $\R^n \hookrightarrow \R^{n,1}$ up to isomorphism yields the unique irreducible $\CCl_n$-module $\Sigma_n$, the only difference being that in one case the involution $e_0 \cdot$ is given by multiplication with the complex volume element $\omega_\C$ of $\R^n$ and in the other with its negative.
If $n = \dim(M)$ is odd, then the irreducible $\CCl_{n,1}$-module $\Sigma_{n,1}$ is essentially unique, but there are two choices $\Sigma^+_n$ and $\Sigma^-_n$ of an irreducible $\CCl_n$-module, distinguished by whether the complex volume element acts by $1$ or $-1$.
However, there is an isomorphism $\Phi \colon \Sigma^-_n \to -\Sigma^+_n$, where the minus indicates that the Clifford multiplication is precomposed by $\R^n \to \R^n,\, X \mapsto -X$.
The restriction of the Clifford multiplication of $\Sigma_{n,1}$ along $\R^n \hookrightarrow \R^{n,1}$ splits into the $\pm1$-eigenspaces of the complex volume element of $\R^n$ and is thus isomorphic to the orthogonal sum $\Sigma^+_n \oplus \Sigma^-_n$, with $e_0 \cdot$ interchanging the components.

This motivates that we define the (classical) \emph{hypersurface spinor bundle} $\overline{\Sigma} M \to M$ to be given by the classical spinor bundle $\Sigma M \to M$ if $n$ is even and by the orthogonal sum $\Sigma^+ M \oplus \Sigma^- M \to M$ if $n$ is odd.
It comes equipped with a (positive definite) Hermitian scalar product $\langle \bullette, \bullette \rangle$ and a Clifford multiplication of $TM$.
Furthermore, it can be endowed with a self-adjoint involution $e_0 \cdot$ that anti-commutes with the $TM$-multiplication:
In the even case given by either $\omega_\C$ or $- \omega_\C$ (there are two different hypersurface spinor bundles in this case) and in the odd case by $\Sigma^+ M \oplus \Sigma^- M \ni (\psi_+,\psi_-) \mapsto (\Phi(\psi_-),\Phi^{-1}(\psi_+))$.
Moreover, we equip it with the connection
\begin{gather*}
	\overline{\nabla}_X \psi \coloneqq \nabla^{\gamma}_X \psi - \frac{1}{2} e_0 \cdot k(X,\bullette)^{\sharp^{\gamma}} \cdot \psi
\end{gather*}
for $\psi \in \Gamma(\overline{\Sigma} M)$ and $X \in TM$.

\begin{lemma} \label{Lem:RestrictSpinorBundles}
	Let $(\gamma,k)$ be the induced initial data set on a spin manifold $M$ and $\overline{\Sigma} M \to M$ be a hypersurface spinor bundle.
	Suppose that $(\overline{M},\overline{g})$ is a time-oriented Lorentzian manifold into which $M$ embeds as spacelike hypersurface such that the induced initial data set is $(\gamma,k)$.
	Assume that there is a spin structure on $(\overline{M},\overline{g})$ that restricts to the given one of $(M,\gamma)$.
	Then the (or one of the two) spinor bundle(s) $\Sigma \overline{M} \to \overline{M}$ has the property that
	\begin{gather} \label{eq:RestrictSpinorBundles}
		\Sigma \overline{M}_{|M} \overset{\cong}{\longrightarrow} \overline{\Sigma}M.
	\end{gather}
	The isomorphism respects scalar product, $TM$-Clifford multiplication and involution $e_0 \cdot$ and maps the restriction of the connection $\nabla^{\overline{g}}$ to $\overline{\nabla}$.
\end{lemma}
\begin{proof}
	Because the spin structure of $(\overline{M},\overline{g})$ restricts to the one of $(M,\gamma)$, the bundle $\Sigma \overline{M}_{|M} \to M$ is associated to the spin structure of $(M,\gamma)$.
	The isomorphisms of irreducible Clifford-modules discussed above then induce the stated isomorphism.
	Moreover, the discussion shows that the definition of $e_0 \cdot$ on $\overline{\Sigma} M$ is made such it is preserved, after possibly changing to the other spinor bundle $\Sigma \overline{M} \to \overline{M}$ (or the other hypersurface spinor bundle) in the case where $n$ is even. 
	Finally, the calculation that the restriction of $\nabla^{\overline{g}}$ is given by the formula for $\overline{\nabla}$ is a well-known consequence of the identity
	\begin{align*}
		\nabla^{\overline{g}}_X Y = \nabla^{\gamma}_X Y + k(X,Y)e_0
	\end{align*}
	for $X, Y \in \Gamma(TM)$ and the definition of the spinorial connection.
	A more detailed proof for this can, for example, be found in \cite[Lem.~2.3.4]{Gloeckle:2024_c}.
\end{proof}

The main property of hypersurface spinor bundles is that they carry a natural Dirac-type operator with a favorable Schrödinger-Lichnerowicz type formula.
We will come to this later and first study the correspondence of lightlike parallel spinors.

\begin{definition}
	The (Riemannian) \emph{Dirac current} of a hypersurface spinor $\psi \in \overline{\Sigma}_p \overline{M}$, $p \in M$, is the vector $U_{\psi} \in T_p M$ uniquely determined by
	\begin{gather*}
	\gamma(U_\psi,X) = \langle e_0 \cdot X \cdot \psi, \psi \rangle \qquad \text{ for all } X \in T_p M.
	\end{gather*}
	A non-zero hypersurface spinor $\psi$ is called \emph{timelike} or \emph{lightlike} if $|\psi|^2 \coloneqq \langle \psi, \psi \rangle > |U_\psi|_{\gamma}$ or $|\psi|^2 = |U_\psi|_{\gamma}$, respectively.
\end{definition}

Similarly as in the Lorentzian case, well-definedness of the Dirac current follows from $\overline{\langle e_0 \cdot X \cdot \psi, \psi \rangle} = \langle \psi, e_0 \cdot X \cdot \psi \rangle = - \langle X \cdot e_0 \cdot \psi, \psi \rangle = \langle e_0 \cdot X \cdot \psi, \psi \rangle$ for all $X \in T_pM$ and $\psi \in \overline{\Sigma}_p M$.

\begin{remark} \label{Rem:CharLlSpinors}
	Also in this case $\overline{\Sigma}_p M \ni \psi \neq 0$ has to be either timelike or lightlike since
	\begin{gather*}
		|U_\psi|_{\gamma}^2 = |\gamma(U_\psi,U_\psi)| = |\langle e_0 \cdot U_\psi \cdot \psi, \psi \rangle| \leq |e_0 \cdot U_\psi \cdot \psi||\psi| = |U_\psi|_{\gamma} |\psi|^2.
	\end{gather*}
	The equality analysis in the Cauchy-Schwarz inequality shows that the lightlikeness can be equivalently characterized as $U_\psi \cdot \psi = |\psi|^2 e_0 \cdot \psi$ for $\psi \neq 0$.
	Furthermore, if $U \cdot \psi = |U|_\gamma e_0 \cdot \psi$ for some $0 \neq U \in T_pM$, then $|U|_\gamma |U_\psi|_\gamma \geq |\gamma(U_\psi, U)| = |U|_\gamma |\psi|^2 \geq |U|_\gamma |U_\psi|_\gamma$ shows that $U_\psi$ is proportional to $U$ and $|U_\psi| = |\psi|^2$.
\end{remark}
\begin{lemma}
	The Riemannian Dirac current $U_\psi$ of a lightlike parallel hypersurface spinor $\psi \in \Gamma(\overline{\Sigma}M)$ is lightlike-parallel.
\end{lemma}
\begin{proof}
	Using that $TM$-Clifford multiplication, $\langle \bullette, \bullette \rangle$ and $e_0 \cdot$ are $\nabla^{\gamma}$-parallel, we can directly verify the claim via
	\begin{align*}
		\gamma(\nabla^{\gamma}_X U_\psi, Y) &= \del_X \gamma(U_\psi, Y) - \gamma(U_\psi, \nabla^{\gamma}_X Y) \\
			&= \del_X \langle e_0 \cdot Y \cdot \psi, \psi \rangle - \langle e_0 \cdot (\nabla^{\gamma}_X Y) \cdot \psi, \psi \rangle \\
			&= \langle e_0 \cdot Y \cdot \nabla^{\gamma}_X \psi, \psi \rangle + \langle e_0 \cdot Y \cdot \psi, \nabla^{\gamma}_X \psi \rangle \\
			&= \frac{1}{2} \langle e_0 \cdot Y \cdot e_0 \cdot k(X,\bullette)^{\sharp^\gamma} \cdot \psi, \psi \rangle + \frac{1}{2}\langle e_0 \cdot Y \cdot \psi, e_0 \cdot k(X,\bullette)^{\sharp^\gamma} \cdot \psi \rangle \\
			&= -\frac{1}{2} \langle (Y \cdot k(X,\bullette)^{\sharp^\gamma} \cdot + k(X,\bullette)^{\sharp^\gamma} \cdot Y \cdot) \psi, \psi \rangle \\
			&= k(X,Y) |\psi|^2 = |U_\psi|_{\gamma} k(X,Y). \qedhere
	\end{align*}
\end{proof}

Existence of a lightlike parallel spinor also implies that the induced metrics on the leaves of the foliation defined by $U_\psi^\perp$ are a smooth family $g_s$, $s \in I$, of Ricci-flat metrics that satisfy the j-equation $\div^{g_s}(\dot{g_s}) - \upd \tr^{g_s}(\dot{g_s}) = 0$.
This follows from the following one-to-one correspondence for (pp-wave) spacetimes and (pp-wave) initial data sets with parallel spinor, \cref{Lem:LlSpinor} and \cref{lem.lor.ricci-flat}.
It is also possible to piece together a proof purely on the level of initial data sets using the calculations in the following subsections that we need for our construction of lightlike parallel spinors, \cf \cref{Rem:NullRicci}.

\begin{lemma} \label{Lem:RestrictSpinors}
	Let $M$ be a spacelike hypersurface in a time-oriented Lorentzian manifold $(\overline{M},\overline{g})$ with induced initial data set $(\gamma,k)$.
	Let $\Sigma \overline{M} \to \overline{M}$ be a spinor bundle for $\overline{M}$ and $\overline{\Sigma} M \to M$ a hypersurface spinor bundle such that $\Sigma \overline{M}_{|M} \cong \overline{\Sigma} M$ as in \cref{Lem:RestrictSpinorBundles}.
	If $\Psi \in \Sigma \overline{M}_{|M}$ is lightlike with Dirac current $V_\psi$, then its image $\psi$ under the isomorphism~\eqref{eq:RestrictSpinorBundles} is lightlike and its Dirac current $U_\psi$ satisfies $V_\Psi = |\psi|^2 e_0 - U_\psi$.
	
	In particular, if $\Psi \in \Gamma(\Sigma \overline{M})$ is lightlike parallel with Dirac current $V_\psi$, then the hypersurface spinor $\psi$ corresponding to $\Psi_{|M}$ under the isomorphism is lightlike and parallel (with respect to $\overline{\nabla}$) with Dirac current subject to $(V_\Psi)_{|M} = |\psi|^2 e_0 - U_\psi$.		
\end{lemma}
\begin{proof}
	For the Dirac current, we observe that the stated equation from
	\begin{align*}
		\overline{g}(V_\Psi, e_0) = -\llangle e_0 \cdot \Psi, \Psi \rrangle = -\langle \Psi, \Psi \rangle_{e_0} &= -\langle \psi, \psi \rangle = -|\psi|^2 = \overline{g}(|\psi|^2e_0,e_0)
	\intertext{and}
		\overline{g}(V_\Psi, X) = -\llangle X \cdot \Psi, \Psi \rrangle = -\langle e_0 \cdot X \cdot \Psi, \Psi \rangle_{e_0} &= -\langle e_0 \cdot X \cdot \psi, \psi \rangle = -\gamma(U_\psi, X) = \overline{g}(-U_\psi,X)
	\end{align*}
	for all $X \in TM$.
	As $V_\Psi$ is lightlike, it also follows directly that $\psi$ is lightlike.
	
	Since the bundle isomorphism preserves the connection, the last statement about parallelism of $\psi$ is immediate.
\end{proof}

Conversely, lightlike parallel hypersurface spinors parallely extend to the Killing development.
The following proposition is shown in \cite[Thm.~4.5]{Ammann.Gloeckle:2023}.

\begin{proposition} \label{Prop:ExtendSpinors}
	Let $(\gamma,k)$ be an initial data set on $M$ with lightlike parallel hypersurface spinor $\psi \in \Gamma(\overline{\Sigma} M)$.
	Let $U_\psi$ be the Riemannian Dirac current of $\psi$ and consider the pp-wave spacetime
	\begin{align*}
		(\overline{M},\overline{g}) \coloneqq (\R \times M, -\upd v \otimes \gamma(U_\psi,\bullette) - \gamma(U_\psi,\bullette) \otimes \upd v + \gamma),
	\end{align*}
	where $v$ is the $\R$-coordinate, which induces $(\gamma,k)$ on $M = \{0\} \times M$ and whose lightlike parallel vector field $\frac{\del}{\del v}$ projects orthogonally to $-U_\psi$ on $M$.
	Equip $(\overline{M},\overline{g})$ with the spin structure and spinor bundle such that $\Sigma \overline{M}_{|M} \cong \overline{\Sigma} M$ (as in \cref{Lem:RestrictSpinorBundles}) and identify these bundles.
	Then $(\overline{M},\overline{g})$ admits a (unique) lightlike parallel spinor~$\Psi$ extending $\psi$, and its Dirac current is given by $\frac{\del}{\del v}$.
\end{proposition}

Recall from \cref{Cor:1-1-IDS} that restricting and forming the Killing development yield a one-to-one correspondence between simple pp-waves with fixed spacelike hypersurface and certain pp-wave initial data sets.
\Cref{Lem:RestrictSpinors,Prop:ExtendSpinors} immediately provide a correspondence of lightlike parallel spinors between the corresponding objects.
 
\begin{corollary} \label{Cor:SimpleIDSSpinors}
	Let
	\begin{gather*}
		(\overline{M},\overline{g}) \coloneqq (\R \times I \times Q, \upd v \otimes \upd s + \upd s \otimes \upd v + u^{-2} \upd s^2 + g_s)
	\end{gather*}
	be a simple pp-wave and $(u^{-2} \upd s^2 + g_s, k)$ be the induced pp-wave initial data set on $\{0\} \times I \times Q$.
	Then restriction and parallel extension yield a one-to-one correspondence between lightlike parallel spinors on $(\overline{M},\overline{g})$ with Dirac current $\frac{\del}{\del v}$ and lightlike parallel hypersurface spinors on the initial data set with Riemannian Dirac current~$-u^2 \frac{\del}{\del s}$.
\end{corollary}

These results show that we in order to construct lightlike parallel spinors on spacetimes, we can equally well construct lightlike parallel hypersurface spinors on initial data sets.

\subsection{Analyzing parallelism for spinors on initial data sets}
Because of \cref{Cor:1-1-IDS,Cor:SimpleIDSSpinors}, our aim to construct lightlike parallel spinors on simple pp-waves modeled on closed manifolds can be achieved by constructing lightlike parallel hypersurface spinors on certain pp-wave initial data sets.
We are lead to consider the following setup, actually only for closed manifolds $Q$:
\begin{setup} \label{Setup:SimpleIDS}
	The initial data set $(\gamma, k)$ is defined on $M \coloneqq I \times Q$ with $Q$ an $(n-1)$-dimensional manifold and $I \subset \R$ an open interval.
	The metric $\gamma$ is given by
	\begin{gather*}
	\gamma = u^{-2} \upd s^2 + g_s,
	\end{gather*}
	where $s$ denotes the $I$-coordinate, $u \in C^\infty(I \times Q)$ is positive and $g_s$, $s \in I$, is a smooth family of Riemannian metrics on $Q$.
	The second fundamental form $k$ is determined by the condition that $U \coloneqq -u^2 \frac{\del}{\del s}$ is lightlike-parallel, \cf \cref{Rem:DetK}.
\end{setup}

Now notice that $M$ is spin if and only if $Q$ is spin and the (topological) spin structures of $M$ and $Q$ one-to-one correspond to each other via restriction and extension.
Here we use the unit normal $\nu \coloneqq u \frac{\del}{\del s}$ on the canonical leaves $\{s\} \times Q$ to identify the underlying frame bundles.
Irrespective of whether $n = \dim(M)$ is even or odd, $\overline{\Sigma}M_{|Q}$ is double the dimension of $\Sigma Q$.
The reason is that any irreducible complex representation $\Sigma_{n-1}$ of $\CCl_{n-1}$ gives rise to an irreducible complex representation $\Sigma_{n-1} \oplus -\Sigma_{n-1}$ of $\CCl_{n,1}$, where the additional basis vectors of $\R^{n,1}$ operate via the matrices
\begin{gather*}
	\begin{pmatrix} 0 & 1 \\ 1 & 0 \end{pmatrix} \qquad \text{ and } \qquad \begin{pmatrix} 0 & -1 \\ 1 & \phantom{-}0 \end{pmatrix}.
\end{gather*}
These, or their negatives, will correspond to $e_0 \cdot$ and $\nu \cdot$, respectively.

\begin{lemma} \label{Lem:RestrictSpinorBundlesQ}
	In \cref{Setup:SimpleIDS}, let $M$ and $Q$ be equipped with compatible spin structures, where we identify $Q$ with $\{s\} \times Q$ for an arbitrary $s \in I$.
	Then the subbundle $\{\psi \in \overline{\Sigma}M \mid -\nu \cdot \psi = e_0 \cdot \psi\} \subset \overline{\Sigma}M$ is $\overline{\nabla}$-parallel.
	Furthermore, there is a bundle isomorphism
	\begin{gather} \label{eq:RestrictSpinorBundlesQ}
		\{\psi \in \overline{\Sigma}M_{|Q} \mid -\nu \cdot \psi = e_0 \cdot \psi\} \overset{\cong}{\longrightarrow} \Sigma Q
	\end{gather}
	that is isometric, $TQ$-Clifford linear and behaves in the following way with respect to the connection:
	If $\overline{\nabla}$ denotes the restriction the connection of $\overline{\Sigma}M$ and $\nabla^{Q}$ is the pullback along the above bundle isomorphism of the Levi-Civita connection of $\Sigma Q$, then
	\begin{gather*}
		\overline{\nabla} \psi = \sqrt{u} \nabla^{Q} \frac{\psi}{\sqrt{u}}
	\end{gather*}
	for any $\psi \in \Gamma(\overline{\Sigma} M_{|Q})$ with $-\nu \cdot \psi = e_0 \cdot \psi$.
\end{lemma}
\begin{proof}
	The connection $\overline{\nabla}$ defined on $\underline{\R} e_0 \oplus TM$ by
	\begin{align} \label{eq:HypersurfaceConnection}
		\overline{\nabla}_X (ye_0 + Y) \coloneqq (\del_X y + k(X,Y))e_0 + (yk(X,\bullette)^{\sharp^\gamma} + \nabla^\gamma_X Y)
	\end{align}
	for $X \in TM$, $y \in \C^\infty(M)$ and $Y \in \Gamma(TM)$ is compatible with the connection $\overline{\nabla}$ of $\overline{\Sigma}M$ in the sense that the Clifford-multiplication by $\underline{\R} e_0 \oplus TM$ becomes parallel.
	This can be checked by direct calculation but also follows immediately from the correspondence \cref{Lem:RestrictSpinorBundles} since $\overline{\nabla}$ is just the restriction of $\nabla^{\overline{g}}$ to $T\overline{M}_{|M} \cong \underline{\R} e_0 \oplus TM$.
	That $U$ is lightlike-parallel amounts to $V \coloneqq |U|e_0 -U$ being $\overline{\nabla}$-parallel.
	But this implies that $\ker(V \cdot) = \{\psi \in \overline{\Sigma}M \mid -\nu \cdot \psi = e_0 \cdot \psi\}$ is an $\overline{\nabla}$-parallel subbundle.
	
	Because the spin structure of $(M,\gamma)$ restricts to the one of $(Q,g_s)$, the bundle $\overline{\Sigma}M_{|Q} \to Q$ is associated to the spin structure of $(Q,g_s)$.
	The isomorphisms of irreducible Clifford-modules discussed above then induce an isometric, $TQ$-Clifford linear isomorphism
	\begin{gather*}
		\overline{\Sigma} M_{|Q} \overset{\cong}{\longrightarrow} \Sigma Q \oplus \Sigma Q,
	\end{gather*}
	where the summands correspond to the $\pm 1$-eigenspaces of $e_0 \cdot \nu \cdot$.
	In particular, taking the $-1$-eigenspace, we get an isomorphism as in the claim.
	
	It remains to check the formula for the connections.
	Similarly as in the proof of \cref{Lem:RestrictSpinorBundles}, the spinorial connections satisfy
	\begin{align*}
		\nabla^{\gamma}_X \psi = \nabla^{Q}_X \psi + \frac{1}{2} \nu \cdot (\nabla^\gamma_X \nu) \cdot \psi
	\end{align*}
	for $X \in TQ$, $\psi \in \Gamma(\overline{\Sigma}M_{|Q})$.
	Using this, we have for all $\psi \in \Gamma(\overline{\Sigma}M_{|Q})$ with $-\nu \cdot \psi = e_0 \cdot \psi$ that
	\begin{align*}
		\overline{\nabla}_X \psi -  \nabla^{Q}_X \psi &= \left(\nabla^\gamma_X \psi - \frac{1}{2} e_0 \cdot k(X,\argu)^{\sharp^\gamma} \cdot \psi \right) - \left(\nabla^\gamma_X - \frac{1}{2} \nu \cdot (\nabla^\gamma_X \nu) \cdot \psi \right)\\
			&= \;\frac{1}{2} k(X,\argu)^{\sharp^\gamma} \cdot e_0 \cdot \psi + \frac{1}{2} \nu \cdot (\nabla^\gamma_X \nu) \cdot \psi\\
			&= - \frac{1}{2} k(X,\argu)^{\sharp^\gamma} \cdot \nu \cdot \psi + \frac{1}{2} \nu \cdot (\overline{\nabla}_X \nu) \cdot \psi - \frac{1}{2} k(X,\nu) \nu \cdot e_0 \cdot \psi \\
			&= \frac{1}{2} \nu \cdot  k(X,\argu)^{\sharp^\gamma} \cdot \psi + k(X,\nu) \psi + \frac{1}{2} \nu \cdot (\overline{\nabla}_X \nu) \cdot \psi - \frac{1}{2} k(X,\nu) \psi \\
			&= \frac{1}{2} \nu \cdot (\overline{\nabla}_X (e_0 + \nu)) \cdot \psi + \frac{1}{2}k(X,\nu) \psi.
	\end{align*}
	Because $V = u(e_0 + \nu)$ is $\overline{\nabla}$-parallel, $\overline{\nabla}_X (e_0 + \nu)$ is proportional to $V$ and so the first summand vanishes because of $V \cdot \psi = 0$.
	For the second summand we use that $U = -u\nu$ satisfies $\del_X u = \del_X |U|_\gamma = k(X,U)$, so that $k(X,\nu) = -u^{-1} \del_X u$.
	Hence we get
	\begin{align*}
		\overline{\nabla}_X \psi &= \nabla^{Q}_X \psi - \frac{1}{2} u^{-1} (\del_X u) \psi \\
			&= \sqrt{u} \nabla^{Q}_X \left(\frac{\psi}{\sqrt{u}}\right). \qedhere
	\end{align*}	
\end{proof}

It follows from \cref{Rem:CharLlSpinors} that apart from the zero section, the subbundle described in \cref{Lem:RestrictSpinorBundlesQ} precisely consists of those lightlike spinors $\psi$ whose Riemannian Dirac current is proportional to $U$ or, equivalently, $-\nu$.
Likewise, the subbundle appearing in the following corollary consists of those spinors~$\Psi$ whose Lorentzian Dirac current is proportional to the lightlike vector field~$V$.

\begin{corollary} \label{Cor:RestrictSpinorBundlesAll}
	Let $(\overline{M},\overline{g})$ be a simple pp-wave as in~\eqref{eq:AGK_metrics_global} with $Q$ spin and set $V \coloneqq \frac{\del}{\del v}$.
	Let $M$ be the spacelike hypersurface $\{0\} \times I \times Q$.
	For some $s \in I$, identify $Q$ with $\{0\} \times \{s\} \times Q$.
	Let $(\overline{M},\overline{g})$, $M$ and $Q$ be equipped with compatible spin structures and the associated spinor bundles.
	Then the subbundle $\{\Psi \in \Sigma \overline{M} \mid V \cdot \Psi = 0\}$ is $\nabla^{\overline{g}}$-parallel.
	Taking together the bundle maps from \cref{Lem:RestrictSpinorBundles,Lem:RestrictSpinorBundlesQ} and the multiplication by $\frac{1}{\sqrt{u}}$, we obtain a bundle isomorphism
	\begin{align} \label{eq:IdentifySpinors}
		\{\Psi \in \Sigma \overline{M}_{|Q} \mid V \cdot \Psi = 0\} \overset{\cong}{\longrightarrow} \Sigma Q,
	\end{align}
	which respects the $TQ$-Clifford multiplication and maps the restriction of $\nabla^{\overline{g}}$ to $\nabla^{g_s}$.
	Moreover, it bijectively maps those spinors with Lorentzian Dirac current equal to $V$ to unit spinors in $\Sigma Q$.
\end{corollary}
\begin{proof}
	Since both the Clifford multiplication in $\Sigma \overline{M}$ and $V$ are parallel, the given subbundle $\ker(V \cdot)$ is parallel.
	Due to the calculation of the restriction of the Dirac current in \cref{Lem:RestrictSpinors}, bundle isomorphisms~\eqref{eq:RestrictSpinorBundles} and~\eqref{eq:RestrictSpinorBundlesQ} compose.
	The statements about the $TQ$-Clifford multiplication and the connections follow immediately from \cref{Lem:RestrictSpinorBundles,Lem:RestrictSpinorBundlesQ}.
	
	Let $\Psi \in \Sigma_p \overline{M}$, $p \in Q$, with $V \cdot \Psi = 0$ and denote by $\psi$ its image under the isomorphism~\eqref{eq:RestrictSpinorBundles} and by $\sqrt{u} \phi$ the image of $\psi$ under the isomorphism~\eqref{eq:RestrictSpinorBundlesQ}, so that $\phi$ is the image of $\Psi$ under~\eqref{eq:IdentifySpinors}. 
	Note that $V_{|M} = ue_0 - U$ for $U = -u^{2} \frac{\del}{\del s}$, where (as always) we take the time-orientation with respect to which $V$ is future-pointing.
	Since all spinors in the given subbundle have Lorentzian Dirac current proportional to~$V$, we have $V_\Psi = V_{p}$ if and only if $-\overline{g}(V_\Psi, e_0) = u(p)$.
	But by \cref{Lem:RestrictSpinors}, this is the case if and only if $|\psi|^2 = u(p)$.
	Since the isomorphism~\eqref{eq:RestrictSpinorBundlesQ} is isometric, this is the case if and only if $|\sqrt{u(p)}\phi|^2 = u(p)$ or, equivalently, $|\phi|^2 = 1$, which was to show.
\end{proof}

\begin{remark}
	The first and second author plan to develop a slightly more conceptual perspective on \cref{Cor:RestrictSpinorBundlesAll} in \cite{Ammann.Gloeckle:prep} showing that the bundle map is independent from $M$.
\end{remark}

It is now clear, what we have to do, in order to obtain lightlike parallel spinors on $\overline{\Sigma}M$:
Given a parallel unit spinor $\phi$ on $\Sigma Q$ for some leaf $Q \subset M$, we consider the spinor $\psi_{|Q} \in \Gamma(\overline{\Sigma}M_{|Q})$ with $-\nu \cdot \psi_{|Q} = e_0 \cdot \psi_{|Q}$ corresponding to $\sqrt{u}\phi$ via~\eqref{eq:RestrictSpinorBundlesQ}.
The spinor $\psi_{|Q}$ will be $\overline{\nabla}$-parallel along $Q$.
Then we $\overline{\nabla}$-parallely extend $\psi_{|Q}$ along the curves $I \ni s \mapsto (s,x)$ for all $x \in Q$ and denote the result by $\psi$.
Note that $\psi$ satisfies $-\nu \cdot \psi = e_0 \cdot \psi$, because the subbundle of hypersurface spinors satisfying this condition is parallel.
If $\psi_{|Q}$ extends to a $\overline{\nabla}$-parallel spinor, then it must be given by $\psi$.

In general, though, non-vanishing of the curvature terms
\begin{gather*}
	\overline{R}(\nu,X) \psi \coloneqq \overline{\nabla}_\nu \overline{\nabla}_X \psi - \overline{\nabla}_X \overline{\nabla}_\nu \psi - \overline{\nabla}_{[\nu,X]} \psi
\end{gather*}
for $X \perp \nu$ is an obstruction to $\overline{\nabla} \psi = 0$.
In fact, $\overline{R}(\nu,X) \psi \neq 0$ for some $X \perp \nu$ is the only obstruction.
To see this, we note that $X$ can be chosen to commute with $\frac{\del}{\del s}$ due to the product decomposition $M = I \times Q$ and thus
\begin{gather*}
		\overline{R}(\nu,X) \psi = u\overline{R}\left(\frac{\del}{\del s},X \right) \psi = u\overline{\nabla}_{\frac{\del}{\del s}} \overline{\nabla}_X \psi - u\overline{\nabla}_X \overline{\nabla}_{\frac{\del}{\del s}} \psi = \overline{\nabla}_{\nu} \overline{\nabla}_X \psi,
\end{gather*}
for such vector fields $X$ because $\psi$ is parallel in $\nu$-direction by construction.
So if $\overline{R}(\nu,X) \psi = 0$ for all $X \perp \nu$, then for any vector field $X \perp \nu$ with $[\frac{\del}{\del s}, X] = 0$ the equations $\overline{\nabla}_{\nu} \overline{\nabla}_X \psi = 0$ and $(\overline{\nabla}_X \psi)_{|Q} = 0$ hold and show that $\overline{\nabla}_X \psi = 0$ on all of $M$.
This implies that $\psi$ is $\overline{\nabla}$-parallel.

However, we are not able to show $\overline{R}(\nu,X) \psi = 0$ directly, without showing that the spinor is parallel.
What we can establish though is an averaged version of this equality, which is the initial data analog of the fact that the Killing development must have null Ricci curvature in order to admit a lightlike parallel spinor.

\begin{proposition} \label{Prop:DerJEq}
	In \cref{Setup:SimpleIDS}, suppose that $\psi \in \Gamma(\overline{\Sigma} M)$ is a hypersurface spinor with $-\nu \cdot \psi = e_0 \cdot \psi$.
	Then in any point $p=(s,x) \in M = I \times Q$
	\begin{gather} \label{eq:DerJEq}
	\sum_{i = 1}^{n-1} e_i \cdot \overline{R}(\nu, e_i) \psi = \frac{1}{2}\sum_{i = 1}^{n-1} \tr^{g_s}\left((\nabla^\gamma k)(e_i,\argu) - (\nabla^\gamma_{e_i}k) \right) e_i \cdot \psi
	\end{gather}
	holds, where $e_1, \ldots, e_{n-1}$ is any orthonormal basis of $T_p(\{s\} \times Q)$.
	
	In particular, if $\psi$ is a nowhere vanishing $\overline{\nabla}$-parallel hypersurface spinor with $-\nu \cdot \psi = e_0 \cdot \psi$, then 
	\begin{gather*}
		\trace^{g_s} \left((\nabla^\gamma k)(X, \argu) - (\nabla^\gamma_X k) \right) = 0
	\end{gather*}
	for all $s \in I$ and $X \in T(\{s\} \times Q)$.
\end{proposition}
\begin{proof}
	Let $e_1, \ldots, e_{n-1}$ be an orthonormal basis of $T(\{s\} \times Q)$ in some point $p= (s, x) \in M$.
	To ease the computations, we choose some Lorentzian manifold $(\overline{M},\overline{g})$ into which $M$ embeds as spacelike hypersurface such that $(\gamma,k)$ is the induced initial data set on $M$.
	Then $e_0, e_1, \ldots, e_{n-1}, \nu$ is a generalized orthonormal basis for $\overline{g}$ in $p$.
	Since, by \cref{Lem:RestrictSpinorBundles}, $\overline{\nabla}$ coincides with the spinorial connection induced by the Levi-Civita connection of $(\overline{M}, \overline{g})$, we have
	\begin{align*}
		\sum_{i = 1}^{n-1} e_i \cdot \overline{R}(\nu, e_i) \psi
			&= - \frac{1}{4} \sum_{i = 1}^{n-1} e_i \cdot e_0 \cdot R^{\overline{g}}(\nu,e_i)e_0 \cdot \psi + \frac{1}{4} \sum_{i,j = 1}^{n-1} e_i \cdot e_j \cdot R^{\overline{g}}(\nu,e_i)e_j \cdot \psi \\
			&\phantom{=}\; + \frac{1}{4} \sum_{i = 1}^{n-1} e_i \cdot \nu \cdot R^{\overline{g}}(\nu,e_i)\nu \cdot \psi.
	\end{align*}
	Since $R^{\overline{g}}(\nu,e_i,e_0, e_0) = 0$, we have that $e_0 \cdot$ anti-commutes with the curvature term in the first summand and the same reasoning also applies to the third summand.
	Moreover, since $R^{\overline{g}}(\nu,e_i) \nu = -R^{\overline{g}}(\nu,e_i) e_0$ as $|U|_\gamma e_0 - U = u(e_0 + \nu)$ is $\nabla^{\overline{g}}$-parallel along $M$ and since $-\nu \cdot \psi = e_0 \cdot \psi$, the first and the last summand cancel.
	We are left with
	\begin{align*}
	\sum_{i = 1}^{n-1} e_i \cdot \overline{R}(\nu, e_i) \psi &= \frac{1}{4} \sum_{i,j = 1}^{n-1} e_i \cdot e_j \cdot R^{\overline{g}}(\nu, e_i)e_j \cdot \psi \\
	&= - \frac{1}{4} \sum_{i,j = 1}^{n-1} R^{\overline{g}}(\nu,e_i,e_j, e_0) e_i \cdot e_j \cdot e_0 \cdot \psi + \frac{1}{4}\sum_{i,j,k = 1}^{n-1} R^{\overline{g}}(\nu,e_i,e_j, e_k) e_i \cdot e_j \cdot e_k \cdot \psi \\
	&\phantom{=}\;+ \frac{1}{4} \sum_{i,j = 1}^{n-1} R^{\overline{g}}(\nu,e_i,e_j, \nu) e_i \cdot e_j \cdot \nu \cdot \psi.
	\end{align*}
	Again, the first and the last summand cancel, and we apply the Codazzi formula in order to obtain
	\begin{align*}
	\sum_{i = 1}^{n-1} e_i \cdot \overline{R}(\nu, e_i) \psi &= \frac{1}{4}\sum_{i,j,k = 1}^{n-1} R^{\overline{g}}(\nu,e_i,e_j, e_k) e_i \cdot e_j \cdot e_k \cdot \psi \\
	&= \frac{1}{4}\sum_{i,j,k = 1}^{n-1} R^{\overline{g}}(e_k,e_j,e_0, e_i) e_i \cdot e_j \cdot e_k \cdot \psi \\
	&= \frac{1}{4}\sum_{i,j,k = 1}^{n-1} \left((\nabla_{e_k}k)(e_j,e_i) - (\nabla_{e_j}k)(e_k,e_i) \right) e_i \cdot e_j \cdot e_k \cdot \psi.
	\end{align*}
	If $j=k$, then the term in brackets is zero.
	If $i \neq j \neq k \neq i$, then $e_i \cdot e_j \cdot e_k \cdot \psi$ is invariant under cyclic permutation of the indices.
	Since adding the term in brackets with its cyclic permutations yields zero, the sum of all those summands with $i \neq j \neq k \neq i$ contributes nothing.
	Thus we only need to consider the summands where either $i=j$ or $i=k$, leading to
	\begin{align*}
	\sum_{i = 1}^{n-1} e_i \cdot \overline{R}(\nu, e_i) \psi &= \frac{1}{4}\sum_{j \neq k = 1}^{n-1} \left((\nabla_{e_k}k)(e_j,e_j) - (\nabla_{e_j}k)(e_k,e_j) \right) e_j \cdot e_j \cdot e_k \cdot \psi \\
	&\phantom{=}\;+ \frac{1}{4}\sum_{j \neq k = 1}^{n-1} \left((\nabla_{e_k}k)(e_j,e_k) - (\nabla_{e_j}k)(e_k,e_k) \right) e_k \cdot e_j \cdot e_k \cdot \psi \\
	&= \frac{1}{2}\sum_{j, k = 1}^{n-1} \left((\nabla_{e_k}k)(e_j,e_k) - (\nabla_{e_j}k)(e_k,e_k) \right) e_j \cdot \psi \\
	&= \frac{1}{2}\sum_{j = 1}^{n-1} \tr^{g_s}\left((\nabla_{\argu}k)(e_j,\argu) - (\nabla_{e_j}k) \right) e_j \cdot \psi.
	\end{align*}
	If $\overline{\nabla} \psi = 0$, then the curvature term on the left hand side vanishes, and so the vector (or the associated $1$-form) with which $\psi \neq 0$ is multiplied on the right hand side has to vanish as well.
\end{proof}

The vector with which $\psi$ gets multiplied on the right hand side of~\eqref{eq:DerJEq} looks very similar to the momentum density $j$ defined in~\eqref{eq:Constraints}.
In fact, the following is true.

\begin{proposition} \label{Prop:CharJEq}
	Assume \cref{Setup:SimpleIDS}.
	Then 
	\begin{gather*}
	\trace^{g_s}((\nabla^\gamma k)(X, \argu) - (\nabla^\gamma_X k)) = j(X) = -\frac{1}{2} u (\div^{g_s}(\dot{g}_s) - \upd \trace^{g_s}(\dot{g}_s))(X)
	\end{gather*}	
	holds for all $s \in I$ and $X \in T(\{s\} \times Q)$.
	In particular, for each $s \in I$ the following are equivalent:
	\begin{itemize}
		\item $\trace^{g_s}\left((\nabla^\gamma k)(X, \argu) - \nabla^\gamma_X k\right) = 0$ for all $X \in T(\{s\} \times Q)$
		\item $j_{|T(\{s\} \times Q)} = 0$
		\item $\div^{g_s}(\dot{g}_s) - \upd \trace^{g_s}(\dot{g}_s) = 0$.
	\end{itemize}
\end{proposition}
\begin{proof}
	We use the same notation as in the proof of \cref{Prop:DerJEq}.
	For the first equality, we recall that $j=\div^\gamma(k)- \upd \tr^\gamma(k)$, so that the difference between the two sides is $(\nabla^\gamma_\nu k)(X,\nu)-(\nabla^\gamma_X k)(\nu,\nu)$.
	Now, using the Codazzi equation and the that $|U|_\gamma e_0 - U = u(e_0+\nu)$ is $\overline{\nabla}$-parallel, we see that
	\begin{gather*}
	(\nabla^\gamma_\nu k)(X,\nu)-(\nabla^\gamma_X k)(\nu,\nu) = R^{\overline{g}}(\nu,X,e_0,\nu) = -R^{\overline{g}}(\nu,X,\nu,\nu) = 0.
	\end{gather*}
	For the proof of the second equality, we refer to \cite[Lem.~2.5]{gloeckle:2025p}.
\end{proof}

\begin{remark} \label{Rem:NullRicci}
	If $\psi$ is a nowhere vanishing $\overline{\nabla}$-parallel spinor with $-\nu \cdot \psi = e_0 \cdot \psi$, then we can conclude the following.
	Because of \cref{Lem:RestrictSpinorBundlesQ}, each of the metrics $g_s$, $s \in I$, admits a parallel spinor an is thus Ricci-flat.
	\Cref{Prop:DerJEq,Prop:CharJEq} together show that the family $g_s$, $s \in I$, satisfies the j-equation~\eqref{eq.level.ricci}.
	These conclusions also follow from \cref{lem.lor.ricci-flat} as the Killing development of $(\gamma,k)$ with respect to $U_\psi$ carries a lightlike parallel spinor by \cref{Prop:ExtendSpinors} and hence has null Ricci curvature by \cref{Lem:LlSpinor}.
\end{remark}

As an immediate consequence, we obtain the following equality, which will play a central role in the main proof of \cref{thm:par_spinor}.

\begin{corollary} \label{Cor:JEq}
	Assume \cref{Setup:SimpleIDS} with $g_s$, $s \in I$, subject to the j-equation.
	Let $\psi \in \Gamma(\overline{\Sigma} M)$ be a hypersurface spinor with $-\nu \cdot \psi = e_0 \cdot \psi$.
	Then in any point $p=(s,x) \in M = I \times Q$
	\begin{gather*}
		\sum_{i = 1}^{n-1} e_i \cdot \overline{R}(\nu, e_i) \psi = 0
	\end{gather*}
	holds, where $e_1, \ldots, e_{n-1}$ is any orthonormal basis of $T_p(\{s\} \times Q)$.
\end{corollary}

\subsection{Proving parallelism}
We will now prove that the hypersurface spinor whose construction we briefly discussed before \cref{Prop:DerJEq} in the last subsection is indeed $\overline{\nabla}$-parallel under suitable assumptions.
As we remarked in \cref{Rem:NullRicci}, we must have that $g_s$, $s \in I$, in \cref{Setup:SimpleIDS} is a family of Ricci-flat metrics subject to the j-equation.
Also, we will have to assume that $Q$ is closed. 

The major tool in the argument is the Dirac-Witten operator $\overline{D} \colon \Gamma(\overline{\Sigma}M) \to \Gamma(\overline{\Sigma}M)$.
It is defined in analogy to the classical Dirac operator of a spinor bundle by the local formula
\begin{align*}
	\overline{D} \psi \coloneqq \sum_{i=1}^{n} e_i \cdot \overline{\nabla}_{e_i} \psi,
\end{align*}
where $e_1, \ldots, e_n$ is any local orthonormal frame of $TM$.

\begin{proposition}
	In \cref{Setup:SimpleIDS}, let $\psi \in \Gamma(\overline{\Sigma}M)$ satisfy $-\nu \cdot \psi = e_0 \cdot \psi$ and $\overline{\nabla}_{\nu} \psi = 0$.
	Let furthermore $s \in I$ and assume that $\psi_{|\{s\} \times Q}$ has compact support.
	Then
	\begin{gather} \label{eq:LeafSL}
		\int_{\{s\} \times Q} |\overline{D} \psi|^2 u^{-1} \dvol^{g_s} = \int_{\{s\} \times Q} |\overline{\nabla}\psi|^2 u^{-1} \dvol^{g_s} +  \frac{1}{4} \int_{\{s\} \times Q} \scal^{g_s} |\psi|^2 u^{-1} \dvol^{g_s}.
	\end{gather}
\end{proposition}
\begin{proof}
	We identify $Q$ with $\{s\} \times Q$.
	Let $\phi$ denote the image of $(u^{-\frac{1}{2}}\psi)_{|Q}$ under the bundle isomorphism~\eqref{eq:RestrictSpinorBundlesQ}.
	Since $\phi$ has compact support the Schrödinger-Lichnerowicz formula for the Dirac operator $D^{g_s}$ of $\Sigma Q$ integrates to
	\begin{gather} \label{eq:RiemSL}
		\int_{Q} |D^{g_s} \phi|^2 \dvol^{g_s} = \int_{Q} |\nabla^{g_s} \phi|^2 \dvol^{g_s} +  \frac{1}{4} \int_{Q} \scal^{g_s} |\phi|^2 \dvol^{g_s}.
	\end{gather}
	Because~\eqref{eq:RestrictSpinorBundlesQ} is an isometry, we have $|\phi|^2 = u^{-1}|\psi|^2$.
	Moreover, its compatibility formula for the connection shows that $(u^{-\frac{1}{2}} \overline{\nabla}_X\psi)_{|Q} = \nabla^{Q}_X(u^{-\frac{1}{2}} \psi)_{|Q}$ gets sent to $\nabla^{g_s}_X \phi$ for all $X \in TQ$.
	Because $\overline{\nabla}_\nu \psi = 0$ by assumption, this implies $|\nabla^{g_s} \phi|^2 = u^{-1} |\overline{\nabla} \psi|^2$ and $|D^{g_s} \phi|^2 = u^{-1}|\overline{D}\psi|^2$.
	Replacing the $\phi$-terms in~\eqref{eq:RiemSL} by the expressions in terms of $\psi$, we get~\eqref{eq:LeafSL}.
\end{proof}
\begin{remark}
	Alternatively, the formula~\eqref{eq:LeafSL} can be derived from the Schrödinger-Lichnerowicz type formula for the Dirac-Witten operator.
	Choosing some $s_0 \in I$, its integrated version for $\psi$ reads
	\begin{gather} \label{eq:IDS_SL}
		\int_{[s_0,s] \times Q} |\overline{D} \psi|^2 \dvol^\gamma = \int_{[s_0,s] \times Q} |\overline{\nabla}\psi|^2 \dvol^\gamma +  \frac{1}{2} \int_{[s_0,s] \times Q} \langle \psi, e_0 \cdot (\mu e_0 - j^\sharp) \cdot \psi \rangle \dvol^\gamma
	\end{gather}
	for any $s \in I$.
	(If $s < s_0$, we view this as these integrals as the negative of the corresponding integrals over $[s,s_0] \times Q$.)
	It is important to note that this formula contains no boundary terms.
	The boundary term from partially integrating $\langle \overline{\nabla}^* \overline{\nabla} \psi, \psi \rangle$ gives boundary integrals over $\pm \langle \overline{\nabla}_\nu \psi, \psi \rangle$, which vanish as $\overline{\nabla}_\nu \psi = 0$.
	The boundary term from partially integrating $\langle \overline{D}^2 \psi, \psi \rangle$ gives boundary integrals over $\pm \langle \nu \cdot \overline{D} \psi, \psi \rangle$ and because 
	\begin{gather*}
		\langle \nu \cdot \overline{D} \psi, \psi \rangle = -\langle \overline{D} \psi , \nu \cdot \psi \rangle = \langle \overline{D} \psi, e_0 \cdot \psi \rangle = \langle e_0 \cdot \overline{D} \psi, \psi \rangle,
	\end{gather*}
	and $\overline{D} \psi$ is also in the subbundle from \cref{Lem:RestrictSpinorBundlesQ}, \ie satisfies $-\nu \cdot \overline{D}\psi = e_0 \cdot \overline{D}\psi$, these terms vanish as well.
	Now we can re-write the integrals in~\eqref{eq:IDS_SL} using Fubini's theorem
	\begin{gather*}
		\int_{[s_0,s] \times Q} \bullet\; \dvol^\gamma = \int_{s_0}^s \int_{\{s^\prime\} \times Q} \bullet\; u^{-1} \dvol^{g_{s^\prime}} \upd s^\prime
	\end{gather*}
	and then differentiate with respect to $s$ to almost obtain~\eqref{eq:LeafSL}.
	The only thing left is to identify the curvature term.
	First of all, a similar argument as for the $\overline{D}$-boundary term above shows that only the $TQ$-orthogonal part of $j^{\sharp}$ gives a non-zero contribution in the scalar product and so
	\begin{gather*}
		\langle \psi, e_0 \cdot (\mu e_0 - j^\sharp) \cdot \psi \rangle = \langle \psi, e_0 \cdot (\mu e_0 - j(\nu) \nu) \cdot \psi \rangle = (\mu + j(\nu)) |\psi|^2.
	\end{gather*}
	The expression $\mu + j(\nu)$ can most conveniently be computed by embedding $M$ into the Killing development $(\overline{M},\overline{g})$ of $(\gamma,k)$ with respect to $u\nu$.
	Since the canonical lightlike vector field of $(\overline{M},\overline{g})$ restricts to $u(e_0 + \nu)$ along $M$, we have $\ric^{\overline{g}}(\bullette,e_0 + \nu) = 0$ and
	\begin{gather*}
		\mu + j(\nu) = \Ein^{\overline{g}}(e_0, e_0 + \nu) = \ric^{\overline{g}}(e_0,e_0+\nu) - \frac{1}{2} \scal^{\overline{g}} \overline{g}(e_0,e_0+\nu) = \frac{1}{2} \scal^{\overline{g}}.
	\end{gather*}
	Finally, $\scal^{\overline{g}} = \scal^{g_s}$ follows with the help of the calculations in the proof of \cite[Prop.~A.1]{gloeckle:2025p}.
\end{remark}

\begin{theorem} \label{Thm:ExParSpinors}
	Assume \cref{Setup:SimpleIDS} for a closed manifold $Q$ and with a family of Ricci-flat (or scalar-non-negative) Riemannian metrics $g_s$, $s \in I$, satisfying the j-equation.
	Assume furthermore that for some $s_0 \in I$, there is a parallel spinor $\phi \in \Gamma(\Sigma Q)$ for $(Q, g_{s_0})$.
	Let $\psi \in \Gamma(\overline{\Sigma} M)$ be the unique hypersurface spinor which is $\overline{\nabla}$-parallel along all the curves $I \ni s \mapsto (s,x)$, $x \in Q$, and for which $\psi_{|\{s_0\} \times Q}$ gets mapped to $\sqrt{u}\phi$ under~\eqref{eq:RestrictSpinorBundlesQ}.
	Then $\psi$ is $\overline{\nabla}$-parallel.
\end{theorem}
\begin{proof}
	Let $J \subset I$ be any compact interval with $s_0 \in J$.
	It suffices to show that $\overline{\nabla} \psi = 0$ on $J \times Q$, because all the possible $J$ exhaust the interval $I$.
	First, we note the following consequence of the Schrödinger-Lichnerowicz type formula~\eqref{eq:LeafSL}.
	Since $\scal^{g_s} \geq 0$ by assumption, there exists a constant $C > 0$ just depending on the maxima of $u$ and $u^{-1}$ on the compact set $J \times Q$ such that
	\begin{gather*}
		\|\overline{\nabla} \psi\|_{L^2(\{s\} \times Q)} \leq C\|\overline{D} \psi\|_{L^2(\{s\} \times Q)}
	\end{gather*}
	for all $s \in J$.
	
	We now consider $G(s) \coloneqq \|\overline{D} \psi\|_{L^2(\{s\} \times Q)}^2$ on $J$.
	Our aim is to show that $G$ vanishes identically on~$J$, while we already know that $G(s_0) = 0$ as $\psi_{|\{s_0\} \times Q}$ is $\overline{\nabla}$-parallel along $\{s_0\} \times Q$ by \cref{Lem:RestrictSpinorBundlesQ} and $\overline{\nabla}_\nu \psi = 0$ everywhere.
	We calculate its derivative
	\begin{align*}
		G^\prime(s) &= \frac{\upd}{\upd s} \int_{\{s\} \times Q} |\overline{D} \psi|^2 u^{-1} \dvol^{g_s} \\
			&= 2 \int_{\{s\} \times Q} \langle \nabla_{u^{-1}\nu} \overline{D} \psi, \overline{D} \psi \rangle u^{-1} \dvol^{g_s} +  \int_{\{s\} \times Q} |\overline{D} \psi|^2 \frac{\upd}{\upd s} \left(u^{-1} \dvol^{g_s} \right)
	\end{align*}
	and try to estimate the appearing summands.
	Clearly, the absolute value of second integral can be estimated by $C_3 G(s)$ for some constant $C_3 > 0$ just depending on $u_{|J \times Q}$ and $g_s$, $s \in J$.
	For the first integral, we use a local orthonormal frame $e_1, \ldots, e_n$ of $TM$ with $e_n = \nu$ and the connection~\eqref{eq:HypersurfaceConnection} in order to calculate
	\begin{align*}
		\overline{\nabla}_\nu \overline{D} \psi - \sum_{i = 1}^{n} e_i \cdot \overline{R}(\nu,e_i) \psi &= \sum_{i = 1}^{n-1} \overline{\nabla}_\nu (e_i \cdot \overline{\nabla}_{e_i} \psi) - \left(\sum_{i = 1}^{n-1} e_i \cdot \overline{\nabla}_{\nu} \overline{\nabla}_{e_i}  \psi - \sum_{i = 1}^{n-1} e_i \cdot \overline{\nabla}_{[\nu,e_i]} \psi \right)\\
	&= \sum_{i = 1}^{n-1} (\overline{\nabla}_\nu e_i) \cdot \overline{\nabla}_{e_i} \psi + \sum_{j = 1}^{n-1} e_j \cdot \overline{\nabla}_{[\nu,e_j]} \psi \\
	&= -\sum_{i = 1}^{n-1} \overline{g}(\overline{\nabla}_\nu e_i, e_0) e_0 \cdot \overline{\nabla}_{e_i} \psi + \sum_{i, j = 1}^{n-1} \overline{g}(\overline{\nabla}_\nu e_i, e_j) e_j \cdot \overline{\nabla}_{e_i} \psi \\
	&\phantom{=}\;+\sum_{i = 1}^{n-1} \overline{g}(\overline{\nabla}_\nu e_i, \nu) \nu \cdot \overline{\nabla}_{e_i} \psi+ \sum_{i,j = 1}^{n-1} \overline{g}([\nu,e_j], e_i) e_j \cdot \overline{\nabla}_{e_i} \psi
	\end{align*}
	using $\overline{\nabla}_{\nu} \psi = 0$.
	Due to $\overline{g}(\overline{\nabla}_\nu e_i, e_0) = -\overline{g}(e_i, \overline{\nabla}_\nu e_0) = \overline{g}(e_i, \overline{\nabla}_\nu \nu) = -\overline{g}(\overline{\nabla}_\nu e_i, e_0)$, the first and the third summand add up to $\sum_{i = 1}^{n-1} \overline{g}(\overline{\nabla}_\nu e_i, \nu) (e_0 + \nu)\cdot \overline{\nabla}_{e_i} \psi = 0$, as $\overline{\nabla}_{e_i} \psi \in \ker((e_0 + \nu) \cdot)$.
	Moreover, we use $\overline{g}(\overline{\nabla}_\nu e_i, e_j) = -\overline{g}(e_i, \overline{\nabla}_\nu e_j)$ and that the second fundamental form $\overline{g}(-\overline{\nabla}_{\argu} \nu, \argu) \in T^*(\{s\} \times Q) \otimes T^*(\{s\} \times Q)$ is symmetric to conclude
	\begin{align*}
		\overline{\nabla}_\nu \overline{D} \psi - \sum_{i = 1}^{n} e_i \cdot \overline{R}(\nu,e_i) \psi &= \sum_{i,j = 1}^{n-1} \overline{g}(-\overline{\nabla}_{e_j} \nu, e_i) e_j \cdot \overline{\nabla}_{e_i} \psi \\
			&=  \sum_{i = 1}^{n-1} (-\nabla_{e_i} \nu) \cdot \overline{\nabla}_{e_i} \psi.
	\end{align*}
	By virtue of \cref{Cor:JEq}, the second summand on the left hand side is zero.
	Thus we have	
	\begin{gather*}
		\nabla_{\nu} \overline{D} \psi = \overline{\nabla}_{\nu} \overline{D} \psi + \frac{1}{2} e_0 \cdot k(\nu,\argu)^\sharp \cdot \overline{D}\psi = \sum_{i=1}^{n-1} (-\nabla_{e_i} \nu) \cdot \overline{\nabla}_{e_i} \psi + \frac{1}{2} e_0 \cdot k(\nu,\argu)^\sharp \cdot \overline{D}\psi,
	\end{gather*}
	which leads us to the respective estimates
	\begin{align*}
		2 \left| \sum_{i=1}^{n-1} \int_{\{s\} \times Q} \langle (-\nabla_{e_i} \nu) \cdot \overline{\nabla}_{e_i} \psi, \overline{D} \psi \rangle u^{-1} \dvol^{g_s} \right| &\leq C_0 \|\overline{\nabla} \psi\|_{L^2(\{s\} \times Q)} \|\overline{D} \psi\|_{L^2(\{s\} \times Q)} \leq C_1 G(s) \\
	\intertext{and}
		\left|\int_{\{s\} \times Q} \langle e_0 \cdot k(\nu,\argu)^\sharp \cdot \overline{D}\psi, \overline{D} \psi \rangle u^{-1} \dvol^{g_s} \right| &\leq C_2 G(s)
	\end{align*}
	with constants again only depending on $u_{|J \times Q}$ and $g_s$, $s \in J$, which determine $\nabla_{e_i}\nu$ and $k$.
	Overall, we obtain $|G^\prime(s)| \leq (C_1 + C_2 + C_3) G(s)$ and Grönwall's lemma then shows $G \equiv 0$ and so $\overline{\nabla} \psi \equiv 0$ on~$J \times Q$.
\end{proof}
\begin{remark}
	Since for any $s \in I$ restriction of $u^{-\frac{1}{2}}\psi$ to $\{s\} \times Q$ along~\eqref{eq:RestrictSpinorBundlesQ} produces a parallel spinor on $(Q,g_s)$, all the metrics $g_s$ are a posteriori Ricci-flat admitting a non-trivial parallel spinor.
	Thus \cref{Thm:ExParSpinors} reproduces a special case of the theorem by Dai, Wang and Wei \cite[Thm.~4.2~+~subseq.~Rem.]{dai.wang.wei:05} that any sequence of non-negative scalar curvature metrics on a closed manifold converging to a limit metric admitting a non-trivial parallel spinor must eventually consist of Ricci-flat metrics admitting a non-trivial parallel spinor.
	Besides requiring a smooth path of such metrics instead of barely a sequence, the main drawback of~\cref{Thm:ExParSpinors} is that we require the j-equation to hold.
	This can potentially be overcome by pulling back a given smooth family $g_s$, $s \in I$, of non-negative scalar curvature metrics by a suitable smooth family of diffeomorphisms.
\end{remark}

\begin{proof}[Proof of \cref{thm:par_spinor}]
	First of all, notice that if there is a lightlike parallel spinor $\Psi$ with Dirac current $\frac{\del}{\del v}$, then it is uniquely determined by its restriction $\Psi_{|\{0\} \times \{s_0\} \times Q}$ and thus by its image $\phi$ along the isomorphism~\eqref{eq:IdentifySpinors}.
	
	We show conversely that the preimage of a given unit spinor $\phi$ of $(Q, g_{s_0})$ parallely extends.
	By assumption, the induced initial data set on the spacelike hypersurface $M = \{0\} \times I \times Q$ in the given pp-wave spacetime is of the form considered in \cref{Setup:SimpleIDS} with $Q$ closed and $g_s$, $s \in I$, a smooth family of Ricci-flat metrics subject to the j-equation.
	Thus \cref{Thm:ExParSpinors} shows that there is a $\overline{\nabla}$-parallel hypersurface spinor $\psi \in \Gamma(\overline{\Sigma} M)$ such that $\psi_{|\{s_0\} \times Q}$ gets mapped to $\sqrt{u}\phi$ via~\eqref{eq:RestrictSpinorBundlesQ}.
	Under~\eqref{eq:RestrictSpinorBundles}, $\psi$ corresponds to a $\nabla^{\overline{g}}$-parallel spinor $\Psi_{|M}$ with the property that $\Psi_{|\{0\} \times \{s_0\} \times Q}$ gets sent to $\phi$ under~\eqref{eq:IdentifySpinors}.
	Note that the Dirac-current of $\Psi_{|M}$ is given by $V \coloneqq \frac{\del}{\del v}$ by \cref{Cor:RestrictSpinorBundlesAll}.
	This means that $\psi$ is lightlike with Dirac current $-u^2\frac{\del}{\del s}$ (\cref{Lem:RestrictSpinors}).
	By \cref{Cor:SimpleIDSSpinors}, there is a unique parallel spinor $\Psi$ on the spacetime with Dirac current $V$, whose restriction to $M$ gets sent to $\psi$ under~\eqref{eq:RestrictSpinorBundles} and thus coincides with $\Psi_{|M}$ considered before.
\end{proof}

We conclude the section with a corollary on a spinorial curvature condition.
Note that if the considered initial data set $(\gamma, k)$ is contained in a spacetime $(\overline{M},\overline{g})$, then the curvatures $R^{\overline{g}}(X,Y,Z,W)$ for $X,Y,Z,W \in TM$ are determined by $(\gamma, k)$ in terms of the Gauß equation.
We use the suggestive notation $\overline{R}$ for the corresponding expression.
Moreover, $j(X) = \ric^{\overline{g}}(e_0,X) = -\ric^{\overline{g}}(\nu,X)$ for all $X \in TM$ and thus the condition $j_{|\nu^\perp} = 0$ is equivalent to~\eqref{eq:TracedCurv}.
Hence, on a pointwise level, the strength of the spinorial condition in \cref{Cor:SpinCurvVan} ranges between the two curvature conditions considered in \cref{subsec:CurvVan}.

\begin{corollary} \label{Cor:SpinCurvVan}
	Assume \cref{Setup:SimpleIDS} for a closed manifold $Q$ and with a family of Ricci-flat (or scalar-non-negative) Riemannian metrics $g_s$, $s \in I$.
	Assume furthermore that for some $s_0 \in I$, there is a parallel spinor $\phi \in \Gamma(\Sigma Q)$ for $(Q, g_{s_0})$.
	Then the following two curvature conditions are equivalent:
	\begin{itemize}
		\item $j(X) = 0$ for all $X \in TM$ with $X \perp \nu$.
		\item For all $p \in M$, there is a spinor $\psi \in \Sigma_p M$ such that
		\begin{align*}
		\sum_{i,j = 1}^{n-1} \overline{R}(\nu,X,e_i,e_j) e_i \cdot e_j \cdot \psi = 0
		\end{align*}
		for all $X \in TM$, where $e_1, \ldots, e_n$ is an orthonormal basis of $\nu^{\perp} \cap T_pM$.
	\end{itemize}
\end{corollary}
\begin{proof}
	The second condition actually implies the first one pointwise.
	This follows by taking Clifford multiplication with $X$ and afterwards tracing over $X$: 
	\begin{align*}
	\sum_{i,j,k = 1}^{n-1} \overline{R}(\nu,e_k,e_i,e_j) e_k \cdot e_i \cdot e_j \cdot \psi = 0.
	\end{align*}
	The remaining calculations are performed in second half of the proof of \cref{Prop:DerJEq} and in \cref{Prop:CharJEq}.
	
	For the converse direction, we note that by \cref{Prop:CharJEq} the family $g_s$, $s \in I$, satisfies the j-equation.
	Hence we can apply \cref{Thm:ExParSpinors} to obtain a non-trivial $\overline{\nabla}$-parallel hypersurface spinor $\Psi$ on $(\gamma,k)$ with $-\nu  \cdot \Psi = e_0 \cdot \Psi$.
	The spinor $\Psi$ has to lie in the kernel of the curvature endomorphism everywhere and the first part of the calculations in \cref{Prop:DerJEq} shows this implies the stated equation for $\Psi$.
	To obtain a spinor on $(M,\gamma)$, we distinguish cases. 
	If $n = \dim(M)$ is even, then $\overline{\Sigma} M = \Sigma M$ and we can take $\psi = \Psi$.
	In the case where $n$ is odd, we have $\overline{\Sigma} M = \Sigma^+ M \oplus \Sigma^- M$, but the condition $-\nu \cdot \Psi = e_0 \cdot \Psi$ shows that $\Psi$ has a non-trivial component $\psi$ in either of the spinor bundles of $M$.
	In any case, $\psi$ will be subject to the curvature equation as well.
\end{proof}

\section{Rigidity in the spatially compact case with energy conditions}\label{sec.spatially_compact}
We recall that the definitions of various standard energy conditions (\cf \eg \cite{Curiel:2017}) for a Lorentzian manifold $(\overline{M},\overline{g})$, where the Einstein curvature is $\Ein^{\overline{g}} = \ric^{\overline{g}} - \frac{1}{2}\scal^{\overline{g}} \overline{g}$.

\begin{definition} \label{Def:EnergyCond}
	$(\overline{M},\overline{g})$ is said to satisfy
	\begin{itemize}
		\item the \emph{dominant energy condition} if $\Ein^{\overline{g}}(W,\tilde{W}) \geq 0$ for all causal vectors $W$, $\tilde{W}$ in the same half of the light cone,
		\item the \emph{weak energy condition} if $\Ein^{\overline{g}}(W,W) \geq 0$ for all causal vectors $W$,
		\item the \emph{strong energy condition} if $\ric^{\overline{g}}(W,W) \geq 0$ for all causal vectors $W$, and
		\item the \emph{null energy condition} if $\Ein^{\overline{g}}(W,W) = \ric^{\overline{g}}(W,W) \geq 0$ for all lightlike vectors $W$.
	\end{itemize}
\end{definition}

\begin{lemma}
	Let $(\overline{M},\overline{g})$ be a Lorentzian manifold with lightlike parallel vector field $V$ and null Ricci curvature $\ric^{\overline{g}} = \rc V^\flat \otimes V^\flat$.
	Then all the energy conditions from \cref{Def:EnergyCond} are equivalent to $\rc \geq 0$.
\end{lemma}
\begin{proof}
	First note that null Ricci curvature implies that the scalar curvature is zero.
	Thus $\Ein^{\overline{g}} = \ric^{\overline{g}}$ and the dominant energy condition implies all the others.
	Two causal vectors $W$ and $\tilde{W}$ are in the same half of the light cone if and only if $\overline{g}(V,W)$ and $\overline{g}(V,\tilde{W})$ have the same sign.
	In this case, $\Ein^{\overline{g}}(W,\tilde{W})$ has the same sign as $\rc$.
	So $\rc \geq 0$ implies all the mentioned energy conditions.
	
	For the converse direction, let $W$ be a lightlike vector that is linearly independent of $V$.
	Then $\overline{g}(V,W) \neq 0$.
	Thus $\Ein^{\overline{g}}(W,W) = \ric^{\overline{g}}(W,W) = \rc \overline{g}(V,W)^2$ has the same sign as $\rc$.
	Hence $\rc \geq 0$ holds if any of the above energy conditions is satisfied.
\end{proof}

If $\rc \geq 0$, then all the terms on the right hand side of~\eqref{eq:ScaleODE} are non-negative.
This leads to the following statement in the case where the spatial topology of $\overline{M}$ is closed.

\begin{theorem}\label{thm:mapping_torus}
	Let $(\overline{M},\overline{g})$ be a connected Lorentzian manifold with lightlike parallel vector field $V$ and null Ricci curvature $\ric^{\overline{g}} = \rc V^\flat \otimes V^\flat$.
	Suppose that there is a closed spacelike hypersurface $M \subset \overline{M}$ that intersects each integral curve of $V$ precisely once and that $M \cap \mathcal{L}$ is compact for some leaf $\mathcal{L}$ of the foliation defined by $V^\flat$.
	Assume that $\rc \geq 0$.
	Then there is a closed, connected Riemannian manifold $(Q,g)$, an isometry $\phi$ of $(Q,g)$, a function $f \in C^\infty(Q)$ and a positive number $\ell \in \R$ such that $(\overline{M},\overline{g})$ isometrically embeds as open subset into the mapping torus
	\begin{align*}
		(\R \times [0,\ell] \times Q)/(v,0,x) \sim (v+f(x),\ell,\phi(x))
	\end{align*}
	equipped with the metric
	\begin{align*}
		\upd v \otimes \upd s + \upd s \otimes \upd v + \upd s^2 + g
	\end{align*}
	such that the vector field $V$ gets mapped to $\frac{\del}{\del v}$.
	In particular, $(\overline{M}, \overline{g})$ is a vacuum spacetime.
\end{theorem}
\begin{proof}
	First of all recall that the distribution defined by $V^\flat$ is integrable and that we call its connected integral manifolds the leaves of the associated foliation.
	For a leaf $\mathcal{L}$ of this foliation, the restriction $M \cap \mathcal{L}$ is a leaf of the foliation associated to the restriction of $V^\flat$ to $M$.
	Suppose that $M \cap \mathcal{L}$ is compact and set $Q \coloneqq M \cap \mathcal{L}$, which is a closed connected manifold.
	Decomposing $V_{|M} = ue_0 -U$ for a positive function $u \in C^\infty(M)$, a vector field $U \in \Gamma(TM)$ and $e_0$ a unit normal on $M$, we consider the flow $\mathrm{Fl}^Z \colon \R \times M \to M$ by $Z = -\frac{1}{u^2} U$.
	Note that it is globally defined since $M$ is compact and that it diffeomorphically maps leaves to leaves as shown in the proof of \cref{lem:global_form}.
	
	We shall now show that the restriction of $\mathrm{Fl}^Z$ to $\pi \colon \R \times Q \to M$ is a covering map.
	At first it is clear that this map is a local diffeomorphism since $Z$ is transversal to all leaves.
	In particular, its image is open.
	The image is also closed:
	As the flow maps leaves to leaves, a leaf is either completely contained in the image or disjoint from it.
	Hence the boundary of the image is a union of leaves.
	Transversality of $Z$ shows that every boundary component can be pushed into the image by either flowing in the direction of $Z$ or $-Z$, so it is part of the image as well.
	Since $M$ is connected and the image is open and closed, $\pi$ is surjective.
	
	Now notice that flowing along $Z$ provides an $\R$-action on the set of leaves of the foliation induced by $V^\flat$ on $M$.
	Hence the set $\{s \in \R \mid \pi(\{s\} \times Q) = Q\}$ is a subgroup of $\R$.
	Since $Q$ is compact, there is some $\epsilon > 0$ such that $\mathrm{Fl}^{Z}$ is injective on $(-\epsilon,\epsilon) \times Q$ and thus the subgroup is discrete.
	Because $M$ is compact and thus the infinite cover $\mathrm{Fl}^Z((k,k+2) \times Q)$, $k \in \Z$, has a finite subcover, we can also conclude that the subgroup is not the trivial group.
	So there is some $\ell > 0$ such that $\{s \in \R \mid \pi(\{s\} \times Q) = Q\} = \ell \Z$.
	Then $\mathrm{Fl}^Z(\ell,\bullette) \colon Q \to Q$ is a diffeomorphism and gives rise to the following $\Z$-action on $\R \times Q$:
	\begin{align*}
		\Z \times (\R \times Q) &\longrightarrow \R \times Q \\
			(k, (s,x)) &\longmapsto (s+k\ell,(\mathrm{Fl}^Z(\ell,\bullette))^k(x)).
	\end{align*}
	The map $\pi$ factors over the quotient $(\R \times Q)/\Z$ and the induced map $(\R \times Q)/\Z \to M$ is injective by construction.
	Since it is also surjective and a local diffeomorphism, it is a diffeomorphism.
	As the $\Z$-action is properly discontinuous, the projection $\R \times Q \to (\R \times Q)/\Z$ is a covering map and thus $\pi$ is a covering map as well.
	
	Since $M$ has the property that each flow line of $V$ intersects it precisely once, we can extend the inverse of the diffeomorphism $(\R \times Q)/\Z \to M$ to an embedding of codimension zero $\overline{M} \hookrightarrow \R \times (\R \times Q)/\Z$ mapping $M$ to $\{0\} \times (\R \times Q)/\Z$ and $V$ to $\frac{\del}{\del v}$, where $v$ denotes the first coordinate.
	Notice that the metric induced by $\overline{g}$ on the image of the embedding uniquely extends to $\R \times (\R \times Q)/\Z$ such that $\frac{\del}{\del v}$ is parallel.
	Moreover, notice that $\R \times (\R \times Q)/\Z$ may equally be viewed as a mapping torus $(\R \times \R \times Q)/\Z$, where $\Z$ acts on the last two factors as above and trivially on the first one.
	
	We now analyze the induced $\Z$-periodic metric on $\R \times \R \times Q$.
	It is clear that it is given by
	\begin{align} \label{eq:OrigMet}
		\upd v \otimes \upd s + \upd s \otimes \upd v + (\pi^*u)^{-2}\upd s^2 + g_s
	\end{align}
	and has null Ricci curvature $(\pi^*\rc) \upd s^2$, where $s$ denotes the second coordinate and $g_s$, $s \in \R$, is a family of Ricci-flat metrics on $Q$ with $(\mathrm{Fl}^Z(\ell,\bullette))^* g_{s+\ell} = g_s$ for all $s \in \R$.
	According to \cref{prop:ScaleODE}, the $\ell$-periodic positive function $\lambda_s = \sqrt[n-1]{\vol^{g_s}(Q)}$, $s \in \R$, satisfies the ODE~\eqref{eq:ScaleODE}.
	Since $\Sigma_s \geq 0$ by definition and $\Rho_s \geq 0$ due to $\rho \geq 0$, the last part of \cref{lem:ODEComparison} shows that $\Rho_s = \Sigma_s = 0$ for all $s \in \R$ and $\lambda$ is constant.
	But this implies that also $\rho \equiv 0$ and that both the trace- and the TT-part of $\dot{g}_s$ vanish for all $s \in \R$.
	According to \cite[App.~D]{ammann.kroencke.mueller:21}, there is a smooth family of diffeomorphisms $\phi_s$, $s \in \R$, such that $\div^{\tilde{g}_s}(\dot{\tilde{g}}_s) = 0$ for all $s \in \R$, where $\tilde{g}_s = \phi_s^* g_s$.
	Because also the trace- and the TT-part of $\dot{\tilde{g}}_s$ vanish for all $s \in \R$, the family is constant, meaning that there is a Ricci-flat metric $g$ on $Q$ such that $\tilde{g}_s = g$ for all $s \in \R$.
	Both the original metric~\eqref{eq:OrigMet} and
	\begin{align} \label{eq:EasyMet}
		\upd v \otimes \upd s + \upd s \otimes \upd v + \upd s^2 + g
	\end{align}
	are Ricci-flat; the latter due to~\eqref{eq:CalcRc}.
	Now \cref{Prop:InjModSp} applies and shows that there is a diffeomorphism $\Psi \colon \R \times \R \times Q \to \R \times \R \times Q$ as in \cref{Lem:ChangeHypersurface} which pulls back~\eqref{eq:OrigMet} to~\eqref{eq:EasyMet}.
	Pulling back the $\Z$-action along $\Psi$, the obtained isometric $\Z$-action on $\R \times \R \times Q$ with metric~\eqref{eq:EasyMet} has the property that operating by $1 \in \Z$ sends $(v,0,x) \mapsto (v+f(x),\ell,\phi(x))$ for some function $f \in C^{\infty}(Q)$ and a diffeomorphism $\phi$, which necessarily has to be an isometry of $(Q,g)$.
\end{proof}

\begin{remark}
	In \cref{thm:mapping_torus}, we fixed the lightlike vector field $V$ onto which $\frac{\del}{\del v} = \grad^{\overline{g}}(s)$ is supposed to map.
	As a consequence, the function $s$ is determined up to addition of constants and the quantity $\ell$ is fixed.
	If we allow for a rescaling of $V$, we can pull-back with the Lorentz boost in the $\R \times \R$-factor
	\begin{gather*}
		(v,s,x) \longmapsto \left(\ell^{-1} v - \frac{\ell-\ell^{-1}}{2}s, \ell s, x \right),
	\end{gather*}
	which preserves the metric~\eqref{eq:EasyMet}.
	This way, we obtain an isometric embedding of $(\overline{M},\overline{g})$ as in \cref{thm:mapping_torus} with $\ell = 1$, where $\frac{\del}{\del v}$ is only required to map to some positive constant multiple of $V$.
	Notice that this procedure generally changes $f$, while $\phi$ is left unchanged.
\end{remark}

\begin{proof}[Proof of \cref{thm:spatially_compact}]
	By definition, spatially compact pp-waves have a closed spacelike hypersurface $M$ intersecting each integral curve of the canonical lightlike vector field $\frac{\del}{\del v}$ precisely once.
	Moreover, since we assume that they are quotients of a simple pp-wave modeled on a closed manifold, all the intersections of $M$ with the leaves are compact.
	Because the dominant energy condition is assumed to hold, $\rc \geq 0$.
	Hence \cref{thm:mapping_torus} can be applied and directly gives the result.
	Notice in addition that in the setting of \cref{thm:spatially_compact} the vector field $\frac{\del}{\del v}$ is complete and thus the isometric embedding is actually an isometric diffeomorphism.
\end{proof}

\bibliographystyle{abbrv}
\bibliography{lps}

\end{document}